\documentclass[11pt]{article}

\usepackage{amsmath,amssymb,amsthm}
\usepackage{bm}
\usepackage{commath}
\usepackage{mathrsfs}
\usepackage{authblk}
\usepackage{url}
\usepackage{pgfplots}
\usepackage{soul}
\usepackage{stmaryrd}
\usepackage{booktabs}
\usepackage{microtype}
\usepackage{placeins}
\usepackage[shortlabels]{enumitem}
\usepackage{cleveref}
\usepackage{graphicx}

\usepackage[margin=1in]{geometry}

\makeatletter

\@ifundefined{linespacing}{\newdimen\linespacing}{}
\AtBeginDocument{\linespacing=\baselineskip\relax}

\def\section{\@startsection{section}{1}\z@
  {.9\linespacing\@plus\linespacing}%
  {.7\linespacing}%
  {\fontsize{12}{14}\selectfont\scshape\centering}}

\def\subsection{\@startsection{subsection}{2}\z@
  {.7\linespacing\@plus.7\linespacing}%
  {.5\linespacing}%
  {\fontsize{11}{13}\selectfont\bfseries}}

\def\subsubsection{\@startsection{subsubsection}{3}\z@
  {.5\linespacing\@plus.5\linespacing}%
  {.3\linespacing}%
  {\normalfont\itshape}}

\makeatother

\newtheorem{theorem}{Theorem}[section]
\numberwithin{equation}{section}
\newtheorem{proposition}[theorem]{Proposition}
\newtheorem{definition}[theorem]{Definition}

\newtheorem{remark}[theorem]{Remark}
\newtheorem{lemma}[theorem]{Lemma}

\newtheorem{algorithm}[theorem]{Algorithm}
\newtheorem{assumption}[theorem]{Assumption}

\usepackage[textsize=small]{todonotes}
\title{Duality Framework for Flux Constrained Flow in Porous Media: Analysis and Numerics
\thanks{This work is partially supported by the Office of Naval Research (ONR) under Award NO: N00014-24-1-2147. NSF grant DMS-2408877, and the Air Force Office of Scientific Research (AFOSR) under Award NO: FA9550-25-1-0231.}}
	\author[1]{Harbir Antil\thanks{Email: \texttt{hantil@gmu.edu}}}	
	\author[2]{Keegan L.A. Kirk\thanks{Email: \texttt{kkirk6@gmu.edu}}}
        \author[3]{Felipe P\'erez\thanks{Email: \texttt{fperezsi@gmu.edu}}}
	\date{}
	\affil[1,2]{\small{Department of Mathematical Sciences and the Center for Mathematics and Artificial Intelligence (CMAI), George Mason University, Fairfax, VA 22030, USA.}}

\begin{document}

\maketitle

\vspace{-1cm}

\begin{abstract}
We introduce and analyze Darcy flow through a saturated porous medium subject to bilateral constraints on the normal flux across a portion of the boundary. The problem is posed as the maximization of a velocity-based dual concave energy over a convex subset of $H(\mathrm{div};\Omega)$; Fenchel duality identifies a pressure-based predual formulation, yields strong duality, and provides convex optimality conditions with a complementarity structure on the constrained boundary. The primal--dual gap satisfies an a posteriori error identity, free of generic constants, valid for arbitrary admissible approximations. The duality structure is inherited by a Raviart--Thomas/Crouzeix--Raviart discretization, from which we derive a discrete error identity and a priori error decay rates under fractional regularity assumptions on the solution and the flux bounds.  Numerical experiments, including adaptive refinement driven by localized primal--dual gap indicators, support the theory.
\end{abstract}

\medskip

\section{Introduction}\label{sec:introduction}

\subsection{Setting and motivation}

Let $\Omega \subset \mathbb{R}^{d}$, $d \in \cbr{2,3}$, be an open, bounded, polyhedral Lipschitz domain whose boundary splits as $\partial\Omega=\overline{\Gamma_{\!D}}\cup\overline{\Gamma_{\!C}}$ into relatively open, disjoint parts $\Gamma_{\!D}$ (Dirichlet) and $\Gamma_{\!C}$ (constraints). We consider the Darcy system for the velocity $\bm u$ and the pressure $p$, subject to a bilateral constraint on the normal flux across $\Gamma_C$:
\begin{subequations}\label{eq:Darcy}
\begin{align}
    {\bm K}^{-1} \bm u + \nabla p      &= 0 \quad \mbox{in } \Omega, \label{eq:Darcy_a}\\
    \mbox{div }\bm u  &= f \quad \mbox{in } \Omega, \label{eq:Darcy_b}\\
   \alpha \le \bm u \cdot \bm n &\le \beta \quad \mbox{on } \Gamma_C,  \label{eq:Darcy_c}\\
    p &= p_D \quad \mbox{on } \Gamma_D. \label{eq:Darcy_d}
\end{align}
\end{subequations}
Here, $f \in L^2(\Omega)$ is a source, $p_D \in H^{\frac12}(\Gamma_D)$ a prescribed boundary pressure, $\bm K$ a symmetric, uniformly elliptic permeability tensor, and $\alpha,\beta \in L^2(\Gamma_C)$ satisfy $\alpha \le \beta$ a.e.\ on $\Gamma_C$. Equations \eqref{eq:Darcy_a}--\eqref{eq:Darcy_b} model the motion of an incompressible fluid through a saturated porous medium, arising in groundwater transport, reservoir simulation, and filtration; their mixed finite element approximation is classical, see e.g. the monographs \cite{Boffi:book,EG21II}. The distinctive feature of \eqref{eq:Darcy} is the boundary condition \eqref{eq:Darcy_c}: on $\Gamma_C$, the normal flux is neither prescribed nor free, but confined to the window $[\alpha,\beta]$. Where $\alpha=\beta$, condition \eqref{eq:Darcy_c} reduces to a Neumann condition with $L^2$ data; where $\alpha<\beta$, the active flux bound at each boundary point is not known a priori.

\begin{remark}[Compatibility condition]
\label{rmk:compat}
In the pure flux case $\Gamma_D=\emptyset$, integrating $\mathrm{div } \, \bm u=f$ over $\Omega$ and applying the divergence theorem yields

\vspace{-5mm}
\begin{align} \label{eq:compatibility}
    \int_{\Gamma_C} \alpha \dif s \le \int_{\Omega} f \dif x  \le \int_{\Gamma_C} \beta \dif s.
\end{align}
\vspace{-5mm}

\noindent
If $\alpha=\beta:=g$, the constraint reduces to  $u\cdot n=g$ on $\Gamma_C$, we recover the usual solvability condition for the Neumann problem
\vspace{-5mm}
 
\begin{align} \label{eq:neumann_compatibility}
\int_\Omega f\,\dif x = \int_{\partial\Omega}  g\,\dif s.
\end{align}
\vspace{-5mm}

\noindent
Note that, in this case, the pressure is only determined up to an additive constant.
\end{remark}

One-sided versions of \eqref{eq:Darcy_c} arise in the unconfined seepage problem, where a complementarity condition of Signorini type \cite{KinderlehrerStampacchia80,Rodrigues87} is imposed on the potential seepage face \cite{ZhengDaiLiu:2009,AlnashriDroniou:2018}, and in the semipermeable membranes of \cite[Ch.~I]{DuvautLions76}, which admit flow in one direction only (see \cite{HanHuangWangXu19} for a nonmonotone variant). The bilateral condition \eqref{eq:Darcy_c} prescribes both a minimal and a maximal admissible throughflow, as is natural for boundaries of limited conveyance capacity such as drains or partially sealing faults. It differs from the bilateral boundary obstacle problems studied in \cite{ArkhipovaUraltseva89}, which confine the trace rather than the normal flux. To the best of our knowledge, problem \eqref{eq:Darcy} has not been analyzed in the literature.

Our analysis rests on convex duality. We formulate \eqref{eq:Darcy} as the maximization of a dual, velocity-based concave energy over a convex subset of $H(\mathrm{div};\Omega)$, identify the pressure-based formulation as its Fenchel predual, and transfer the duality structure to a discretization by the lowest-order Raviart--Thomas element \cite{RT77} for the velocity and the Crouzeix--Raviart element \cite{CR73} for the pressure. This approach follows the framework of Bartels and Kaltenbach \cite{BartelsKaltenbach24,BartelsKaltenbach26}, which builds on orthogonality relations between the Crouzeix--Raviart and Raviart--Thomas spaces \cite{BartelsWang21} and has been applied to the scalar Signorini problem \cite{BartelsGudiKaltenbach25}, the obstacle problem \cite{BartelsKaltenbach25b}, gradient constraints \cite{AntilBartelsKaltenbachKhandelwal25}, and optimal insulation \cite{AntilKaltenbachKirk26}. Building on these developments, we extend the primal--dual analysis to bilateral flux constraints and derive an exact error identity, free of generic constants, that forms the basis of both the a priori error analysis and the computable a posteriori estimator.

\subsection{Contributions}
The main contributions of this paper are the following:
\begin{enumerate}[noitemsep,topsep=2pt]
    \item \textbf{A new PDE model and its well-posedness.} We formulate the bilateral flux-constrained Darcy problem as the maximization of a dual energy $D$ over a convex subset of $H(\mathrm{div};\Omega)$ and prove existence and uniqueness of a maximizer.
    \item \textbf{Fenchel duality and optimality conditions.} Through convex duality, we identify a primal, pressure-based formulation as the minimization of a convex energy $I$, prove strong duality, and derive convex optimality conditions, including a complementarity relation coupling normal flux and pressure on $\Gamma_C$.
    \item \textbf{An exact a posteriori error identity.} We prove that the primal-dual gap estimator $\eta_{\mathrm{gap}}^2(q,\bm v) = I(q) - D(\bm v)$ coincides, for arbitrary admissible pairs, with the total error measure built from optimal convexity measures of the two energies; the identity involves no generic constants.
    \item \textbf{A priori convergence analysis.} We establish a discrete strong-duality identity for the Raviart--Thomas/Crouzeix--Raviart discretization  and derive from the resulting discrete gap identity  convergence under minimal regularity and explicit error decay rates under fractional regularity assumptions on the solution and the flux bounds.
    \item \textbf{Numerical algorithm.} We characterize the discrete dual solution by a KKT system with facetwise multipliers, solve it by a semismooth Newton method, recover the discrete primal solution by a generalized inverse Marini formula requiring no additional linear solve, and employ the discrete primal-dual gap as a stopping criterion.
\end{enumerate}

\subsection{Outline}

The remainder of the paper is organized as follows. In \Cref{sec:preliminaries}, we introduce the notation and the relevant function spaces and finite element spaces. In \Cref{sec:flux-constrained}, a Fenchel duality theory for the continuous problem is developed (\Cref{thm:cts-dual-exist,lem:strong-duality}), which is used in \Cref{sec:cts-aposteriori} to derive an exact a posteriori error identity (\Cref{thm:aposteriori-id}). In \Cref{sec:discrete-flux-constrained}, a discrete Fenchel duality theory based on the Raviart--Thomas and Crouzeix--Raviart elements is developed (\Cref{thm:disc-dual-exist,thm:disc-strong-duality}), which is used in \Cref{ssec:disc-apriori} to derive a discrete error identity (\Cref{thm:disc-aposteriori-id}) and a priori error decay rates under fractional regularity assumptions (\Cref{thm:disc-apriori}). In \Cref{Numerical-algorithm}, the discrete dual problem is solved by a semismooth Newton method (\Cref{alg:ssn}) and the discrete primal solution is recovered by a generalized inverse Marini formula (\Cref{lem:marini}). In \Cref{sec:Numerical-results}, we carry out numerical experiments that support these findings and present an application to miscible displacement in an SPE10 benchmark reservoir.
\section{Preliminaries}\label{sec:preliminaries}
This section introduces the notation and collects a number of preliminary results. \Cref{sec:prelim-domain} is concerned with the relevant function spaces and the trace machinery needed to formulate the bilateral flux constraint, while \Cref{sec:prelim-mesh} is concerned with the relevant finite element spaces and the projection and quasi-interpolation operators employed throughout the paper.

\subsection{Classical function spaces}
\label{sec:prelim-domain}

Throughout the paper, $\Omega \subset \mathbb{R}^{d}$, $d\in\{2,3\}$, is a
bounded Lipschitz polyhedral domain with outward unit normal
$\bm{n}:\partial\Omega\to\mathbb{S}^{d-1}$ defined
$\mathcal H^{d-1}$-almost everywhere on $\partial\Omega$, where
$\mathcal H^{d-1}$ denotes the $(d-1)$-dimensional Hausdorff measure.
For $k\in\mathbb{N}_{0}$ and $p\in[1,\infty]$, we denote by $W^{k,p}(\Omega)$
the Sobolev space of $p$-integrable functions with $p$-integrable weak
derivatives up to order $k$, with the convention
$W^{0,p}(\Omega):=L^{p}(\Omega)$; in the Hilbertian case we write
$H^{k}(\Omega):=W^{k,2}(\Omega)$. For non-integer $s>0$, the fractional
Sobolev space $H^{s}(\Omega)$ is defined through the standard Slobodeckij
norm \cite[Ch. 2]{EG21I}. The $L^{2}(\Omega)$ inner product
is denoted by $(\cdot,\cdot)_{\Omega}$. For an open subset $\omega$ of
$\partial\Omega$ or of a lower-dimensional skeleton, we write
$\langle\cdot,\cdot\rangle_{\omega}$ for the duality pairing on $\omega$ and,
when both arguments belong to $L^{2}(\omega)$, for the $L^{2}(\omega)$ inner
product with which the pairing then coincides; whether a given
$\langle\cdot,\cdot\rangle_{\omega}$ denotes a duality pairing or an inner
product will be clear from the regularity of its arguments. Vector- and matrix-valued
analogues are written
$(L^{p}(\Omega))^{d}$, $(H^{s}(\Omega))^{d}$, $(L^{p}(\Omega))^{d \times d}$, $(H^{s}(\Omega))^{d \times d}$, etc.; inner products and
norms on these spaces are taken componentwise and use the same symbols
when no confusion can arise. We denote by $H(\textup{div};\Omega)$ the
space of vector fields in $(L^{2}(\Omega))^{d}$ with square-integrable
weak divergence, endowed with the usual graph norm.

 \subsubsection{Scalar trace operators}
The trace operator
$\gamma: H^{1}(\Omega)\to H^{\frac{1}{2}}(\partial\Omega)$ is bounded and
surjective \cite[Ch.~3]{EG21I}. For $X\in\{D,C\}$, we denote by
\begin{equation}
  \label{eq:restriction-half}
  H^{\frac{1}{2}}(\Gamma_{\!X})
  \;:=\;\bigl\{\gamma(w)\big|_{\Gamma_{\!X}}\;:\;w\in H^{1}(\Omega)\bigr\}
\end{equation}
the space of traces on $\Gamma_{\!X}$, equipped with the quotient norm
inherited from the surjection $H^{1}(\Omega)\to H^{\frac{1}{2}}(\Gamma_{\!X})$.
Since $\Omega$ is Lipschitz, every $v \in H^{\frac{1}{2}}(\Gamma_{D})$ extends
to an element of $H^{\frac{1}{2}}(\partial\Omega)$ and hence, by
\cite[Thm.~3.10]{EG21I}, admits a bounded lifting
$\widehat{v} \in H^{1}(\Omega)$ with $\gamma(\widehat{v})|_{\Gamma_{D}}=v$.
The closed subspace of $H^{\frac{1}{2}}(\Gamma_{\!X})$ consisting of those
elements whose extension by zero outside $\Gamma_{\!X}$ remains in
$H^{\frac{1}{2}}(\partial\Omega)$ is the Lions--Magenes space
\cite[Ch.~1]{LionsMagenes72},
\begin{equation}
  \label{eq:LM-half}
  H^{\frac{1}{2}}_{00}(\Gamma_{\!X})
  \;:=\;\bigl\{\eta\in H^{\frac{1}{2}}(\Gamma_{\!X})\;:\;
     E_{0}^{X}\eta\in H^{\frac{1}{2}}(\partial\Omega)\bigr\},
\end{equation}
where $E_{0}^{X}\colon H^{\frac{1}{2}}(\Gamma_{\!X})\to L^{2}(\partial\Omega)$ 
denotes extension by zero outside $\Gamma_{\!X}$. We denote the topological duals of $H^{\frac{1}{2}}(\Gamma_X)$
and $H_{00}^{\frac{1}{2}}(\Gamma_X)$ by $H^{-\frac{1}{2}}(\Gamma_X)$ and
$H_{00}^{-\frac{1}{2}}(\Gamma_X)$, respectively; both
$H^{\frac{1}{2}}(\Gamma_X)$ and $H^{\frac{1}{2}}_{00}(\Gamma_X)$ form
Gelfand triples with pivot space $L^{2}(\Gamma_{X})$.

\subsubsection{Normal trace operators}

Every $\bm{v}\in H(\textup{div};\Omega)$ admits a normal trace
$\gamma_{\bm{n}}(\bm{v})\in H^{-\frac{1}{2}}(\partial\Omega)$, characterized
by the Green formula
\begin{equation}
  \label{eq:green-full-boundary}
  \langle\gamma_{\bm{n}}(\bm{v}),\gamma(w)\rangle_{\partial\Omega}
  \;=\;(\bm{v},\nabla w)_{\Omega}+(\textup{div}\,\bm{v},w)_{\Omega}
  \qquad\forall\,w\in H^{1}(\Omega),
\end{equation}
and the map $\gamma_{\bm{n}}:H(\textup{div};\Omega)\to
H^{-\frac{1}{2}}(\partial\Omega)$ is linear, bounded, and surjective;
see \cite[Thm.~4.15]{EG21I}. For $\bm{v}\in H(\textup{div};\Omega)$, we define
the restriction of its normal trace to $\Gamma_{\!X}$, $X\in\{D,C\}$, by
duality against test functions in $H_{00}^{\frac{1}{2}}(\Gamma_X)$:
\begin{equation}
  \label{eq:normal-trace-restriction}
  \bigl\langle\gamma_{\bm{n}}(\bm{v})|_{\Gamma_{\!X}},\eta\bigr\rangle_{\Gamma_{\!X}}
  \;:=\;\bigl\langle\gamma_{\bm{n}}(\bm{v}),E_{0}^{X}\eta\bigr\rangle_{\partial\Omega},
  \qquad\forall\,\eta\in H^{\frac{1}{2}}_{00}(\Gamma_{\!X}),
\end{equation}
which defines $\gamma_{\bm{n}}(\bm{v})|_{\Gamma_{\!X}}$ as an element of
$H^{-\frac{1}{2}}_{00}(\Gamma_{\!X})$. Whenever
$\gamma_{\bm{n}}(\bm{v})|_{\Gamma_{X}}\in L^{2}(\Gamma_{X})$, the
pairing~\eqref{eq:normal-trace-restriction} coincides with the
$L^{2}(\Gamma_{X})$ inner product by density of
$H^{\frac{1}{2}}_{00}(\Gamma_{X})$ in $L^{2}(\Gamma_{X})$.

To make the bilateral flux constraint $\alpha\le\bm{v} \cdot \bm n\le\beta$
meaningful in a pointwise a.e.\ sense on $\Gamma_{C}$,
we work in the subspace of $H(\textup{div};\Omega)$ on which the normal trace
admits an $L^{2}(\Gamma_{C})$-representative,
\begin{equation}
  \label{eq:def-V}
  V\;:=\;\bigl\{\bm{v}\in H(\textup{div};\Omega)\;:\;
     \gamma_{\bm{n}}(\bm{v})|_{\Gamma_{\!C}}\in L^{2}(\Gamma_{\!C})\bigr\},
\end{equation}
where the inclusion $\gamma_{\bm{n}}(\bm{v})|_{\Gamma_{\!C}}\in
L^{2}(\Gamma_{\!C})$ is understood in the sense that there exists
$\mu\in L^{2}(\Gamma_{\!C})$ such that the
duality~\eqref{eq:normal-trace-restriction} reduces to the $L^{2}$
pairing $\langle \mu,\eta\rangle_{\Gamma_{\!C}}$ for every
$\eta\in H^{\frac{1}{2}}_{00}(\Gamma_{\!C})$. We identify
$\gamma_{\bm{n}}(\bm{v})|_{\Gamma_{\!C}}$ with $\mu$ and equip $V$ with the graph
norm
\begin{equation}
  \label{eq:V-norm}
  \|\bm{v}\|_{V}^{2}
  \;:=\;\|\bm{v}\|_{H(\textup{div})}^{2}
     +\|\gamma_{\bm{n}}(\bm{v})|_{\Gamma_{\!C}}\|_{L^{2}(\Gamma_{\!C})}^{2}.
\end{equation}

\begin{proposition}
\label{prop:V-Hilbert}
The space $(V,\|\cdot\|_{V})$ is a Hilbert space.
\end{proposition}

\begin{proof}
Let $(\bm{v}_{j})_{j\in\mathbb{N}}\subset V$ be a Cauchy sequence in $V$. It is easily shown that $\bm{v}_{j}\to\bm{v}\ \text{in}\ H(\textup{div};\Omega)$ and $\gamma_{\bm{n}}(\bm{v}_{j})|_{\Gamma_{\!C}}\to\mu\ \text{in}\ L^{2}(\Gamma_{\!C})$.
It remains to prove $\bm{v}\in V$ and
$\gamma_{\bm{n}}(\bm{v})|_{\Gamma_{\!C}}=\mu$ in $L^{2}(\Gamma_{\!C})$.
By the continuity of $\gamma_{\bm{n}}:H(\textup{div};\Omega)\to
H^{-\frac{1}{2}}(\partial\Omega)$ we have $\gamma_{\bm{n}}(\bm{v}_{j})\to
\gamma_{\bm{n}}(\bm{v})$ in $H^{-\frac{1}{2}}(\partial\Omega)$.
Therefore,
\begin{equation*}
  \bigl\langle\gamma_{\bm{n}}(\bm{v})|_{\Gamma_{\!C}},\eta\bigr\rangle_{\Gamma_{\!C}} =
\lim_{j\to\infty}\bigl\langle\gamma_{\bm{n}}(\bm{v}_{j}),E_{0}^{C}\eta\bigr\rangle_{\partial\Omega}
  \;=\;\lim_{j\to\infty}\bigl\langle \gamma_{\bm{n}}(\bm{v}_{j})|_{\Gamma_{C}},\eta\bigr\rangle_{\Gamma_{C}}
  \;=\;\langle \mu,\eta\rangle_{\Gamma_{C}}, \quad \forall \eta \in H_{00}^{\frac{1}{2}}(\Gamma_C).
\end{equation*}
Since  $H^{\frac{1}{2}}_{00}(\Gamma_{C}) \hookrightarrow L^{2}(\Gamma_{C})$ densely, the map
$\eta\mapsto\langle\gamma_{\bm{n}}(\bm{v})|_{\Gamma_{\!C}},\eta\rangle_{\Gamma_{\!C}}$
extends uniquely to a bounded linear functional on $L^{2}(\Gamma_{\!C})$ that coincides
with $\langle \mu,\cdot\rangle_{\Gamma_{C}}$. By the Riesz representation
theorem, $\mu$ is the $L^{2}$-representative of
$\gamma_{\bm{n}}(\bm{v})|_{\Gamma_{C}}$. In particular $\bm{v}\in V$ and
$\|\bm{v}_{j}-\bm{v}\|_{V}\to0$.
\end{proof}
For the remainder of the article we write
$\bm v \cdot \bm n:=\gamma_{\bm{n}}(\bm{v})$ and denote by $\bm v\cdot \bm n|_{\Gamma_{X}}$ its restriction to
$\Gamma_{\!X}\subset\partial\Omega$ in the sense of
\eqref{eq:normal-trace-restriction} when no ambiguity arises.
In the sequel, we must pair the Dirichlet datum $p_D \in H^{\frac{1}{2}}(\Gamma_D)$ with the restricted normal trace
$\bm v\cdot\bm n|_{\Gamma_D}\in H^{-\frac{1}{2}}_{00}(\Gamma_D)$ of fields $\bm v \in V$. As these spaces are
not in duality with each other, we give the pairing a meaning via a lifting. 
\begin{definition}[Dirichlet boundary pairing]\label{def:pD-pairing}
For $\bm v \in V$ satisfying $\textup{div} \, \bm v = f$, we define a pairing on $\Gamma_D$ by setting
\begin{align}\label{eq:pD-pairing}
    \langle p_D,\,\bm v\cdot \bm n\rangle_{\Gamma_D}
    &:= (\bm v,\nabla\widehat p_D)_\Omega
       + (f,\widehat p_D)_\Omega
       - \langle  \bm v\cdot \bm n,\widehat p_D\rangle_{\Gamma_C},
\end{align}
where $\widehat p_D \in H^{1}(\Omega)$ is any lifting of $p_D$. The right-hand side of~\eqref{eq:pD-pairing} does not depend on the choice of lifting.
\end{definition}

\begin{remark}[Compatibility with the classical pairing]\label{rem:pD-classical}
If, in addition, $\bm v\cdot\bm n|_{\Gamma_D}\in L^2(\Gamma_D)$, the pairing
$\langle p_D,\bm v\cdot\bm n\rangle_{\Gamma_D}$ coincides with the $L^2(\Gamma_D)$ inner product. Moreover, if $p_D\in H_{00}^{1/2}(\Gamma_D)$, the pairing is instead realized through the restriction~\eqref{eq:normal-trace-restriction} of the normal trace to $\Gamma_D$.
\end{remark}

\subsection{Finite element spaces}
\label{sec:prelim-mesh}

Let $(\mathcal T_{h})_{h>0}$ be a family of conforming simplicial triangulations of $\Omega$ assumed to be shape-regular in the sense of Ciarlet (see e.g., \cite{Ciarlet:FEMbook, EG21I}) indexed by the mesh-size
$h:=\max_{T\in\mathcal T_{h}}h_{T}$, where $h_{T}:=\text{diam}(T)$ for every
$T\in\mathcal T_{h}$. 
For each $T\in\mathcal T_{h}$, we denote by $\bm{n}_{T}:\partial
T\to\mathbb{S}^{d-1}$ its outward unit normal, by $|T|$ its $d$-dimensional
Lebesgue measure, and by $x_{T}$ its barycenter. 
 The set of facets 
$\mathcal T_{h}$ is denoted by $\mathcal S_{h}$ and splits into
interior and boundary facets,
\begin{align*}
  \mathcal S_{h}^{i}
  &:=\bigl\{T\cap T'\;:\;T,T'\in\mathcal T_{h},\
     \mathrm{dim}_{\mathcal H}(T\cap T')=d-1\bigr\},\\
  \mathcal S_{h}^{\partial}
  &:=\bigl\{T\cap\partial\Omega\;:\;T\in\mathcal T_{h},\
     \mathrm{dim}_{\mathcal H}(T\cap\partial\Omega)=d-1\bigr\},
\end{align*}
with $\mathcal S_{h}=\mathcal S_{h}^{i}\dot\cup\mathcal S_{h}^{\partial}$,
where $\mathrm{dim}_{\mathcal H}$ denotes the Hausdorff dimension. We assume $(\mathcal S_{h}^\partial)_{h>0}$ is such that every boundary facet is
contained in exactly one of the closures of $\Gamma_{D}$ or
$\Gamma_{C}$, so that
\begin{equation}
  \label{eq:boundary-facets}
  \mathcal S_{h}^{\partial}
  \;=\;\mathcal S_{h}^{D}\,\dot\cup\,\mathcal S_{h}^{C},
  \qquad
  \mathcal S_{h}^{X}
  :=\{S\in\mathcal S_{h}^{\partial}\;:\;\mathrm{int}(S)\subset\Gamma_{\!X}\},
  \quad X\in\{D,C\}.
\end{equation}
For each facet $S\in\mathcal S_{h}$, we denote by $\bm{n}_{S}:
S\to\mathbb{S}^{d-1}$ its outward unit normal, $h_{S}:=\text{diam}(S)$, and $|S|$ for its
$(d-1)$-dimensional Hausdorff measure.
For every $S\in\mathcal S_{h}$, we denote by $x_{S}$ the barycenter of
$S$.
For $s\ge 0$ and $p\in[1,\infty]$, the broken Sobolev space subordinate to
$\mathcal T_h$ is
\begin{equation*}
  W^{s,p}(\mathcal T_h)
  := \cbr[1]{v\in L^{p}(\Omega)\,:\, v|_{T}\in W^{s,p}(T)
     \text{ for all } T\in\mathcal T_h},
\end{equation*}
where, for non-integer $s$, $W^{s,p}(T)$ denotes the Sobolev--Slobodeckij
space; in the Hilbertian case we write
$H^{s}(\mathcal T_h):=W^{s,2}(\mathcal T_h)$. Broken
Sobolev spaces subordinate to a collection of facets
$\mathcal F\subseteq\mathcal S_h$ (in particular
$\mathcal F=\mathcal S_h^{C}$) are defined analogously: for $t\ge 0$,
\begin{equation*}
  H^{t}(\mathcal F)
  := \cbr[1]{v\in L^{2}(\Sigma_{\mathcal F})\,:\, v|_{S}\in H^{t}(S)
     \text{ for all } S\in\mathcal F},
\end{equation*}
where $\Sigma_{\mathcal F}:=\mathrm{int}\bigl(\bigcup_{S\in\mathcal F}\overline S\bigr)$. Vector- and matrix-valued analogues, e.g.
$(H^{s}(\mathcal T_h))^{d}$ and $(W^{1,\infty}(\mathcal T_h))^{d\times d}$,
are defined componentwise with the same notational conventions. For more details, see \cite[Ch. 1]{Pietro:book}.

\subsubsection{Broken polynomial spaces}
For $k\in \cbr{0,1}$ and $T\in\mathcal T_{h}$ (resp. $S \in \mathcal{S}_h$), let $\mathbb{P}^{k}(T)$ (resp. $\mathbb{P}^{k}(S)$) be the
space of polynomials of total degree at most $k$ on $T$ (resp. $S$). The
corresponding broken polynomial spaces are
\begin{align*}
  \mathcal{L}_h^{k}(\mathcal T_{h})
  \;&:=\;\{v_{h}\in L^{\infty}(\Omega)\;:\;v_{h}|_{T}\in\mathbb{P}^{k}(T)
     \text{ for all }T\in\mathcal T_{h}\} \\
  \mathcal{L}_h^{k}(\mathcal S_{h})
  \;&:=\;\{v_{h}\in L^{\infty}(\Gamma_X)\;:\;v_{h}|_{S}\in\mathbb{P}^{k}(T)
     \text{ for all }S\in\mathcal S_{h}\} .
\end{align*}
We define broken polynomial spaces on collections of facets analogously, with the obvious modifications.
Vector- and matrix-valued analogues 
are defined component-wise.
For $q_{h}\in\mathcal{L}^{n}(\mathcal T_{h})$, $n\in\mathbb{N}_{0}$, and
$S\in\mathcal S_{h}$, the \emph{jump} of $q_{h}$ across $S$ is
\begin{equation*}
  \llbracket q_{h}\rrbracket_{S}
  \;:=\;\begin{cases}
    q_{h}|_{T_{+}}-q_{h}|_{T_{-}}
    & \text{if } S\in\mathcal S_{h}^{i},\ \partial T_{+}\cap\partial T_{-}=S,\\
    q_{h}|_{T}
    & \text{if } S\in\mathcal S_{h}^{\partial},\ S\subset\partial T,
  \end{cases}
\end{equation*}
where the labels $T_{\pm}$ are fixed (but otherwise arbitrary) on each
interior facet. For $\bm{v}_{h}\in(\mathcal{L}^{k}(\mathcal T_{h}))^{d}$,
$k\in\mathbb{N}_{0}$, and $S\in\mathcal S_{h}$, the \emph{normal jump} is
\begin{equation*}
  \llbracket\bm{v}_{h}\cdot\bm{n}\rrbracket_{S}
  \;:=\;\begin{cases}
    \bm{v}_{h}|_{T_{+}}\cdot\bm{n}_{T_{+}}+\bm{v}_{h}|_{T_{-}}\cdot\bm{n}_{T_{-}}
    & \text{if } S\in\mathcal S_{h}^{i},\ \partial T_{+}\cap\partial T_{-}=S,\\
    \bm{v}_{h}|_{T}\cdot\bm{n}_{T}
    & \text{if } S\in\mathcal S_{h}^{\partial},\ S\subset\partial T.
  \end{cases}
\end{equation*}
The broken gradient
$\nabla_{h}:\mathcal{L}^{1}(\mathcal T_{h})\to(\mathcal{L}^{0}(\mathcal T_{h}))^{d}$, is defined by
$(\nabla_{h}v_{h})|_{T}:=\nabla(v_{h}|_{T})$ for every
$T\in\mathcal T_{h}$.

We collect here the three local $L^2$-projections used throughout, together with
their approximation properties.
We denote by $\Pi_h$, $\pi_h$, and $\pi_h^1$ the local
$L^2$-projections onto $\mathcal L_h^0(\mathcal T_h)$ (elementwise),
$\mathcal L_h^0(\mathcal S_h)$ (facetwise), and $\mathcal L_h^1(\mathcal S_h^C)$
(facetwise), respectively; see, e.g., \cite{Pietro:book,EG21I}.
There is a constant $C>0$, depending only on the shape-regularity of $\mathcal{T}_h$, such that
\begin{subequations}\label{eq:proj-approx}
\begin{alignat}{2}
   \|v-\Pi_h v\|_{L^2(T)}
   &\le C\,h_T^{r}\,|v|_{H^{t}(T)},
   &&\qquad v\in H^{r}(T),\ t\in[0,1],\label{eq:proj-approx-T}\\
   \|v-\pi_h v\|_{L^2(S)}
   &\le C\,h_S^{r}\,|v|_{H^{r}(S)},
   &&\qquad v\in H^{r}(S),\ r\in[0,1],\label{eq:proj-approx-S}\\
   \|v-\pi_h^1 v\|_{L^2(S)}
   &\le C\,h_S^{r}\,|v|_{H^{r}(S)},
   &&\qquad v\in H^{r}(S),\ r\in[0,2],\label{eq:proj-approx-S1}\\
   \|v-\Pi_h v\|_{L^\infty(T)}
   &\le C\,h_T^{r}\,|v|_{W^{r,\infty}(T)},
   &&\qquad v\in W^{r,\infty}(T),\ r\in[0,1].\label{eq:proj-Linfty}
\end{alignat}
\end{subequations}
These estimates hold componentwise for vector- and matrix-valued arguments.
Since $\mathbb P^0(S)\subset\mathbb P^1(S)$, 
\begin{equation}\label{eq:sneak_in_P1_projection}
   \pi_h\circ\pi_h^1=\pi_h\qquad\text{on }L^1(\Gamma_C).
\end{equation}
Moreover, denoting by $\nabla_S:=(\mathrm{Id}-\bm n_S\otimes\bm n_S)\nabla$ the tangential
gradient on a given facet $S \in \mathcal{S}_h$, the following approximation properties hold:
\begin{align}
   \|\nabla_S v-\pi_h(\nabla_S v)\|_{L^2(S)}
   &\le c\,h_S^{r}\,|\nabla_S v|_{H^{r}(S)}, & r&\in[0,1],\label{eq:proj-approx-S-grad}\\
   \|\nabla_S(v-\pi_h^1 v)\|_{L^2(S)}
   &\le c\,h_S^{\,r-1}\,|v|_{H^{r}(S)}, & r&\in[1,2].\label{eq:proj-approx-S1-grad}
\end{align}
Estimates~\eqref{eq:proj-approx-T}--\eqref{eq:proj-Linfty},
\eqref{eq:proj-approx-S-grad}, and \eqref{eq:proj-approx-S1-grad}, are classical; see, e.g.,
\cite[Ch.~11 and Rem.~12.19]{EG21I}. 

\begin{lemma}
\label{lem:mean_zero_tangential_gradient}
For every $q\in L^1(\Gamma_C)$ and every $S\in\mathcal S_h^C$, the affine
function $(\pi_h q-\pi_h^1 q)|_S$ has vanishing facet mean and
\begin{equation}\label{eq:mean_zero_identity}
   (\pi_h q-\pi_h^1 q)|_S=-\,\nabla_S(\pi_h^1 q)\cdot(x-x_S).
\end{equation}
\end{lemma}
\begin{proof}
By~\eqref{eq:sneak_in_P1_projection},
$\pi_h(\pi_h q-\pi_h^1 q)=\pi_h q-\pi_h q=0$, so $w:=(\pi_h q-\pi_h^1 q)|_S$ is
affine with zero mean on $S$; since $x_S$ is the barycenter,
$w(x)=\nabla_S w\cdot(x-x_S)$. As $\pi_h q$ is constant on $S$,
$\nabla_S w=-\nabla_S\pi_h^1 q$, which gives~\eqref{eq:mean_zero_identity}.
\end{proof}
\noindent Finally,
we note the following local-average bound (\emph{cf}. \cite[Lem.~8.2.3]{DHHR11}): for any measurable $A\subseteq S$ with
$|A|>0$,
\begin{equation}\label{eq:avg_bound_on_S}
   \|q-\langle q\rangle_A\|_{L^2(S)}
   \le\tfrac{2|S|}{|A|}\,\|q-\pi_h q\|_{L^2(S)},\qquad\forall\,q\in L^2(S),
\end{equation}
where $\langle q\rangle_A:=\tfrac{1}{|A|}\int_A v\,dx$ denotes its average.

\subsubsection{The Crouzeix--Raviart finite element space}
 
The Crouzeix--Raviart space~\cite{CR73} is defined as the space of elementwise affine functions whose facet
averages are single-valued across interior facets: 
\begin{align}
  \label{eq:CR-def}
  \mathcal S^{1,\mathrm{cr}}(\mathcal T_{h})
  &:= \cbr[1]{q_{h}\in\mathcal{L}_h^{1}(\mathcal T_{h})\,:\,
     \pi_{h}\llbracket q_{h}\rrbracket_{S}=0\ \text{for all }
     S\in\mathcal S_{h}^{i}}.
\end{align}
Note that functions in $\mathcal S^{1,\mathrm{cr}}(\mathcal T_{h})$ can equivalently be characterized by continuity at the barycenter of each interior facet $S \in \mathcal{S}_h^i$.
The canonical
basis of the Crouzeix--Raviart space $\mathcal S^{1,\mathrm{cr}}(\mathcal T_{h})$ is furnished by the set of functions $\varphi_S \in S^{1,\mathrm{cr}}(\mathcal T_{h})$ satisfying
$\varphi_{S}(x_{S'})=\delta_{S,S'}$ for $S,S'\in\mathcal S_{h}$. The Crouzeix--Raviart quasi-interpolant is the linear operator
$\Pi_{h}^{\mathrm{cr}}: H^{1}(\Omega)\to
\mathcal S^{1,\mathrm{cr}}(\mathcal T_{h})$ defined by
 \begin{align} \label{eq:CR-qi}
 \Pi_{h}^{\mathrm{cr}} q
  \;:=\;\sum_{S\in\mathcal S_{h}}\langle q\rangle_{S}\,\varphi_{S}.    
 \end{align} 
We also introduce the following subspaces of $\mathcal{S}^{1,cr}(\mathcal{T}_h)$ with vanishing trace:
\begin{align*}
  \mathcal S_0^{1,\mathrm{cr}}(\mathcal T_{h})
  &:= \cbr[1]{q_{h}\in \mathcal{S}^{1,cr}(\mathcal{T}_h)\,:\,
     q_h(x_S) = 0, \text{ for all }
     S\in\mathcal S_{h}^\partial}, \\
     \mathcal S_D^{1,\mathrm{cr}}(\mathcal T_{h})
  &:= \cbr[1]{q_{h}\in \mathcal{S}^{1,cr}(\mathcal{T}_h)\,:\,
     q_h(x_S) = 0, \text{ for all }
     S\in\mathcal S_{h}^D}.
\end{align*}
We record a number of key properties satisfied by the Crouzeix--Raviart interpolant below.
\begin{lemma}[Crouzeix--Raviart interpolant {\rm \cite{EG21I}}]
\label{lem:CR-qi} 
The Crouzeix--Raviart interpolant $\Pi_{h}^{\mathrm{cr}}: H^{1}(\Omega)\to
\mathcal S^{1,\mathrm{cr}}(\mathcal T_{h})$ satisfies the following commutation properties: for every $q \in H^1(\Omega)$,
\begin{subequations}
\label{eq:CR-commutation}
\begin{align}
  \nabla_{h}\Pi_{h}^{\mathrm{cr}}q
  &\;=\;\Pi_{h}\nabla q
  \qquad\text{a.e. in }\Omega,\label{eq:CR-comm-grad}\\
  \pi_{h}\Pi_{h}^{\mathrm{cr}}q
  &\;=\;\pi_{h}q
  \qquad\text{a.e. on }\cup_{S \in \mathcal{S}_h} S.\label{eq:CR-comm-trace}
\end{align}
\end{subequations}
For every $s\in[0,1]$, there exists $c>0$, independent of $h>0$, such
that for every $q\in H^{1+s}(\Omega)$ and every $T\in\mathcal T_{h}$,
\begin{subequations}
\label{eq:CR-approx}
\begin{align}
  \|q-\Pi_{h}^{\mathrm{cr}}q\|_{L^{2}(T)}
  +h_{T}\|\nabla q-\nabla_{h}\Pi_{h}^{\mathrm{cr}}q\|_{L^{2}(T)}
  &\;\le\;c\,h_{T}^{1+s}\,|q|_{H^{1+s}(T)}.
\label{eq:CR-approx-vol}
\end{align}
Moreover, for every $S\in\mathcal S_{h}$ and every
$q\in H^{1+s}(T_{S})$ with $S\subset\partial T_{S}$,
\begin{align}
  \|q-\Pi_{h}^{\mathrm{cr}}q\|_{L^{2}(S)}
  &\;\le\;c\,h_{S}^{1/2+s}\,|q|_{H^{1+s}(T_{S})}.
\label{eq:CR-approx-facet}
\end{align}
\end{subequations}
\end{lemma}
 
\subsubsection{The Raviart--Thomas finite element space}

The (lowest-order) Raviart--Thomas space~\cite{RT77} is defined as the following space of piecewise affine vector fields:
\begin{equation*}
  \mathcal{RT}^{0}(\mathcal T_{h})
  \;:=\;\bigl\{\bm{v}_{h}\in H(\textup{div};\Omega)\;:\;
     \bm{v}_{h}|_{T}\in\mathbb{P}^{0}(T)^{d}\oplus\bm{x}\,\mathbb{P}^{0}(T)\ \text{for all }T\in\mathcal T_{h}\bigr\}.
\end{equation*}
Equivalently, $\bm{v}_{h}\in\mathcal{RT}^{0}(\mathcal T_{h})$ if and only
if $\bm{v}_{h}|_{T}(\bm{x})=a_{T}+b_{T}\bm{x}$ with
$a_{T}\in\mathbb{R}^{d}$ and $b_{T}\in\mathbb{R}$ on each
$T\in\mathcal T_{h}$ and  
$\llbracket\bm{v}_{h}\cdot\bm{n}\rrbracket_{S} = 0$ across every interior facet
$S\in\mathcal S_{h}^{i}$.
We will also require the following subspaces of $\mathcal{R}T^0(\mathcal{T}_h)$ with vanishing flux on the boundary:
\begin{align*}
\mathcal{R}T^0_0(\mathcal{T}_h) &:= \{\bm v_h\in \mathcal{R}T^0(\mathcal{T}_h) \, : \, \bm v_h\cdot \bm n= 0\text{ a.e.\ on }\partial\Omega\}, \\
\mathcal{R}T^0_C(\mathcal{T}_h)& := \{\bm v_h\in \mathcal{R}T^0(\mathcal{T}_h) \, : \, \bm v_h\cdot \bm n|_{S}= 0, \text{ for all }
     S\in\mathcal S_{h}^C\}.
\end{align*}
The canonical
basis of the Raviart--Thomas space $\mathcal{RT}^{0}(\mathcal T_{h})$ is furnished by the set of vector fields $\bm \psi_S \in  \mathcal{R}T^0(\mathcal{T}_h)$, $S \in  \mathcal{S}_h$, satisfying  $\bm \psi_S|_{S'}\cdot \bm n_{S'} = \delta_{S,S'}$ on $S'$ for all ${S' \in  \mathcal{S}_h}$,  where~$\bm n_S$ is the unit normal vector on $S$ pointing from $T_-$ to $T_+$~if~${T_+,T_- \in  \mathcal{T}_h}$~with~${S = \partial T_+\cap \partial T_-}$. For $s > \frac{1}{2}$, the Raviart--Thomas quasi-interpolant is the linear operator
$\Pi_{h}^{\mathrm{rt}}:(H^{s}(\Omega))^{d}\cap H(\textup{div};\Omega)\to
\mathcal{RT}^{0}(\mathcal T_{h})$~defined~by\footnote{Alternatively, one may define $\Pi_h^{rt} : V_{p,q}(\Omega) :=  \cbr[1]{\bm v \in  (L^p(\Omega))^d \mid   \textup{div}\,\bm v \in  L^q(\Omega)} \to  \smash{\mathcal{R}T^{0}(\mathcal{T}_h)}$, where $p>2$ and $q>\frac{2d}{d+2}$. The space $V_{p,q}(\Omega)$ arises naturally when studying flows in heterogeneous porous media (\textit{cf}.~\mbox{\cite[Chapter 40]{EG21II}}).}
	\begin{align}
		\Pi_h^{rt} y := \sum_{S\in \mathcal{S}_h}{\langle y\cdot n_S\rangle_S\,\psi_S}.\label{RT-interpolant}
	\end{align}
We collect a number of key properties satisfied by the Raviart--Thomas interpolant in the following:
\begin{lemma}[Raviart--Thomas quasi-interpolant {\rm \cite{Boffi:book,EG21I}}]
\label{lem:RT-qi} 
The Raviart--Thomas interpolant 
$\Pi_{h}^{\mathrm{rt}}:(H^{s}(\Omega))^{d}\cap H(\textup{div};\Omega)\to
\mathcal{RT}^{0}(\mathcal T_{h})$, $s > \tfrac{1}{2}$, satisfies the following commutation properties: for every
$\bm{v}\in(H^{s}(\Omega))^{d}\cap H(\textup{div};\Omega)$,
\begin{subequations}
\label{eq:RT-commutation}
\begin{align}
  \textup{div}(\Pi_{h}^{\mathrm{rt}}\bm{v})
  &\;=\;\Pi_{h}(\textup{div}\,\bm{v})
  \qquad\text{a.e. in }\Omega,\label{eq:RT-comm-div}\\
  \Pi_{h}^{\mathrm{rt}}\bm{v}\cdot\bm{n}
  &\;=\;\pi_{h}(\bm{v}\cdot\bm{n})
  \qquad\text{a.e. on }\cup_{S \in \mathcal{S}_h} S.\label{eq:RT-comm-trace}
\end{align}
\end{subequations}
Moreover, there exists $C>0$, independent of $h>0$, such that for every $T\in\mathcal T_{h}$, the following approximation property holds:
\begin{equation}
  \label{eq:RT-approx}
  \|\bm{v}-\Pi_{h}^{\mathrm{rt}}\bm{v}\|_{L^{2}(T)}
  \;\le\;C\,h_{T}^{s}\,|\bm{v}|_{H^{s}(T)}.
\end{equation}
\end{lemma}
\noindent The discrete duality framework in
Section~\ref{sec:discrete-flux-constrained} relies on the following
integration-by-parts (IBP) identity relating
$\mathcal{RT}^{0}(\mathcal T_{h})$ and
$\mathcal S^{1,\mathrm{cr}}(\mathcal T_{h})$:
for every $q_{h}\in\mathcal S^{1,\mathrm{cr}}(\mathcal T_{h})$ and $\bm{v}_{h}\in\mathcal{RT}^{0}(\mathcal T_{h})$, it holds that
\begin{equation}
  \label{eq:discrete-IBP}
  (\nabla_{h}q_{h},\Pi_{h}\bm{v}_{h})_{\Omega}
    +(\Pi_{h}q_{h},\textup{div}\,\bm{v}_{h})_{\Omega}
  \;=\;(\pi_{h}q_{h},\bm{v}_{h}\cdot\bm n)_{\partial\Omega}.
\end{equation}
Note that \eqref{eq:discrete-IBP} is a simple consequence of the fact that the facet averages of Crouzeix--Raviart functions and the normal components of Raviart--Thomas functions are single-valued across interior facets $S \in \mathcal{S}_h^i$.
In the sequel, we require the following lifting result, which is a special case of \cite[Lemma A.1]{BartelsGudiKaltenbach25}:
\begin{lemma}[Discrete lifting]\label{lem:disc-lifting}
Let $\bar{\bm u}_h \in (\mathcal{L}^0(\mathcal{T}_h))^d$ and
$f_h \in \mathcal{L}^0(\mathcal{T}_h)$ satisfy the compatibility
condition
\begin{equation}\label{eq:compat-lift}
   (\bar{\bm u}_h,\nabla_h r_h)_\Omega + (f_h,\Pi_h r_h)_\Omega \;=\; 0
   \qquad\forall\,r_h\in\mathcal{S}^{1,\mathrm{cr}}_{0}(\mathcal{T}_h).
\end{equation}
 Then there
exists $\bm u_h^{rt}\in\mathcal{R}T^0(\mathcal{T}_h)$ with
$$\Pi_h\bm u_h^{rt}=\bar{\bm u}_h\quad\text{a.e.\ in }\Omega,
\qquad
\mathrm{div}\,\bm u_h^{rt}=f_h\quad\text{a.e.\ in }\Omega.$$
\end{lemma}

\section{Variational formulation and duality}\label{sec:flux-constrained}

In this section, we formulate a variational problem whose Euler--Lagrange system corresponds, in the distributional sense, to the Darcy system with bilateral flux constraints \eqref{eq:Darcy_a}--\eqref{eq:Darcy_d}, and develop a Fenchel duality theory at the continuous level: we pose a velocity-based (\emph{dual}) formulation and prove its well-posedness (\Cref{thm:cts-dual-exist}), identify a pressure-based (\emph{primal}) formulation as its Fenchel predual, and establish strong duality together with the convex optimality conditions (\Cref{lem:strong-duality}). To ensure a well-posed formulation, we make the following assumptions on the problem data:
\begin{assumption} \label{assumption:data-assumption}
Throughout, we make the following assumptions on the data: 
	\begin{enumerate}[noitemsep,topsep=2pt,leftmargin=!,labelwidth=\widthof{(ii)}]
		\item[(i)] The source term and Dirichlet data satisfy $f \in L^2(\Omega)$ and $p_D \in H^{\frac{1}{2}}(\Gamma_D)$, respectively.
		\item[(ii)]  The tensor $\bm{K}: \Omega \to \mathbb{R}^{d\times d}$ is symmetric, uniformly bounded, and elliptic. Thus, there exist constants $0 < k_0 < k_1$ such that
		\[
		k_0 |\bm \xi|^2 \le \bm \xi^T \bm K(x) \bm \xi \le k_1 |\bm \xi|^2, \quad x \in \Omega, \, \bm \xi \in \mathbb{R}^d.
		\]
		For ease of notation, we suppress the spatial dependence of $\bm K$ below.
		\item[(iii)] The lower and upper flux bounds satisfy $\alpha, \beta \in L^2(\Gamma_C)$ and $\alpha \le \beta $ a.e. on $\Gamma_C$, as well as the compatibility condition in \Cref{rmk:compat}.
	\end{enumerate}
\end{assumption}
\subsection{The dual problem}
For given data $\alpha, \beta, f, \bm K,$ and $p_D$ satisfying  \Cref{assumption:data-assumption}, we define a functional
$D : V \to \mathbb{R}\cup\{+\infty\}$ by
\begin{align} \label{eq:dual_functional}
    D(\bm v)
&:=  -\tfrac{1}{2}\| \bm K^{-\frac{1}{2}} \bm v\|_{2,\Omega}^2
    - I_{K^\star}(\bm v) -\langle p_D,\bm v\cdot \bm n\rangle_{\Gamma_D}.
\end{align} 
Here, we have defined the
indicator functional
\begin{align*}
    I_{K^\star}(\bm {v}) &:= \begin{cases}
        0, \quad \bm v \in K^\star, \\
        +\infty, \quad \text{else},
    \end{cases}
\end{align*}
with the dual admissible set defined as
\begin{equation}
  \label{eq:dual_admissible_set}
  K^\star := \cbr{\bm{v}\in V\,:\,
     \textup{div}\,\bm{v}=f\ \text{in}\ \Omega \text{ and }
     \alpha\le\bm v\cdot \bm n|_{\Gamma_{C}}\le\beta\ \text{a.e. on }\Gamma_{C}}.
\end{equation}
Note that the pairing $\langle \cdot,\cdot \rangle_{\Gamma_D}$ in \eqref{eq:dual_functional} is interpreted
via~\eqref{eq:pD-pairing} for $\bm v\in K^\star$.
We seek $\bm u \in V$ solving the optimization problem
\begin{align} \label{eq:cts_max_prob}
D(\bm u)
= \sup_{\bm v \in K^\star}  D(\bm v),
\end{align}
which we will henceforth refer to as the \emph{dual problem}. The remainder of this subsection is devoted to showing that \eqref{eq:cts_max_prob} admits a unique solution.

\begin{proposition}
\label{prop:K-weakly-closed}
The admissible dual set $K^\star$
is non-empty if either: (i) $\Gamma_D \ne \emptyset$, or (ii) the compatibility condition \eqref{eq:compatibility} holds. Moreover, $K^\star$ is convex and weakly closed in $V$.
\end{proposition}
\begin{proof}
The proof that $K^\star$ is convex and weakly closed follows from standard arguments, and therefore we only prove here that $K^\star$ is non-empty. \medskip

(i) Suppose first that $\Gamma_D \ne \emptyset$. Let $g = \tfrac{1}{2}(\alpha + \beta) \in L^2(\Gamma_C)$. To construct $\bm v \in K^\star$, one can simply set $\bm v = -\bm K \nabla q$, where $q \in H_D^1(\Omega)$ satisfies the following boundary value problem:
\begin{subequations} \label{eq:Kstar_bvp}
\begin{align}
    -\text{div} \,(\bm K \nabla q) &= f, \quad \text{in } \Omega, \label{eq:Kstar_bvp_a}\\
    -(\bm K\nabla q )\cdot \bm n &= g, \quad \text{on } \Gamma_C, \label{eq:Kstar_bvp_b}\\
    q &= 0, \quad \text{on } \Gamma_D. \label{eq:Kstar_bvp_c}
\end{align}
\end{subequations}
Note \eqref{eq:Kstar_bvp_a}--\eqref{eq:Kstar_bvp_c} is well-posed owing to assumed ellipticity of $\bm K$ and the Lax--Milgram theorem.
\medskip

(ii) If $\Gamma_D = \emptyset$,  \eqref{eq:Kstar_bvp_a}--\eqref{eq:Kstar_bvp_b} is well-posed for a given $g \in L^2(\Gamma_C)$ if and only if the classic compatibility condition \eqref{eq:neumann_compatibility} for the Neumann problem holds.
Thus, $K^\star$ is non-empty if one can construct  $g^\star \in L^2(\Gamma_C)$ with $\alpha \le g^\star \le \beta$ that further satisfies \eqref{eq:neumann_compatibility}. For $t \in [0,1]$, define $h(t) := (1-t) \int_{\Gamma_C}\alpha(s) \dif s + t \int_{\Gamma_C}\beta(s) \dif s$. The map $t \mapsto h(t) $
is continuous with $h(0) = \int_{\Gamma_C} \alpha \dif s$ and $h(1) = \int_{\Gamma_C} \beta \dif s$. Thus, if \eqref{eq:compatibility} holds, then the Intermediate Value Theorem guarantees $t^\star \in [0,1]$ such that $h(t^\star) = \int_\Omega f \dif x$. Therefore, the choice $g^\star(s) = (1-t^\star) \alpha(s) + t^\star \beta(s)$ for all $s \in \Gamma_C$ yields \eqref{eq:neumann_compatibility}. The result follows.
\qedhere
\end{proof}

\begin{theorem}[Existence and uniqueness for the dual problem]\label{thm:cts-dual-exist}
Suppose that either (i) $\Gamma_D \ne \emptyset$, or (ii) compatibility condition~\eqref{eq:compatibility} holds if
$\Gamma_D=\emptyset$. Then, there exists a unique solution $\bm u \in K^\star$ of
problem~\eqref{eq:cts_max_prob}.
\end{theorem}
\begin{proof}
We apply the Direct Method of the Calculus of Variations to the equivalent
convex minimization problem: find $\bm u \in K^\star$ satisfying
\[
-D(\bm u)
= \inf_{\bm v \in K^\star} \tilde D(\bm v),
\qquad
\tilde D(\bm v) := \tfrac{1}{2}\|\bm K^{-\frac{1}{2}} \bm v\|_{2,\Omega}^2 + I_{K^\star}(\bm v)
                  + \langle p_D,\bm v\cdot \bm n\rangle_{\Gamma_D}.
\]
By \Cref{prop:K-weakly-closed}, $K^\star$ is convex and weakly closed.
The functional $\tilde D$ is strictly convex on $V$. Observe that for any $\bm v \in K^\star$, the ellipticity of $\bm K$, the definition of the pairing $\langle \cdot, \cdot \rangle_{\Gamma_D}$, the Cauchy--Schwarz inequality, trace inequality, and Young's inequality with a sufficiently small $\epsilon >0$, and the fact that $\text{div} \, \bm v = f$ and $\bm v \cdot \bm n \le \beta$ a.e. on $\Gamma_C$ yields
\begin{align*}
    \tilde D(\bm v) &= \tfrac{1}{2}\|\bm K^{-\frac{1}{2}} \bm v\|_{2,\Omega}^2 
                  + \langle p_D,\bm v\cdot \bm n\rangle_{\Gamma_D} \\
&\ge \tfrac{1}{2}\|\bm K^{-\frac{1}{2}} \bm v\|_{2,\Omega}^2 + (\bm v,\nabla\widehat p_D)_\Omega
       + (f,\widehat p_D)_\Omega
       - \langle \bm v\cdot \bm n,\widehat p_D\rangle_{\Gamma_C} \\
    & \gtrsim  \| \bm v \|_{V}^2 -  \|\nabla \widehat p_D\|_{2,\Omega}^2 - \|f\|_{2,\Omega}^2 - \|\beta\|_{2,\Gamma_C}^2.
\end{align*}

Hence, 
$\tilde D$ is coercive on $K^\star$.
It remains to show weak lower semicontinuity of $\tilde D$ on $V$, for which
strong lower semicontinuity suffices since $\tilde D$ is convex. The map
$\bm v\mapsto\tfrac12\|\bm K^{-\frac12} \bm v\|_{2,\Omega}^{2}$ is
continuous on $L^{2}(\Omega)^{d}$, hence on $V$. The map
$\bm v\mapsto \langle p_D,\bm v\cdot\bm n\rangle_{\Gamma_D}$ is, by~\eqref{eq:pD-pairing}, continuous on $V$.
Thus, the Direct Method yields existence and strict convexity yields uniqueness.
\end{proof}
 
\subsection{Fenchel (pre)dual problem}\label{sec:cts-primal}

We next introduce a (pre)dual problem in the sense of Fenchel--Rockafellar corresponding to the optimization problem \eqref{eq:cts_max_prob}. 
For given data $\alpha, \beta, f, \bm K,$ and $p_D$ satisfying  \Cref{assumption:data-assumption}, we define a functional
$I : H^1(\Omega) \to \mathbb{R}\cup\{+\infty\}$ by
\begin{align} \label{eq:predual_functional}
    I(q) := \tfrac{1}{2} \|\bm K^{\frac12} \nabla q\|_{2,\Omega}^2
           - (f,q)_\Omega
           + \langle \beta, q^+\rangle_{\Gamma_C}
           - \langle \alpha, q^-\rangle_{\Gamma_C}
           + I_{K}(q).
\end{align}
where, for a given function $q \in L^2(\Gamma_C)$, we have defined its \emph{positive part} $q^+ := \max(0,q)$ and its \emph{negative part} $q^- := \max(0,-q)$. Here, we have defined the
indicator functional
\begin{align*}
    I_{K}(q) &:= \begin{cases}
        0, \quad q \in K, \\
        +\infty, \quad \text{else},
    \end{cases}
\end{align*}
with the primal admissible set defined as
\begin{align} \label{eq:primal_admissible_set}
K := \{q\in H^1(\Omega) : \gamma(q)|_{\Gamma_D}=p_D\}. 
\end{align}
We seek $p \in H^1(\Omega)$ solving the optimization problem
\begin{align} \label{eq:cts_min_prob}
I(p) = \inf_{ q \in H^1(\Omega)} I(q).
\end{align}
which we will henceforth refer to as the \emph{primal problem}. The following result shows that the dual problem \eqref{eq:cts_max_prob} is indeed the Fenchel dual of the primal problem \eqref{eq:cts_min_prob}.

\begin{theorem}[Fenchel predual and strong duality]\label{lem:strong-duality}
Suppose \Cref{assumption:data-assumption} holds and let  $\bm u \in K^\star$ be the unique maximizer of \eqref{eq:cts_max_prob}. Then, the following statements apply:
\begin{itemize}[noitemsep,topsep=2pt,leftmargin=!,labelwidth=\widthof{(ii)}]
    \item[(i)] The Fenchel (pre)dual problem to the maximization of \eqref{eq:cts_max_prob} is the minimization of \eqref{eq:predual_functional}.
    \item[(ii)] There exists a minimizer $p \in K$ of \eqref{eq:predual_functional} which is furthermore unique if $\Gamma_D \ne \emptyset$. Moreover, there holds a strong duality relation: 
    \begin{align} \label{eq:cts_strong_duality}
        I(p) = D(\bm u).
    \end{align}
    \item[(iii)] The following convex optimality conditions hold:
    \begin{subequations}\label{eq:cts_optimality}
\begin{align}
    \bm u &= -\bm K(x) \nabla p, \label{eq:cts_optimality_1} \\
       \langle\bm u\cdot \bm n, p\rangle_{\Gamma_C}
   & = \langle \beta, p^+\rangle_{\Gamma_C} - \langle \alpha, p^-\rangle_{\Gamma_C}   \label{eq:cts_optimality_2} 
\end{align}
\end{subequations}
\end{itemize}
\end{theorem}
\begin{proof}
(i) Define $G:(L^2(\Omega))^d \to \mathbb{R}$ and $F:H^1(\Omega) \to \mathbb{R}\cup\{+\infty\}$ by
\begin{align*}
    G(\bm v) &:= \tfrac{1}{2}\|\bm K^{\frac{1}{2}} \bm v\|_{2,\Omega}^2, \\
    F(q)    &:= -(f,q)_\Omega + \langle \beta, q^+\rangle_{\Gamma_C}
                - \langle \alpha, q^-\rangle_{\Gamma_C} + I_{K}(q),
\end{align*}
so that $I(q) = G(\nabla q) + F(q)$ for every $q \in H^1(\Omega)$. We aim to show that, in fact,
\begin{align} \label{eq:dual_sumof_conjugate}
   D(\bm u) = -G^\star(-\bm u) - F^\star(\nabla^\star \bm u),
\end{align}
where $F^\star$ and $G^\star$ denote the Fenchel conjugates of $F$ and $G$, respectively.
By \cite[Prop.~13.19 and Prop.~13.23(iv)]{bauschke2017convex}, it holds that
\begin{equation} \label{eq:G_conjugate}
    G^\star(-\bm v) = \tfrac{1}{2}\|\bm K^{-\frac12}  \bm v\|_{2,\Omega}^2.
\end{equation}
Here and throughout, for $w \in L^2(\Omega)$ and $\zeta \in L^2(\Gamma_C)$ we use the
constraint indicators
\begin{align*}
I_{\{f\}}^{\Omega}(w) &=
\begin{cases}
0, & w = f \text{ a.e. in } \Omega, \\
+\infty, & \text{otherwise},
\end{cases}
&
I_{[\alpha,\beta]}^{\Gamma_C}(\zeta) &=
\begin{cases}
0, & \alpha \le \zeta \le \beta \text{ a.e. on } \Gamma_C, \\
+\infty, & \text{otherwise}.
\end{cases}
\end{align*}
 By definition of $K^\star$,
\begin{align} \label{eq:Kstar-decomp}
I_{K^\star}(\bm v) = I_{\{f\}}^{\Omega}(\operatorname{div}\bm v)
+ I_{[\alpha,\beta]}^{\Gamma_C}(\bm v\cdot\bm n).
\end{align}
For every $\bm v \in (L^2(\Omega))^d$, using the definition of the pairing on $\Gamma_D$ \eqref{eq:pD-pairing}, it holds that
\begin{align*}
\begin{aligned}
F^\star(\nabla^\star \bm v)
&=
\sup_{q\in H^1(\Omega)}
\cbr[1]{
(\bm v,\nabla q)_\Omega
+(f,q)_\Omega
-\langle\beta,q^+\rangle_{\Gamma_C}
+\langle\alpha,q^-\rangle_{\Gamma_C}
-I_K(q)
}
\\
&=
\sup_{\widehat q\in H^1_D(\Omega)}
\cbr[1]{
(\bm v,\nabla(\widehat q+\widehat p_D))_\Omega
+(f,\widehat q+\widehat p_D)_\Omega
-\langle\beta,(\widehat q+\widehat p_D)^+\rangle_{\Gamma_C}
+\langle\alpha,(\widehat q+\widehat p_D)^-\rangle_{\Gamma_C}
}
\\
&=
\begin{cases}
\begin{aligned}
&
I_{\cbr{f}}^\Omega(\operatorname{div}\bm v)
+(\bm v,\nabla\widehat p_D)_\Omega
+(f,\widehat p_D)_\Omega
\\
&\quad
+\displaystyle\sup_{\widehat q\in H^1_D(\Omega)}
\cbr[1]{
\langle \bm v\cdot\bm n,\widehat q\rangle_{\Gamma_C}
-\langle\beta,(\widehat q+\widehat p_D)^+\rangle_{\Gamma_C}
+\langle\alpha,(\widehat q+\widehat p_D)^-\rangle_{\Gamma_C}
}
\end{aligned}
&\text{if }\bm v\in V,\\
+\infty
&\text{otherwise},
\end{cases}
\\
&=
\begin{cases}
\begin{aligned}
&
I_{\cbr{f}}^\Omega(\operatorname{div}\bm v)
+(\bm v,\nabla\widehat p_D)_\Omega
+(f,\widehat p_D)_\Omega
\\
&\quad
+\displaystyle\sup_{\bar q\in \gamma(K)|_{\Gamma_C}}
\cbr[1]{
\langle \bm v\cdot\bm n,\bar q-\widehat p_D\rangle_{\Gamma_C}
-\langle\beta,\bar q^{\,+}\rangle_{\Gamma_C}
+\langle\alpha,\bar q^{\, -}\rangle_{\Gamma_C}
}
\end{aligned}
&\text{if }\bm v\in V,\\
+\infty
&\text{otherwise},
\end{cases}
\\
&=
\begin{cases}
\begin{aligned}
&
I_{\cbr{f}}^\Omega(\operatorname{div}\bm v)
+(\bm v,\nabla\widehat p_D)_\Omega
+(f,\widehat p_D)_\Omega
-\langle \bm v\cdot\bm n,\widehat p_D\rangle_{\Gamma_C}
\\
&\quad
+\displaystyle\sup_{\bar q\in \gamma(K)|_{\Gamma_C}}
\cbr[1]{
\langle \bm v\cdot\bm n,\bar q\rangle_{\Gamma_C}
-\langle\beta,\bar q^{\,+}\rangle_{\Gamma_C}
+\langle\alpha,\bar q^{\, -}\rangle_{\Gamma_C}
}
\end{aligned}
&\text{if }\bm v\in V,\\
+\infty
&\text{otherwise},
\end{cases}
 \end{aligned}
\end{align*}
Since $\gamma(K)|_{\Gamma_C} \subset L^2(\Gamma_C)$ is dense, 
\cite[Prop.~2.1, pp.~271]{ET99} gives
\begin{align*}
&\sup_{\bar q\in\gamma(K)|_{\Gamma_C}}
\cbr[0]{
\langle \bm v\cdot\bm n\,\bar q-\beta\bar q^+
+\alpha\bar q^-,1\rangle_{\Gamma_C}
}
\\
&\quad =
\sup_{\bar q\in L^2(\Gamma_C)}
\cbr[0]{
\langle \bm v\cdot\bm n\,\bar q-\beta\bar q^+
+\alpha\bar q^-,1\rangle_{\Gamma_C}
}
\\
&\quad =
\langle
\sup_{\xi\in\mathbb R}
\cbr[0]{(\bm v\cdot\bm n)\xi-\beta\xi^+ +\alpha\xi^-},
1
\rangle_{\Gamma_C}.
\end{align*}
Computing the supremum pointwise for a.e. $s \in \Gamma_C$ yields
\begin{align*}
\sup_{\xi\in\mathbb R}
\cbr[0]{\bm v(s) \cdot\bm n(s)\xi-\beta\xi^+ +\alpha\xi^-}
&=
\max\cbr[1]{
\sup_{\xi\ge0}\del{\bm v(s) \cdot\bm n(s)-\beta}\xi, \,
\sup_{\xi\le0}\del{\bm v(s) \cdot\bm n(s)-\alpha}\xi
}
\\
&=
\begin{cases}
0,
& \alpha\le\bm v\cdot\bm n\le\beta,\\
+\infty,
& \text{otherwise},
\end{cases}
\end{align*}
Using \eqref{eq:Kstar-decomp}, we have
\begin{align} \label{eq:Fstar-final}
\begin{aligned}
    F^\star(\nabla^\star\bm v) 
  &=
    \begin{cases}
        \begin{aligned}
            &I_{\{f\}}^\Omega(\operatorname{div}\bm v)
            +I_{[\alpha,\beta]}^{\Gamma_C}(\bm v\cdot\bm n)      +\langle p_D,\bm v\cdot\bm n\rangle_{\Gamma_D},
        \end{aligned}
        &\text{if }\bm v\in V,                                      \\
        +\infty,
        &\text{otherwise}.
    \end{cases}                                                       \\
    &
    = I_{K^\star}(\bm v)+\langle p_D,\bm v\cdot\bm n\rangle_{\Gamma_D}.
\end{aligned}
\end{align}
Combining~\eqref{eq:G_conjugate} and~\eqref{eq:Fstar-final} yields \eqref{eq:dual_sumof_conjugate}.

\medskip 
(ii) Both $G$
and $F$ are proper, convex, and lower semi-continuous; $G$ is moreover
strictly convex and continuous on $(L^2(\Omega))^d$, and at $\widehat p_D
\in \mathrm{dom}(F)$ the composition $G\circ\nabla$ is continuous. Thus, the Fenchel--Rockafellar theorem~\cite[Rem.~4.2, pp.~60-61]{ET99} applies and yields
\begin{align}\label{eq:FR}
    \inf_{q\in H^1(\Omega)} I(q)
    \;=\; \sup_{\bm v\in L^2(\Omega)^d}\cbr{-G^\star(-\bm v)-F^\star(\nabla^\star \bm v)} = \sup_{\bm v \in V} D(\bm v),
\end{align}
together with the existence of a maximizer $\bm u\in (L^2(\Omega))^d$ of the
right-hand side and at least one minimizer $p\in H^1(\Omega)$ of the
left-hand side, as well as the strong-duality
identity~\eqref{eq:cts_strong_duality}.
The Fenchel--Rockafellar theorem also
gives the existence of at least one primal minimizer $p\in K$ as well as
the strong duality identity~\eqref{eq:cts_strong_duality}.
If $|\Gamma_D|>0$, the Poincar\'e inequality on
$H^1_{D}(\Omega)$ together with the assumed ellipticity of $\bm K$ shows 
$I$ is
strictly convex on $K$ and the minimizer is unique.
 
\medskip 
(iii)
By the standard Fenchel optimality relations, equality in the strong duality
relation implies
\begin{align*}
    -\bm u\in \partial G(\nabla p),
    \qquad
    \nabla^\star\bm u\in \partial F(p).
\end{align*}
The first inclusion gives \eqref{eq:cts_optimality_1}. For the second inclusion, using the
definition of the subdifferential and rearranging gives, for every $q\in K$,
\begin{align*}
&-(f,q-p)_\Omega
+\langle\beta,q^+\rangle_{\Gamma_C}
-\langle\alpha,q^-\rangle_{\Gamma_C}
-\langle\beta,p^+\rangle_{\Gamma_C}
+\langle\alpha,p^-\rangle_{\Gamma_C}
\ge
(\bm u,\nabla(q-p))_\Omega .
\end{align*}
Since $(p,\bm u) \in K \times K^\star$, a short calculation shows that 
\begin{align*}
\langle\beta,q^+\rangle_{\Gamma_C}
-\langle\alpha,q^-\rangle_{\Gamma_C}
-\langle\beta,p^+\rangle_{\Gamma_C}
+\langle\alpha,p^-\rangle_{\Gamma_C}
&\ge
\langle\bm u\cdot\bm n,q-p\rangle_{\Gamma_C},
\end{align*}
which, by the density of  $\gamma(K)|_{\Gamma_C} \subset L^2(\Gamma_C)$, is equivalent to the following  inclusion in $L^2(\Gamma_C)$:
\begin{align*}
\bm u\cdot\bm n
\in
\partial\Bigl(
 q\mapsto
\langle\beta, q^+\rangle_{\Gamma_C}
-\langle\alpha, q^-\rangle_{\Gamma_C}
\Bigr)(p).
\end{align*}
The equality condition in the Fenchel--Young inequality (\emph{cf}. \cite[Prop.~5.1, pp.~21]{ET99} then yields \eqref{eq:cts_optimality_2}.
The proof is now complete.
\end{proof}

\begin{remark}[Complementarity conditions]\label{rem:complementarity}
Decomposing $p = p^+ - p^-$ on $\Gamma_C$, the boundary optimality
condition~\eqref{eq:cts_optimality_2} is equivalent to
\begin{equation}\label{eq:alternative-optimality.3}
    \langle\beta - \bm u\cdot\bm n,\, p^+\rangle_{\Gamma_C}
    + \langle\bm u\cdot\bm n - \alpha,\, p^-\rangle_{\Gamma_C} = 0.
\end{equation}
Since $\bm u\in K^\star$, we have $\alpha \le \bm u\cdot\bm n \le \beta$ a.e.\ on
$\Gamma_C$, so both terms in~\eqref{eq:alternative-optimality.3} are
non-negative. Since they sum to zero, each vanishes, giving the pointwise
complementarity relations
\begin{equation}\label{eq:complementarity_pointwise}
    (\beta - \bm u\cdot\bm n)\,p^+ = 0
    \quad\text{and}\quad
    (\bm u\cdot\bm n - \alpha)\,p^- = 0
    \qquad\text{a.e.\ on }\Gamma_C.
\end{equation}
Equivalently, for a.e.\ $s\in\Gamma_C$,
\begin{align*}
    p(s) > 0 &\;\Longrightarrow\; \bm u\cdot\bm n(s) = \beta(s)
        &&\text{(upper bound active)},\\
    p(s) < 0 &\;\Longrightarrow\; \bm u\cdot\bm n(s) = \alpha(s)
        &&\text{(lower bound active)},\\
    \alpha(s) < \bm u\cdot\bm n(s) < \beta(s) &\;\Longrightarrow\; p(s) = 0
        &&\text{(bounds inactive)}.
\end{align*}
\end{remark}

\begin{remark}[Boundary sparsity of the pressure]\label{rem:sparsity-boundary}
Suppose $A\subseteq\Gamma_C$ is measurable with
$|A|>0$. If
$\alpha(s) < (\bm u\cdot\bm n)(s) < \beta(s)$ for a.e. $s \in A$, then
the complementarity condition in \Cref{rem:complementarity} forces $p = 0$ a.e. on $A$.
Note that we can equivalently write
\begin{align} \label{eq:alternative_L1}
    \langle\beta,p^+\rangle_{\Gamma_C}
      - \langle\alpha,p^-\rangle_{\Gamma_C} = \tfrac{1}{2}\langle \alpha + \beta, p\rangle_{\Gamma_C} + \tfrac{1}{2} \|(\beta - \alpha)p\|_{1,\Gamma_C},
\end{align}
which reveals the connection with the typical sparsity-promoting $L^1$-regularization {\rm(cf. \cite{Stadler2009Elliptic})}. In particular, in the case of symmetric bounds where $\alpha = -\beta$, \eqref{eq:alternative_L1} reduces to
\begin{align} \label{eq:alternative_L1_sym}
    \langle\beta,p^+\rangle_{\Gamma_C}
      - \langle\alpha,p^-\rangle_{\Gamma_C} =  \|\beta p\|_{1,\Gamma_C},
\end{align}
with $\beta$ playing the role of a regularization parameter.
\end{remark}
 
\section{A posteriori error analysis}\label{sec:cts-aposteriori}
 
We now derive an exact primal--dual a posteriori error identity. Throughout, $K$ and
$K^\star$ denote the admissible primal and dual sets introduced in \eqref{eq:primal_admissible_set} and \eqref{eq:dual_admissible_set}, respectively, and
$(p,\bm u)\in K\times K^\star$ is the primal-dual solution pair from
\Cref{lem:strong-duality}. 
Define the \emph{primal-dual gap estimator}
$\eta_{\rm gap}^2 \colon K\times K^\star \to [0,+\infty)$, for every $q \in K$ and $\bm v \in K^\star$, via
\begin{align} \label{eq:gap_estimator}
\eta_{\rm gap}^2(q,\bm v) := I(q) - D(\bm v).
\end{align}
The primal-dual gap estimator serves as a \emph{distance measure} between a given admissible primal-dual pair $(q,\bm v) \in K \times K^\star$ to  $(p,\bm u) \in K \times K^\star$. The following lemma shows that the primal-dual gap estimator decomposes into contribution measuring the violation of the optimality condition \eqref{eq:cts_optimality_1} and a contribution measuring the violation of the optimality condition \eqref{eq:cts_optimality_2}.

\begin{lemma}[Decomposition of the gap estimator]\label{lem:gap-decomp}
For every $q \in K$ and every $\bm v \in K^\star$,
\begin{align*}
    \eta_{\textup{gap}}^2(q,\bm v) = \eta_{\textup{gap,\,I}}^2(q,\bm v) + \eta_{\textup{gap,\,II} }^2(q,\bm v),
\end{align*}
where
\begin{align*}
    \eta_{\textup{gap,\,I}}^2(q,\bm v) & := \tfrac{1}{2}\|\bm K^{\frac{1}{2}} \nabla q
                                   + \bm K^{-\frac{1}{2}} \bm v\|_{2,\Omega}^2,\\
   \eta_{\textup{gap,\,II}}^2(q,\bm v)
    & := -  \langle\bm v\cdot \bm n, q\rangle_{\Gamma_C}
    + \langle \beta, q^+\rangle_{\Gamma_C} - \langle \alpha, q^-\rangle_{\Gamma_C}.
\end{align*}
\end{lemma}
\begin{proof}
Note that for all $q \in K$ and $\bm v \in K^\star$,
\begin{align*}
    \eta_{\rm gap}^2(q,\bm v)
    &= \tfrac{1}{2} \|\bm K^{\frac12} \nabla q\|_{2,\Omega}^2
       - (f,q)_\Omega
       +\langle \beta, q^+\rangle_{\Gamma_C}
                - \langle \alpha, q^-\rangle_{\Gamma_C} 
 +\tfrac{1}{2} \|\bm K^{-\frac{1}{2}}\bm v\|_{2,\Omega}^2
       +\langle p_D,\bm v\cdot \bm n\rangle_{\Gamma_D} \\
      &= \tfrac{1}{2} \|\bm K^{\frac12} \nabla q\|_{2,\Omega}^2
       - (\textup{div}\,\bm v ,q)_\Omega
       +\langle \beta, q^+\rangle_{\Gamma_C}
                - \langle \alpha, q^-\rangle_{\Gamma_C} 
 +\tfrac{1}{2} \|\bm K^{-\frac{1}{2}}\bm v\|_{2,\Omega}^2
       +\langle p_D,\bm v\cdot \bm n\rangle_{\Gamma_D},
\end{align*}
where we have used the fact that $\textup{div}\,\bm v = f$. Integrating by parts and rearranging then yields
\begin{align*}
    \eta_{\rm gap}^2(q,\bm v)
    &= \tfrac{1}{2} \|\bm K^{\frac12} \nabla q\|_{2,\Omega}^2 + (\bm v,\nabla q)_\Omega
       - \langle \bm v\cdot\bm n, q\rangle_{\Gamma_C}
       +\langle \beta, q^+\rangle_{\Gamma_C}
                - \langle \alpha, q^-\rangle_{\Gamma_C} 
 +\tfrac{1}{2} \|\bm K^{-\frac{1}{2}}\bm v\|_{2,\Omega}^2.
\end{align*}
Since $(\bm v,\nabla q)_\Omega
= (\bm K^{\frac12} \nabla q,\,\bm K^{-\frac12}\bm v)_\Omega$, 
the result follows after completing the square.
\end{proof}
 
Next, we identify \emph{optimal strong convexity measures} $\rho_I^2: K \to [0,+\infty)$ and $\rho_{-D}^2: K^\star \to [0,+\infty)$ for the primal energy functional \eqref{eq:predual_functional}  at a primal solution $p \in K$, and for the negative of the dual energy functional \eqref{eq:dual_functional}  at the dual solution $\bm u \in K^\star$. Let
\begin{align*}
    \rho_I^2(q) &:= I(q) - I(p), \\
    \rho_{-D}^2(\bm v) &:= -D(\bm v) + D(\bm u).
\end{align*}
 The following lemma shows that, similar to the primal-dual gap estimator, the strong convexity measures decompose into a contribution measuring the violation of the optimality condition \eqref{eq:cts_optimality_1} and a contribution measuring the violation of the optimality condition \eqref{eq:cts_optimality_2}.
\begin{lemma}[Optimal strong convexity measures]\label{lem:rho-measures}
For every $q \in K$ and every $\bm v \in K^\star$,
\begin{align}
    \rho_I^2(q) &= \tfrac{1}{2}\|\bm K^{\frac12}\nabla (q - p)\|_{2,\Omega}^2
                  - \langle \bm u \cdot \bm n,\, q\rangle_{\Gamma_C}
                  + \langle \beta, q^+\rangle_{\Gamma_C} - \langle \alpha, q^-\rangle_{\Gamma_C}, \label{eq:rho-I} \\
    \rho_{-D}^2(\bm v) &= \tfrac{1}{2} \|\bm K^{-\frac{1}{2}}(\bm v - \bm u)\|_{2,\Omega}^2
                  - \langle \bm v \cdot \bm n,\, p\rangle_{\Gamma_C}
                  +\langle \beta, p^+\rangle_{\Gamma_C} - \langle \alpha, p^-\rangle_{\Gamma_C}. \label{eq:rho-D}
\end{align}
\end{lemma}
\begin{proof}
To show~\eqref{eq:rho-I}, we expand $\rho_I^2(q)=I(q)-I(p)$:
\begin{align*}
    \rho_I^2(q)
    &= \tfrac{1}{2} \|\bm K^{\frac12}\nabla q\|_{2,\Omega}^2
       -\tfrac{1}{2} \|\bm K^{\frac12} \nabla p\|_{2,\Omega}^2
       - (f,q-p)_\Omega +\langle \beta, q^+ - p^+\rangle_{\Gamma_C}
                - \langle \alpha, q^- - p^-\rangle_{\Gamma_C}.
\end{align*}
Substituting $f=\textup{div}\,\bm u$ integrating by parts, and using the fact that that $q-p \in H^1_D(\Omega)$ and $\bm u = -\bm K\nabla p$ by \eqref{eq:cts_optimality_1},
\begin{align*}
    \rho_I^2(q)
    &= \tfrac{1}{2}\|\bm K^{\frac12}\nabla(q-p)\|_{2,\Omega}^2
       - \langle \bm u\cdot\bm n,\, q-p\rangle_{\Gamma_C}
       +\langle \beta, q^+ - p^+\rangle_{\Gamma_C}
                - \langle \alpha, q^- - p^-\rangle_{\Gamma_C} \\
    &= \tfrac{1}{2}\|\bm K^{\frac12}\nabla(q-p)\|_{2,\Omega}^2
       - \langle \bm u\cdot\bm n,\, q\rangle_{\Gamma_C}
       +\langle \beta, q^+\rangle_{\Gamma_C}
                - \langle \alpha, q^- \rangle_{\Gamma_C},
\end{align*}%
where we have used that $\langle\bm u\cdot \bm n, p\rangle_{\Gamma_C}
    = \langle \beta, p^+\rangle_{\Gamma_C} - \langle \alpha, p^-\rangle_{\Gamma_C}$ by~\eqref{eq:cts_optimality_2} to pass to the second line.

To show~\eqref{eq:rho-D}, we expand $\rho_{-D}^2(\bm v)=-D(\bm v)+D(\bm u)$:
\begin{align*}
    \rho_{-D}^2(\bm v)
    &= \tfrac{1}{2}\|\bm K^{-\frac{1}{2}}\bm v\|_{2,\Omega}^2
       - \tfrac{1}{2}\|\bm K^{-\frac{1}{2}}\bm u\|_{2,\Omega}^2
       + \langle p_D, \bm v\cdot\bm n -\bm u\cdot\bm n\rangle_{\Gamma_D} \\
        &=  \tfrac{1}{2}\|\bm K^{-\frac{1}{2}}(\bm v-\bm u)\|_{2,\Omega}^2
       + (\bm K^{-\frac{1}{2}}(\bm v-\bm u),\,\bm K^{-\frac{1}{2}} \bm u)_\Omega
       + \langle p_D,(\bm v-\bm u)\cdot\bm n\rangle_{\Gamma_D}
\end{align*}
Then, using \eqref{eq:cts_optimality_1} and integrating by parts,
\begin{align*}
        \rho_{-D}^2(\bm v) 
    &=  \tfrac{1}{2}\|\bm K^{-\frac{1}{2}}(\bm v-\bm u)\|_{2,\Omega}^2
       - (\bm v-\bm u, \nabla p)_\Omega
       + \langle p_D,(\bm v-\bm u)\cdot\bm n\rangle_{\Gamma_D} \\
    &=  \tfrac{1}{2}\|\bm K^{-\frac{1}{2}}(\bm v-\bm u)\|_{2,\Omega}^2
       - \langle \bm v \cdot \bm n -\bm u \cdot \bm n, p \rangle_{\Gamma_C},
\end{align*}
where we have used the fact that $\textup{div}\,(\bm v - \bm u) = 0$. Finally, another application of \eqref{eq:cts_optimality_2} yields
\begin{align*}
        \rho_{-D}^2(\bm v) 
    &=  \tfrac{1}{2}\|\bm K^{-\frac{1}{2}}(\bm v-\bm u)\|_{2,\Omega}^2
       - \langle \bm v \cdot \bm n, p \rangle_{\Gamma_C} +\langle \beta, p^+\rangle_{\Gamma_C} - \langle \alpha, p^-\rangle_{\Gamma_C}.
\end{align*}
The proof is now complete.
\end{proof}

We end this section by deriving an \emph{a posteriori} error identity that characterizes the \emph{primal-dual total error} $\rho_{\rm tot}^2 : K \times K^\star \to [0,+\infty)$, defined for every $(q,\bm v) \in K \times K^\star$ by
\begin{align} \label{eq:total_error}
    \rho_{\rm tot}^2(q,\bm v) := \rho_{I}^2(q) + \rho_{-D}^2(\bm v),
\end{align}
in terms of the primal-dual gap estimator \eqref{eq:gap_estimator}:
\begin{theorem}[A posteriori error identity]\label{thm:aposteriori-id}
For every $q\in K$ and every $\bm v\in K^\star$,
\begin{equation*}
    \rho_{\rm tot}^2(q,\bm v) = \eta_{\rm gap}^2(q,\bm v).
\end{equation*}
\end{theorem}
\begin{proof}
By the strong duality identity~\eqref{eq:cts_strong_duality}, $I(p)=D(\bm u)$.
Hence
\begin{align*}
    \rho_{\rm tot}^2(q,\bm v)
    &= \del{I(q) - I(p)} + \del{-D(\bm v) + D(\bm u)} = I(q) - D(\bm v)  = \eta_{\rm gap}^2(q,\bm v). \qedhere
\end{align*}
\end{proof}

\section{The discretized flux-constrained flow problem}\label{sec:discrete-flux-constrained}

In this section, we discuss the discretized flux-constrained Darcy problem, employing the Raviart--Thomas element for the dual formulation and the Crouzeix--Raviart element for the primal formulation, and transfer the duality theory of \Cref{sec:flux-constrained} to the discrete level: we prove well-posedness of the discrete dual problem (\Cref{thm:disc-dual-exist}) and establish discrete strong duality together with the discrete convex optimality conditions (\Cref{thm:disc-strong-duality}).
\subsection{Discrete dual problem}\label{ssec:disc-dual}
We begin by formulating the discrete counterpart of the dual problem \eqref{eq:cts_max_prob}. For given data
$\alpha, \beta, f, \bm K,$ and $p_D$ satisfying
\Cref{assumption:data-assumption}, we define the discrete data $\bm K_{h}:=\Pi_{h}\bm K \in (\mathcal{L}_h^0(\mathcal{T}_h))^{d \times d}$, $ f_{h}:=\Pi_{h}f \in \mathcal{L}_h^0(\mathcal{T}_h)$, $\alpha_{h}:=\pi_{h}\alpha \in \mathcal{L}_h^0(\mathcal{S}_h^C)$, $\beta_{h}:=\pi_{h}\beta \in \mathcal{L}_h^0(\mathcal{S}_h^C)$, and $  p_{D}^{h}:=\Pi_{h}^{\mathrm{cr}}\widehat p_{D} \in \mathcal{S}^{1,cr}(\mathcal{T}_h)$,
 where $\widehat p_{D}\in
H^{1}(\Omega)$ is any fixed lifting of the Dirichlet data $p_{D} \in H^{\frac{1}{2}}(\Gamma_D)$. We then define a functional
$D_h^{rt} : \mathcal{R}T^0(\mathcal{T}_h) \to \mathbb{R}\cup\{+\infty\}$ by
\begin{align} \label{eq:discrete_dual_functional}
    D_h^{rt}(\bm v_h)
&:=  -\tfrac{1}{2}\| \bm K_h^{-\frac{1}{2}} \Pi_h \bm v_h\|_{2,\Omega}^2
    - I_{K_h^\star}(\bm v_h) -\langle p_D^h,\bm v_h\cdot \bm n\rangle_{\Gamma_D}.
\end{align}
Here, we have defined the indicator functional
\begin{align*}
    I_{K_h^\star}(\bm {v}_h) &:= \begin{cases}
        0, \quad \bm v_h \in K_h^\star, \\
        +\infty, \quad \text{else},
    \end{cases}
\end{align*}
with the discrete dual admissible set defined as
\begin{equation}
  \label{eq:disc-Kh}
  K_h^\star := \cbr{\bm{v}_h\in\mathcal{R}T^0(\mathcal{T}_h)\,:\,
     \textup{div}\,\bm{v}_h=f_h\ \text{in}\ \Omega \text{ and }
     \alpha_h\le\bm v_h\cdot \bm n\le\beta_h\ \text{on }\mathcal{S}_h^C}.
\end{equation}
We seek $\bm u_h^{rt} \in
\mathcal{R}T^0(\mathcal{T}_h)$ solving the optimization problem
\begin{align} \label{eq:inner_max_unconstrained}
D_h^{rt}(\bm u_h^{rt})
= \sup_{\bm v_h \in K_h^\star} D_h^{rt}(\bm v_h),
\end{align}
which we will henceforth refer to as the \emph{discrete dual problem}. The
remainder of this subsection is devoted to showing that
\eqref{eq:inner_max_unconstrained} admits a unique solution.

\begin{proposition}
\label{prop:disc-K-weakly-closed}
The discrete admissible dual set $K_h^\star$ is non-empty if either: (i)
$\Gamma_D \ne \emptyset$, or (ii) the following discrete compatibility condition holds:
\begin{equation}\label{eq:disc-compatibility}
   \sum_{S\in\mathcal{S}_h^C}|S|\,\alpha_h|_S
   \le \sum_{T\in\mathcal{T}_h}|T|\,f_h|_T
   \le \sum_{S\in\mathcal{S}_h^C}|S|\,\beta_h|_S.
\end{equation}
 Moreover, $K_h^\star$ is convex and closed in
$\mathcal{R}T^0(\mathcal{T}_h)$.
\end{proposition}

\begin{proof}
The proof that $K_h^\star$ is convex and closed follows from standard
arguments, and therefore we only prove here that $K_h^\star$ is non-empty.
\medskip

(i) Suppose first that $\Gamma_D \ne \emptyset$. Fix any $g_h \in
\mathcal{L}_h^0(\mathcal{S}_h^C)$ with $\alpha_h \le g_h \le \beta_h$, and consider the following
discretization of the Darcy problem with homogeneous pressure:
find $(\bm v_h, r_h) \in \mathcal{R}T^0(\mathcal{T}_h) \times
\mathcal{L}_h^0(\mathcal{T}_h)$ such that 
\begin{subequations} \label{eq:disc-Kstar_mixed}
\begin{align}
(\bm K_h^{-1}\Pi_h\bm v_h, \Pi_h\bm w_h)_\Omega
- (r_h, \textup{div}\,\bm w_h)_\Omega
&= 0, 
\quad \forall\, \bm w_h \in \mathcal{R}T^0_C(\mathcal{T}_h), \label{eq:disc-Kstar_mixed_a}\\
(\textup{div}\,\bm v_h, s_h)_\Omega
&= (f_h, s_h)_\Omega,
\quad \forall\, s_h \in \mathcal{L}^0(\mathcal{T}_h), \label{eq:disc-Kstar_mixed_b} \\
\bm v_h \cdot \bm n|_{S}
&= g_h,
\quad \forall \, S \in \mathcal{S}_h^C. \label{eq:disc-Kstar_mixed_c}
\end{align}
\end{subequations}
The ellipticity of $\bm K_h^{-1}$ and the discrete inf--sup condition
for the pair $\mathcal{R}T^0(\mathcal{T}_h) \times
\mathcal{L}_h^0(\mathcal{T}_h)$ (\emph{cf}. \cite{Boffi:book}), guarantees that
\eqref{eq:disc-Kstar_mixed_a}--\eqref{eq:disc-Kstar_mixed_c} admits a unique solution $(\bm v_h, r_h)$. In particular, $\bm v_h \in K_h^\star$.
\medskip

(ii) If $\Gamma_D = \emptyset$, then \eqref{eq:disc-Kstar_mixed_a}--\eqref{eq:disc-Kstar_mixed_c} is well-posed for a given $g_h \in \mathcal{L}_h^0(\mathcal{S}_h^C)$
if and only if the discrete compatibility condition for the Neumann problem holds:
\begin{align} \label{eq:disc-neumann_compatibility}
    \sum_{T\in\mathcal{T}_h}|T|\,f_h|_T = \sum_{S\in\mathcal{S}_h^C}|S|\,g_h|_S,
\end{align}
which follows from testing~\eqref{eq:disc-Kstar_mixed_b} with $s_h \equiv 1$ and applying the divergence theorem. Thus, $K_h^\star$ is non-empty if one can construct a
$g_h^\star \in \mathcal{L}_h^0(\mathcal{S}_h^C)$ with $\alpha_h \le
g_h^\star \le \beta_h$ that further
satisfies~\eqref{eq:disc-neumann_compatibility}. For a.e. $s \in \Gamma_C$ and $t \in [0,1]$, define $h(t) = (1-t) \sum_{S \in \mathcal{S}_h^C}\alpha_h + t \sum_{S \in \mathcal{S}_h^C}\beta_h$. The map $t \mapsto h(t)$
is continuous with $h(0) = \sum_{S \in \mathcal{S}_h^C} \alpha_h$ and $h(1) = \sum_{S \in \mathcal{S}_h^C} \beta_h$. Thus, if \eqref{eq:disc-compatibility} holds, then the Intermediate Value Theorem guarantees $t^\star \in [0,1]$ such that $h(t^\star) = \sum_{T\in\mathcal{T}_h}|T|\,f_h|_T$. Therefore, the choice $g_h^\star|_{S} = (1-t^\star) \alpha_h|_{S} + t^\star \beta_h|_{S}$ for all $S \in \mathcal{S}_h^C$ yields \eqref{eq:disc-neumann_compatibility}. The result follows.
\qedhere
\end{proof}

\begin{theorem}[Existence and uniqueness for the discrete dual problem]\label{thm:disc-dual-exist}
Suppose that either (i) $\Gamma_D \ne \emptyset$, or (ii) the discrete
compatibility condition~\eqref{eq:disc-compatibility} holds if
$\Gamma_D=\emptyset$. Then, there exists a unique solution $\bm u_h^{rt} \in
K_h^\star$ of problem~\eqref{eq:inner_max_unconstrained}.
\end{theorem}

\begin{proof}
The proof is analogous to that of \Cref{thm:cts-dual-exist}, and is therefore omitted.
\end{proof}

\subsection{Discrete Fenchel (pre)dual problem}
\label{ssec:disc-predual}

We next introduce a discrete (pre)dual problem in the sense of
Fenchel--Rockafellar corresponding to the optimization
problem~\eqref{eq:inner_max_unconstrained}. For given discrete data $\alpha_h, \beta_h, f_h, \bm K_h$, and $p_D^h$ defined as in the previous subsection, we define a functional
$I_h^{cr} : \mathcal{S}^{1,\mathrm{cr}}(\mathcal{T}_h) \to \mathbb{R}\cup\{+\infty\}$
by
\begin{align} \label{eq:discrete_predual}
    I_h^{cr}(q_h) := \tfrac{1}{2} \|\bm K_h^{\frac12} \nabla_h q_h\|_{2,\Omega}^2
           - (f_h,\Pi_h q_h)_\Omega
          + \langle \beta_h, (\pi_h q_h)^+\rangle_{\Gamma_C}
           - \langle \alpha_h, (\pi_h q_h)^-\rangle_{\Gamma_C}
           + I_{K_h}(q_h).
\end{align}
 Here, we have defined
the indicator functional
\begin{align*}
    I_{K_h}(q_h) &:= \begin{cases}
        0, \quad q_h \in K_h, \\
        +\infty, \quad \text{else},
    \end{cases}
\end{align*}
with the discrete primal admissible set defined as
\begin{align} \label{eq:disc-primal-admissible}
K_h
   := \{q_h\in\mathcal{S}^{1,\mathrm{cr}}(\mathcal{T}_h) :
     q_h(x_S)=p_D^h(x_S)\ \forall S\in\mathcal{S}_h^D\}.
\end{align}
We seek $p_h^{cr} \in \mathcal{S}^{1,\mathrm{cr}}(\mathcal{T}_h)$ solving the
optimization problem
\begin{align} \label{eq:disc-primal-prob}
I_h^{cr}(p_h^{cr}) = \inf_{ q_h \in \mathcal{S}^{1,\mathrm{cr}}(\mathcal{T}_h)} I_h^{cr}(q_h),
\end{align}
which we will henceforth refer to as the \emph{discrete primal problem}.
The following result shows that the discrete dual
problem~\eqref{eq:inner_max_unconstrained} is indeed the Fenchel dual of
the discrete primal problem~\eqref{eq:disc-primal-prob}.

\begin{theorem}[Discrete Fenchel predual and strong duality]
\label{thm:disc-strong-duality}
Suppose \Cref{assumption:data-assumption} holds, assume $K_h^\star\ne\emptyset$, and let
$\bm u_h^{rt}\in K_h^\star$ be the unique maximizer
of~\eqref{eq:inner_max_unconstrained}. Then the following statements hold.
\begin{itemize}[noitemsep,topsep=2pt,leftmargin=!,labelwidth=\widthof{(ii)}]
   \item[(i)] The Fenchel {\rm(pre)}dual problem to the maximization
   of~\eqref{eq:discrete_dual_functional} is the minimization
   of~\eqref{eq:discrete_predual}. 
   \item[(ii)] There exists a minimizer $p_h^{cr}\in K_h$
   of~\eqref{eq:discrete_predual}, unique if $|\Gamma_D|>0$, and there holds the
   discrete strong-duality relation
   \begin{align}\label{eq:disc_strong_duality}
      I_h^{\mathrm{cr}}(p_h^{cr})=D_h^{\mathrm{rt}}(\bm u_h^{rt}).
   \end{align}
   \item[(iii)] The convex optimality conditions hold:
   \begin{subequations}\label{eq:disc_optimality}
   \begin{align}
      \Pi_h\bm u_h^{rt} &= -\bm K_h\nabla_h p_h^{cr}, \label{eq:disc_optimality_1}\\
      \langle\bm u_h^{rt}\cdot\bm n,\pi_h p_h^{cr}\rangle_{\Gamma_C}
      &= \langle\beta_h,(\pi_h p_h^{cr})^+\rangle_{\Gamma_C}
      -\langle\alpha_h,(\pi_h p_h^{cr})^-\rangle_{\Gamma_C}. \label{eq:disc_optimality_2}
   \end{align}
   \end{subequations}
\end{itemize}
\end{theorem}

\begin{proof}
(i) Define $G_h:(\mathcal L_h^0(\mathcal T_h))^d\to\mathbb R$ and
$F_h:\mathcal S^{1,\mathrm{cr}}(\mathcal T_h)\to\mathbb R\cup\{+\infty\}$ by
\begin{align*}
    G_h(\bm z_h) &:= \tfrac12\|\bm K_h^{\frac12}\bm z_h\|_{2,\Omega}^2, \\
    F_h(q_h) &:= -(f_h,\Pi_h q_h)_\Omega
       +\langle\beta_h,(\pi_h q_h)^+\rangle_{\Gamma_C}
       -\langle\alpha_h,(\pi_h q_h)^-\rangle_{\Gamma_C}
       + I_{K_h}(q_h).
\end{align*}
Then $I_h^{\mathrm{cr}}(q_h)=G_h(\nabla_h q_h)+F_h(q_h)$ for every
$q_h\in\mathcal S^{1,\mathrm{cr}}(\mathcal T_h)$. We show that
\begin{align}\label{eq:disc-dual-sumof-conjugate}
   D_h^{\mathrm{rt}}(\bm v_h)
   =-G_h^*(-\Pi_h\bm v_h)-F_h^*(\nabla_h^*\Pi_h\bm v_h).
\end{align}
By \cite[Prop.~13.19 and Prop.~13.23(iv)]{bauschke2017convex}, it holds that
\begin{align}\label{eq:disc-G-conjugate}
    G_h^*(-\bm z_h)=\tfrac12\|\bm K_h^{-\frac12}\bm z_h\|_{2,\Omega}^2
    \qquad\forall\,\bm z_h\in(\mathcal L_h^0(\mathcal T_h))^d.
\end{align}
For every $\bm z_h\in(\mathcal L_h^0(\mathcal T_h))^d$, using
\Cref{lem:disc-lifting} and the discrete integration-by-parts formula
\eqref{eq:discrete-IBP}, we find that
\begin{align*}
\begin{aligned}
F_h^*(\nabla_h^*\bm z_h)
&=
\sup_{q_h\in\mathcal S^{1,\mathrm{cr}}(\mathcal T_h)}
\cbr[1]{
(\bm z_h,\nabla_h q_h)_\Omega
+(f_h,\Pi_h q_h)_\Omega
-\langle\beta_h,(\pi_h q_h)^+\rangle_{\Gamma_C}
+\langle\alpha_h,(\pi_h q_h)^-\rangle_{\Gamma_C}
-I_{K_h}(q_h)
}
\\
&=
\sup_{\widehat q_h\in\mathcal S^{1,\mathrm{cr}}_{D}(\mathcal T_h)}
\cbr[1]{
(\bm z_h,\nabla_h(\widehat q_h+p_D^h))_\Omega
+(f_h,\Pi_h(\widehat q_h+p_D^h))_\Omega
}
\\
&\qquad
-\langle\beta_h,(\pi_h(\widehat q_h+p_D^h))^+\rangle_{\Gamma_C}
+\langle\alpha_h,(\pi_h(\widehat q_h+p_D^h))^-\rangle_{\Gamma_C}
\\
&=
\begin{cases}
\begin{aligned}
&\langle p_D^h,\bm v_h\cdot\bm n\rangle_{\Gamma_D}
+I_{\{f_h\}}^\Omega(\operatorname{div}\bm v_h)
\\
&\quad
+\displaystyle\sup_{\bar q_h\in\pi_h(K_h|_{\Gamma_C})}
\cbr[1]{
\langle \bm v_h\cdot\bm n,\bar q_h\rangle_{\Gamma_C}
-\langle\beta_h,\bar q_h^{\, +}\rangle_{\Gamma_C}
+\langle\alpha_h,\bar q_h^{\, -}\rangle_{\Gamma_C}}
\end{aligned}
&
\begin{aligned}
&\text{if }\bm z_h=\Pi_h\bm v_h,\\
&\text{where }\bm v_h\in\mathcal{R}T^0(\mathcal T_h),
\end{aligned}
\\[2ex]
+\infty,
& \text{else}.
\end{cases}
\end{aligned}
\end{align*}
Since $\pi_h(K_h|_{\Gamma_C})=\mathcal L_h^0(\mathcal S_h^C)$, every
$\bar q_h\in\pi_h(K_h|_{\Gamma_C})$ is determined by independent constants
on each facet. Hence the remaining supremum is the sum of scalar
suprema over the facets. On a fixed facet $S\in\mathcal S_h^C$, write
$\bar q_h|_S=\xi$. Then
\begin{align*}
&\sup_{\xi\in\mathbb R}
\cbr[1]{
(\bm v_h\cdot\bm n)|_S\xi
-\beta_h|_S\xi^+
+\alpha_h|_S\xi^-}
\\
&\quad=
\max\cbr[1]{
\sup_{\xi\ge0}\del{(\bm v_h\cdot\bm n)|_S-\beta_h|_S}\xi,
\sup_{\xi\le0}\del{(\bm v_h\cdot\bm n)|_S-\alpha_h|_S}\xi}
\\
&\quad=
\begin{cases}
0,
& \alpha_h\le \bm v_h\cdot\bm n\le\beta_h\quad\text{on }S,\\
+\infty,
& \text{otherwise}.
\end{cases}
\end{align*}
Indeed, the first supremum is finite exactly when
$\bm v_h\cdot\bm n\le\beta_h$ on $S$, and the second is finite exactly when
$\bm v_h\cdot\bm n\ge\alpha_h$ on $S$. Therefore,
\begin{align*}
&\sup_{\bar q_h\in\pi_hK_h|_{\Gamma_C}}
\cbr[1]{
\langle \bm v_h\cdot\bm n,\bar q_h\rangle_{\Gamma_C}
-\langle\beta_h,\bar q_h^+\rangle_{\Gamma_C}
+\langle\alpha_h,\bar q_h^-\rangle_{\Gamma_C}}
= I_{[\alpha_h,\beta_h]}^{\Gamma_C}(\bm v_h\cdot\bm n).
\end{align*}
Thus, using the fact that
\begin{align*}
I_{K_h^\star}(\bm v_h)
=I_{\{f_h\}}^\Omega(\operatorname{div}\bm v_h)
+I_{[\alpha_h,\beta_h]}^{\Gamma_C}(\bm v_h\cdot\bm n),
\end{align*}
we obtain
\begin{align}\label{eq:disc-Fstar-final}
F_h^*(\nabla_h^*\bm z_h)
=
\begin{cases}
\begin{aligned}
&I_{K_h^\star}(\bm v_h)
+\langle p_D^h,\bm v_h\cdot\bm n\rangle_{\Gamma_D},
\end{aligned}
&
\begin{aligned}
&\text{if }\bm z_h=\Pi_h\bm v_h,\\
&\text{with }\bm v_h\in\mathcal{R}T^0(\mathcal T_h),
\end{aligned}
\\[2ex]
+\infty,
& \text{otherwise}.
\end{cases}
\end{align}
Combining \eqref{eq:disc-G-conjugate}, \eqref{eq:disc-Fstar-final} and
\eqref{eq:disc-dual-sumof-conjugate} yields \eqref{eq:discrete_dual_functional}.

\medskip
(ii) Both $G_h$ and $F_h$ are proper, convex, and lower semicontinuous;
$G_h$ is moreover continuous on
$(\mathcal L_h^0(\mathcal T_h))^d$. Thus, the Fenchel--Rockafellar theorem
applies and yields
\begin{align}\label{eq:disc-FR}
    \inf_{q_h\in\mathcal S^{1,\mathrm{cr}}(\mathcal T_h)} I_h^{\mathrm{cr}}(q_h)
    =
    \sup_{\bm z_h\in(\mathcal L_h^0(\mathcal T_h))^d}
    \cbr{-G_h^*(-\bm z_h)-F_h^*(\nabla_h^*\bm z_h)},
\end{align}
together with the existence of a maximizer of the right-hand side and at least
one minimizer $p_h^{cr}\in\mathcal S^{1,\mathrm{cr}}(\mathcal T_h)$ of the left-hand
side. Since the supremum in \eqref{eq:disc-FR} may be restricted to
$\Pi_h(\mathcal{R}T^0(\mathcal T_h))$ and satisfies
$-G_h^*(-\Pi_h\bm v_h)-F_h^*(\nabla_h^*\Pi_h\bm v_h)=D_h^{\mathrm{rt}}(\bm v_h)$,
the right-hand side is precisely the maximization of $D_h^{\mathrm{rt}}$ over
$K_h^\star$, attained at $\bm u_h^{rt}$. Hence there exists a minimizer $p_h^{cr}\in K_h$
and the discrete strong-duality identity \eqref{eq:disc_strong_duality} holds.
If $|\Gamma_D|>0$, the discrete Poincar\'e inequality on
$\mathcal S^{1,\mathrm{cr}}_{D}(\mathcal T_h)$ together with the ellipticity of
$\bm K_h$ shows that
$q_h\mapsto\tfrac12\|\bm K_h^{1/2}\nabla_h q_h\|_{2,\Omega}^2$ is strictly
convex on $K_h$. The remaining terms in $I_h^{\mathrm{cr}}$ are convex. Hence
$I_h^{\mathrm{cr}}$ is strictly convex on $K_h$ and the minimizer is unique.

\medskip
(iii) By the standard Fenchel optimality relations, equality in the strong
duality relation implies
\begin{align*}
    -\Pi_h\bm u_h^{rt}\in\partial G_h(\nabla_h p_h^{cr}),
    \qquad
    \nabla_h^*\Pi_h\bm u_h^{rt}\in\partial F_h(p_h^{cr}).
\end{align*}
The first inclusion gives \eqref{eq:disc_optimality_1}. For the second inclusion,
using the definition of the subdifferential and rearranging gives, for every
$q_h\in K_h$,
\begin{align*}
&-(f_h,\Pi_h(q_h-p_h^{cr}))_\Omega
+\langle\beta_h,(\pi_h q_h)^+\rangle_{\Gamma_C}
-\langle\alpha_h,(\pi_h q_h)^-\rangle_{\Gamma_C}
-\langle\beta_h,(\pi_h p_h^{cr})^+\rangle_{\Gamma_C}
+\langle\alpha_h,(\pi_h p_h^{cr})^-\rangle_{\Gamma_C}
\\
&\qquad\ge
(\Pi_h\bm u_h^{rt},\nabla_h(q_h-p_h^{cr}))_\Omega .
\end{align*}
Since $(p_h^{cr},\bm u_h^{rt})\in K_h\times K_h^\star$, the discrete integration-by-parts
identity gives
\begin{align*}
\langle\beta_h,(\pi_h q_h)^+\rangle_{\Gamma_C}
-\langle\alpha_h,(\pi_h q_h)^-\rangle_{\Gamma_C}
-\langle\beta_h,(\pi_h p_h^{cr})^+\rangle_{\Gamma_C}
+\langle\alpha_h,(\pi_h p_h^{cr})^-\rangle_{\Gamma_C}
\ge
\langle\bm u_h^{rt}\cdot\bm n,\pi_h(q_h-p_h^{cr})\rangle_{\Gamma_C}.
\end{align*}
Equivalently,
\begin{align*}
\bm u_h^{rt}\cdot\bm n
\in
\partial\Bigl(
 r_h\mapsto
\langle\beta_h,r_h^+\rangle_{\Gamma_C}
-\langle\alpha_h,r_h^-\rangle_{\Gamma_C}
\Bigr)(\pi_h p_h^{cr})
\quad\text{on }\mathcal L_h^0(\mathcal S_h^C).
\end{align*}
The equality condition in the Fenchel--Young inequality yields
\eqref{eq:disc_optimality_2}. The proof is complete.
\end{proof}

\begin{remark}[Discrete complementarity conditions]\label{rem:disc_complementarity}
Decomposing $\pi_h p_h^{cr}=(\pi_h p_h^{cr})^+-(\pi_h p_h^{cr})^-$ on $\Gamma_C$, the
boundary optimality condition~\eqref{eq:disc_optimality_2} is equivalent to
\begin{align}\label{eq:disc-complementarity-identity}
    \langle\beta_h-\bm u_h^{rt}\cdot\bm n,\,(\pi_h p_h^{cr})^+\rangle_{\Gamma_C}
    +\langle\bm u_h^{rt}\cdot\bm n-\alpha_h,\,(\pi_h p_h^{cr})^-\rangle_{\Gamma_C}=0.
\end{align}
Since $\bm u_h^{rt}\in K_h^\star$, we have
$\alpha_h\le \bm u_h^{rt}\cdot\bm n\le\beta_h$ on every facet
$S\in\mathcal S_h^C$, so both terms in
\eqref{eq:disc-complementarity-identity} are non-negative. Since they sum to
zero, each vanishes. Since the factors are facetwise constant on
$\mathcal S_h^C$, this gives the facetwise complementarity relations
\begin{align}\label{eq:disc-complementarity-facetwise}
    (\beta_h-\bm u_h^{rt}\cdot\bm n)\,(\pi_h p_h^{cr})^+ = 0
    \quad\text{and}\quad
    (\bm u_h^{rt}\cdot\bm n-\alpha_h)\,(\pi_h p_h^{cr})^- = 0
    \qquad\text{on every } S\in\mathcal S_h^C.
\end{align}
Equivalently, for every $S\in\mathcal S_h^C$,
\begin{align*}
    (\pi_h p_h^{cr})|_S>0
    &\;\Longrightarrow\; \bm u_h^{rt}\cdot\bm n|_S=\beta_h|_S
        &&\text{(upper bound active)},\\
    (\pi_h p_h^{cr})|_S<0
    &\;\Longrightarrow\; \bm u_h^{rt}\cdot\bm n|_S=\alpha_h|_S
        &&\text{(lower bound active)},\\
    \alpha_h|_S<\bm u_h^{rt}\cdot\bm n|_S<\beta_h|_S
    &\;\Longrightarrow\; (\pi_h p_h^{cr})|_S=0
        &&\text{(bounds inactive)}.
\end{align*}
\end{remark}

\section{A priori error analysis }\label{ssec:disc-apriori}

In this section, resorting to the discrete convex duality relations established in \Cref{sec:discrete-flux-constrained}, we derive a discrete error identity (\Cref{thm:disc-aposteriori-id}) that applies to arbitrary admissible approximations of the discrete primal and discrete dual problem at the same time. From this identity, evaluated at quasi-interpolants of the exact solution, we derive a priori error estimates with error decay rates given only fractional regularity assumptions on the solution and the flux bounds (\Cref{thm:disc-apriori}).

\subsection{Discrete primal-dual gap estimator}
We now derive an exact discrete primal--dual gap error identity.
Throughout, $K_h$ and $K_h^\star$ denote the discrete admissible primal and
dual sets introduced in~\eqref{eq:disc-primal-admissible}
and~\eqref{eq:disc-Kh}, respectively, and $(p_h^{cr},\bm u_h^{rt})\in K_h\times
K_h^\star$ is the discrete primal-dual solution pair from
Theorem~\ref{thm:disc-strong-duality}. Define the \emph{discrete
primal-dual gap estimator}
$\eta_{\mathrm{gap},h}^2 \colon K_h\times K_h^\star \to [0,+\infty)$, for
every $q_h \in K_h$ and $\bm v_h \in K_h^\star$, via
\begin{align} \label{eq:disc-gap}
\eta_{\mathrm{gap},h}^2(q_h,\bm v_h) := I_h^{cr}(q_h) - D^{rt}(\bm v_h).
\end{align}
The discrete primal-dual gap estimator serves as a \emph{distance measure}
between a given admissible discrete primal-dual pair $(q_h,\bm v_h) \in
K_h\times K_h^\star$ and the discrete solution $(p_h^{cr},\bm u_h^{rt}) \in K_h\times
K_h^\star$. The following lemma shows that the discrete primal-dual gap
estimator decomposes into a contribution measuring the violation of the
optimality condition~\eqref{eq:disc_optimality_1} and a contribution
measuring the violation of the optimality
condition~\eqref{eq:disc_optimality_2}.

\begin{lemma}[Decomposition of the discrete gap estimator]\label{lem:disc-gap-decomp}
For every $q_h \in K_h$ and every $\bm v_h \in K_h^\star$,
\begin{align*}
    \eta_{\mathrm{gap},h}^2(q_h,\bm v_h) = \eta_{\mathrm{gap},h,\mathrm{I}}^2(q_h,\bm v_h) + \eta_{\mathrm{gap},h,\mathrm{II}}^2(q_h,\bm v_h),
\end{align*}
where
\begin{align*}
    \eta_{\mathrm{gap},h,\mathrm{I}}^2(q_h,\bm v_h) & := \tfrac{1}{2}\|\bm K_h^{\frac{1}{2}} \nabla_h q_h
                                   + \bm K_h^{-\frac{1}{2}} \Pi_h\bm v_h\|_{2,\Omega}^2,\\
   \eta_{\mathrm{gap},h,\mathrm{II}}^2(q_h,\bm v_h)
    & := -  \langle\bm v_h\cdot \bm n, \pi_h q_h\rangle_{\Gamma_C}
    + \langle \beta_h, (\pi_h q_h)^+\rangle_{\Gamma_C}
    - \langle \alpha_h, (\pi_h q_h)^-\rangle_{\Gamma_C}.
\end{align*}
\end{lemma}

\begin{proof}
Note that for all $q_h \in K_h$ and $\bm v_h \in K_h^\star$,
\begin{align*}
    \eta_{\mathrm{gap},h}^2(q_h,\bm v_h)
    &= \tfrac{1}{2} \|\bm K_h^{\frac12} \nabla_h q_h\|_{2,\Omega}^2
       - (f_h,\Pi_h q_h)_\Omega
       + \langle \beta_h, (\pi_h q_h)^+\rangle_{\Gamma_C}
       - \langle \alpha_h, (\pi_h q_h)^-\rangle_{\Gamma_C} \\
      &\quad +\tfrac{1}{2} \|\bm K_h^{-\frac{1}{2}}\Pi_h\bm v_h\|_{2,\Omega}^2
       +\langle p_D^h,\bm v_h\cdot \bm n\rangle_{\Gamma_D} \\
      &= \tfrac{1}{2} \|\bm K_h^{\frac12} \nabla_h q_h\|_{2,\Omega}^2
       +(\Pi_h\bm v_h ,\nabla_h q_h)_\Omega
       - \langle\bm v_h\cdot\bm n,\pi_h q_h\rangle_{\partial\Omega}
       + \langle \beta_h, (\pi_h q_h)^+\rangle_{\Gamma_C} \\
      &\quad - \langle \alpha_h, (\pi_h q_h)^-\rangle_{\Gamma_C}
       +\tfrac{1}{2} \|\bm K_h^{-\frac{1}{2}}\Pi_h\bm v_h\|_{2,\Omega}^2
       +\langle p_D^h,\bm v_h\cdot \bm n\rangle_{\Gamma_D},
\end{align*}
where we have used $\textup{div}\,\bm v_h = f_h$ and the discrete
integration-by-parts identity~\eqref{eq:discrete-IBP}. Using the fact that 
$\pi_h q_h = \pi_h p_D^h$ on $\Gamma_D$ and
rearranging yields
\begin{align*}
    \eta_{\mathrm{gap},h}^2(q_h,\bm v_h)
    &= \tfrac{1}{2} \|\bm K_h^{\frac12} \nabla_h q_h\|_{2,\Omega}^2 + (\Pi_h\bm v_h,\nabla_h q_h)_\Omega
       - \langle \bm v_h\cdot\bm n, \pi_h q_h\rangle_{\Gamma_C} \\
      &\quad + \langle \beta_h, (\pi_h q_h)^+\rangle_{\Gamma_C}
       - \langle \alpha_h, (\pi_h q_h)^-\rangle_{\Gamma_C}
       +\tfrac{1}{2} \|\bm K_h^{-\frac{1}{2}}\Pi_h\bm v_h\|_{2,\Omega}^2.
\end{align*}
Since $(\Pi_h\bm v_h,\nabla_h q_h)_\Omega
= (\bm K_h^{\frac12} \nabla_h q_h,\,\bm K_h^{-\frac12}\Pi_h\bm v_h)_\Omega$,
the result follows after completing the square.
\end{proof}

Next, we identify \emph{optimal strong convexity measures}
$\rho_{I,h}^2: K_h \to [0,+\infty)$ and $\rho_{-D,h}^2: K_h^\star \to
[0,+\infty)$ for the discrete primal energy
functional~\eqref{eq:discrete_predual} at a discrete primal solution $p_h^{cr}
\in K_h$, and for the negative of the discrete dual energy
functional~\eqref{eq:discrete_dual_functional} at the discrete dual
solution $\bm u_h^{rt} \in K_h^\star$. Let
\begin{align*}
    \rho_{I,h}^2(q_h) &:= I_h^{cr}(q_h) - I_h^{cr}(p_h^{cr}), \\
    \rho_{-D,h}^2(\bm v_h) &:= -D^{rt}(\bm v_h) + D^{rt}(\bm u_h^{rt}).
\end{align*}
The following lemma shows that, similar to the discrete primal-dual gap
estimator, the strong convexity measures decompose into a contribution
measuring the violation of the optimality
condition~\eqref{eq:disc_optimality_1} and a contribution measuring the
violation of the optimality condition~\eqref{eq:disc_optimality_2}.

\begin{lemma}[Optimal strong convexity measures]\label{lem:disc-convexity}
For every $q_h \in K_h$ and every $\bm v_h \in K_h^\star$,
\begin{align}
    \rho_{I,h}^2(q_h) &= \tfrac{1}{2}\|\bm K_h^{\frac12}\nabla_h (q_h - p_h^{cr})\|_{2,\Omega}^2
                  - \langle \bm u_h^{rt} \cdot \bm n,\, \pi_h q_h\rangle_{\Gamma_C}
                  + \langle \beta_h, (\pi_h q_h)^+\rangle_{\Gamma_C}
                  - \langle \alpha_h, (\pi_h q_h)^-\rangle_{\Gamma_C}, \label{eq:disc-rho-I} \\
    \rho_{-D,h}^2(\bm v_h) &= \tfrac{1}{2} \|\bm K_h^{-\frac{1}{2}}\Pi_h(\bm v_h - \bm u_h^{rt})\|_{2,\Omega}^2
                  - \langle \bm v_h \cdot \bm n,\, \pi_h p_h^{cr}\rangle_{\Gamma_C}
                  + \langle \beta_h, (\pi_h p_h^{cr})^+\rangle_{\Gamma_C}
                  - \langle \alpha_h, (\pi_h p_h^{cr})^-\rangle_{\Gamma_C}. \label{eq:disc-rho-D}
\end{align}
\end{lemma}

\begin{proof}
To show~\eqref{eq:disc-rho-I}, we expand $\rho_{I,h}^2(q_h)=I_h^{cr}(q_h)-I_h^{cr}(p_h^{cr})$:
\begin{align*}
    \rho_{I,h}^2(q_h)
    &= \tfrac{1}{2} \|\bm K_h^{\frac12}\nabla_h q_h\|_{2,\Omega}^2
       -\tfrac{1}{2} \|\bm K_h^{\frac12} \nabla_h p_h^{cr}\|_{2,\Omega}^2
       - (f_h,\Pi_h(q_h-p_h^{cr}))_\Omega \\
    &\quad + \langle \beta_h, (\pi_h q_h)^+ - (\pi_h p_h^{cr})^+\rangle_{\Gamma_C}
       - \langle \alpha_h, (\pi_h q_h)^- - (\pi_h p_h^{cr})^-\rangle_{\Gamma_C}.
\end{align*}
Substituting $f_h=\textup{div}\,\bm u_h^{rt}$ and applying the discrete
integration-by-parts identity~\eqref{eq:discrete-IBP},
\begin{align*}
    \rho_{I,h}^2(q_h)
    &= \tfrac{1}{2} \|\bm K_h^{\frac12}\nabla_h q_h\|_{2,\Omega}^2
       -\tfrac{1}{2} \|\bm K_h^{\frac12} \nabla_h p_h^{cr}\|_{2,\Omega}^2
       +(\Pi_h\bm u_h^{rt},\nabla_h(q_h-p_h^{cr}))_\Omega
       - \langle\bm u_h^{rt}\cdot\bm n,\pi_h(q_h-p_h^{cr})\rangle_{\Gamma_C} \\
    &\quad + \langle \beta_h, (\pi_h q_h)^+ - (\pi_h p_h^{cr})^+\rangle_{\Gamma_C}
       - \langle \alpha_h, (\pi_h q_h)^- - (\pi_h p_h^{cr})^-\rangle_{\Gamma_C} \\
    &= \tfrac{1}{2}\|\bm K_h^{\frac12}\nabla_h(q_h-p_h^{cr})\|_{2,\Omega}^2
       - \langle \bm u_h^{rt}\cdot\bm n,\, \pi_h q_h\rangle_{\Gamma_C}
       + \langle \beta_h, (\pi_h q_h)^+\rangle_{\Gamma_C}
       - \langle \alpha_h, (\pi_h q_h)^-\rangle_{\Gamma_C},
\end{align*}
where we have used that $q_h-p_h^{cr} \in
\mathcal{S}^{1,\mathrm{cr}}_{D}(\mathcal{T}_h)$, that $\Pi_h\bm u_h^{rt} = -\bm
K_h\nabla_h p_h^{cr}$ by~\eqref{eq:disc_optimality_1}, and that $\langle\bm u_h^{rt}\cdot \bm n, \pi_h p_h^{cr}\rangle_{\Gamma_C}
= \langle\beta_h, (\pi_h p_h^{cr})^+\rangle_{\Gamma_C} - \langle\alpha_h,
(\pi_h p_h^{cr})^-\rangle_{\Gamma_C}$
by~\eqref{eq:disc_optimality_2} to pass to the second line.

To show~\eqref{eq:disc-rho-D}, we expand $\rho_{-D,h}^2(\bm v_h)=-D^{rt}(\bm v_h)+D^{rt}(\bm u_h^{rt})$:
\begin{align*}
    \rho_{-D,h}^2(\bm v_h)
    &= \tfrac{1}{2}\|\bm K_h^{-\frac{1}{2}}\Pi_h\bm v_h\|_{2,\Omega}^2
       - \tfrac{1}{2}\|\bm K_h^{-\frac{1}{2}}\Pi_h\bm u_h^{rt}\|_{2,\Omega}^2
       + \langle p_D^h, (\bm v_h -\bm u_h^{rt})\cdot\bm n\rangle_{\Gamma_D} \\
    &=  \tfrac{1}{2}\|\bm K_h^{-\frac{1}{2}}\Pi_h(\bm v_h-\bm u_h^{rt})\|_{2,\Omega}^2
       + (\bm K_h^{-\frac{1}{2}}\Pi_h(\bm v_h-\bm u_h^{rt}),\,\bm K_h^{-\frac{1}{2}}\Pi_h\bm u_h^{rt})_\Omega
       + \langle p_D^h,(\bm v_h-\bm u_h^{rt})\cdot\bm n\rangle_{\Gamma_D}.
\end{align*}
Then, using \eqref{eq:disc_optimality_1} and the discrete
integration-by-parts identity~\eqref{eq:discrete-IBP},
\begin{align*}
        \rho_{-D,h}^2(\bm v_h)
    &=  \tfrac{1}{2}\|\bm K_h^{-\frac{1}{2}}\Pi_h(\bm v_h-\bm u_h^{rt})\|_{2,\Omega}^2
       - (\Pi_h(\bm v_h-\bm u_h^{rt}), \nabla_h p_h^{cr})_\Omega
       + \langle p_D^h,(\bm v_h-\bm u_h^{rt})\cdot\bm n\rangle_{\Gamma_D} \\
    &=  \tfrac{1}{2}\|\bm K_h^{-\frac{1}{2}}\Pi_h(\bm v_h-\bm u_h^{rt})\|_{2,\Omega}^2
       - \langle (\bm v_h -\bm u_h^{rt})\cdot \bm n, \pi_h p_h^{cr} \rangle_{\Gamma_C},
\end{align*}
where we have used the fact that $\textup{div}\,(\bm v_h - \bm u_h^{rt}) = 0$.
Finally, another application of \eqref{eq:disc_optimality_2}
yields
\begin{align*}
        \rho_{-D,h}^2(\bm v_h)
    &=  \tfrac{1}{2}\|\bm K_h^{-\frac{1}{2}}\Pi_h(\bm v_h-\bm u_h^{rt})\|_{2,\Omega}^2
       - \langle \bm v_h \cdot \bm n, \pi_h p_h^{cr} \rangle_{\Gamma_C}
       + \langle \beta_h, (\pi_h p_h^{cr})^+\rangle_{\Gamma_C}
       - \langle \alpha_h, (\pi_h p_h^{cr})^-\rangle_{\Gamma_C}.
\end{align*}
The proof is now complete.
\end{proof}

We next derive an error identity that
characterizes the \emph{discrete primal-dual total error}
$\rho_{\mathrm{tot},h}^2 : K_h \times K_h^\star \to [0,+\infty)$, defined
for every $(q_h,\bm v_h) \in K_h \times K_h^\star$ by
\begin{align} \label{eq:disc-total_error}
    \rho_{\mathrm{tot},h}^2(q_h,\bm v_h) := \rho_{I,h}^2(q_h) + \rho_{-D,h}^2(\bm v_h),
\end{align}
in terms of the discrete primal-dual gap estimator~\eqref{eq:disc-gap}:
\begin{theorem}[Discrete a posteriori error identity]\label{thm:disc-aposteriori-id}
For every $q_h\in K_h$ and every $\bm v_h\in K_h^\star$,
\begin{equation*}
    \rho_{\mathrm{tot},h}^2(q_h,\bm v_h) = \eta_{\mathrm{gap},h}^2(q_h,\bm v_h).
\end{equation*}
\end{theorem}

\begin{proof}
By the discrete strong duality identity~\eqref{eq:disc_strong_duality},
$I_h^{cr}(p_h^{cr})=D^{rt}(\bm u_h^{rt})$. Hence
\begin{align*}
    \rho_{\mathrm{tot},h}^2(q_h,\bm v_h)
    &= \del{I_h^{cr}(q_h) - I_h^{cr}(p_h^{cr})} + \del[1]{-D^{rt}(\bm v_h) + D^{rt}(\bm u_h^{rt})}
    = I_h^{cr}(q_h) - D^{rt}(\bm v_h)
    = \eta_{\mathrm{gap},h}^2(q_h,\bm v_h). \qedhere
\end{align*}
\end{proof}

\subsection{Convergence and a priori error estimate}

Using the discrete error identity of \Cref{thm:disc-aposteriori-id}, we now derive convergence under minimal regularity and explicit a priori error decay rates under fractional regularity assumptions.

\begin{theorem}[A priori error estimates]\label{thm:disc-apriori}
Suppose that $\alpha, \beta$, and $\bm K$ satisfy \Cref{assumption:data-assumption}. Let $(p,\bm u) \in K \times K^\star$ be an exact primal--dual pair from~\Cref{lem:strong-duality} and suppose that $\bm u \in (L^{q}(\Omega))^d$ with $q > 2$ so that $\Pi_h^{rt} \bm u$ is well-defined. The following statements apply:
\begin{itemize}
    \item[(i)] If merely $(p,\bm u) \in K \times \del[1]{K^\star \cap (L^{q}(\Omega))^d}$, then
    \begin{align} \label{eq:convergence_no_regularity}
        \lim_{h \to 0^+} \rho^2_{\textup{tot},h}(\Pi_h^{cr} p,\Pi_h^{rt}\bm u) = 0.
    \end{align}
    \item[(ii)] If, moreover,  $p \in H^{1+s}(\Omega)$ and $\bm K \in (W^{1,\infty}(\mathcal{T}_h))^{d\times d}$ with $\tfrac{1}{2} < s \le 1$, and $\alpha \in H^{t_\alpha}(\mathcal{S}_h^C)$, $\beta \in H^{t_\beta}(\mathcal{S}_h^C)$ with $0 < t_\alpha,t_\beta \le s-\tfrac{1}{2}$, then for $\varepsilon = 0$ when $s < 1$ and for all $\varepsilon > 0$ when $s=1$,
    \begin{align} \label{eq:convergence_regularity}
\rho_{\textup{tot},h}^2(\Pi_h^{cr} p,\Pi_h^{rt}\bm u) \lesssim h^{s+\textup{min}(t_\beta,t_\alpha)+\frac{1}{2} - \varepsilon}.
    \end{align}
In particular, if $t_\alpha = t_\beta = s-\tfrac{1}{2}$, then
    \begin{align} \label{eq:convergence_max_regularity}
\rho_{\textup{tot},h}^2(\Pi_h^{cr} p,\Pi_h^{rt}\bm u) \lesssim h^{2s - \varepsilon}.
    \end{align}
\end{itemize}
\end{theorem}
\begin{proof}
(i) By \Cref{thm:disc-aposteriori-id} and \Cref{lem:disc-gap-decomp},
\begin{align*}
    \rho^2_{\text{tot},h}(\Pi_h^{cr} p,\Pi_h^{rt}\bm u)
 &= \eta^2_{\text{gap},h}(\Pi_h^{cr} p,\Pi_h^{rt}\bm u) = I_h^1 + I_h^2,
\end{align*}
where we have defined
\begin{align}
       I_h^1 &:= \tfrac{1}{2}\|\bm K_h^{\frac{1}{2}}\,\nabla_h\Pi_h^{cr} p
       + \bm K_h^{-\frac{1}{2}}\,\Pi_h\Pi_h^{rt}\bm u\|_{2,\Omega}^2, \label{eq:Ih1}\\
   I_h^2 &:=  -\langle \Pi_h^{rt}\bm u\cdot \bm n, \pi_h \Pi_h^{cr} p\rangle_{\Gamma_C}
   + \langle \beta_h, (\pi_h \Pi_h^{cr}p)^+\rangle_{\Gamma_C} - \langle \alpha_h,(\pi_h \Pi_h^{cr}p)^-\rangle_{\Gamma_C}.  \label{eq:Ih2}
\end{align}
We begin by bounding $I_h^1$. By \eqref{eq:CR-comm-grad} and \eqref{eq:cts_optimality_1}, $\nabla_h \Pi_h^{cr} p = \Pi_h \nabla p = -\Pi_h(\bm K^{-1} \bm u)$. Therefore,
\begin{align} \label{eq:Ih1.1}
    I_h^1 &= \tfrac{1}{2}\|\bm K_h^{-\frac{1}{2}}\,\Pi_h\Pi_h^{rt}\bm u -\bm K_h^{\frac{1}{2}}\Pi_h(\bm K^{-1} \bm u)
       \|_{2,\Omega}^2.
\end{align}
Since $\bm K_h \in (\mathcal{L}_h^0(\mathcal{T}_h))^{d\times d}$, it holds that
$\bm K_h^{\frac{1}{2}}\,\Pi_h(\bm K_h^{-1}\bm u)
= \bm K_h^{-\frac{1}{2}}\,\Pi_h\bm u$.  Adding and subtracting this
quantity inside the norm in~\eqref{eq:Ih1.1} and using
the triangle inequality and the $L^\infty$-stability of $\Pi_h$, 
\begin{align}
   I_h^1
   &\lesssim \|\bm K_h^{-\frac{1}{2}}\,\Pi_h(\Pi_h^{rt}\bm u-\bm u) \|_{2,\Omega}^2
    + \|\bm K_h^{\frac{1}{2}} \,\Pi_h((\bm K_h^{-1}-\bm K^{-1})\bm u\bigr)\|_{2,\Omega}^2  
   \label{eq:I_h^1_bound}
\end{align}
Next, we turn to bounding $I_h^2$. On the one hand, the identities \eqref{eq:CR-comm-trace} and \eqref{eq:RT-comm-trace} yield
\begin{align}
       I_h^2& =  -\langle \bm u\cdot \bm n, \pi_h  p\rangle_{\Gamma_C}
   + \langle \beta, (\pi_h p)^+\rangle_{\Gamma_C} - \langle \alpha,(\pi_h p)^-\rangle_{\Gamma_C} \notag \\
   &= \langle \beta - \bm u\cdot \bm n, (\pi_h p)^+\rangle_{\Gamma_C}
   + \langle \bm u\cdot \bm n - \alpha, (\pi_h p)^-\rangle_{\Gamma_C},
\end{align}
where we have decomposed $\pi_h p$ into its positive and negative parts and used the fact that $q_h^+,q_h^- \in \mathcal{L}_h^0(\mathcal{S}_h^C)$ for any $q_h \in \mathcal{L}_h^0(\mathcal{S}_h^C)$. Optimality condition \eqref{eq:cts_optimality_2} then yields, after splitting $p$ into its positive and negative parts,
\begin{align}
   I_h^2 &=  \langle \beta - \bm u\cdot \bm n, (\pi_h p)^+ - p^+\rangle_{\Gamma_C}
   + \langle \bm u\cdot \bm n - \alpha, (\pi_h p)^- - p^-\rangle_{\Gamma_C}. \label{eq:Ih2_reg_bnd}
\end{align}
Using the Cauchy--Schwarz inequality and the fact that the maps $p \mapsto p^\pm$ are $1$-Lipschitz,
\begin{align}
   I_h^2 
    \le \del{\|\beta - \bm u \cdot \bm n\|_{2,\Gamma_C} + \|\bm u \cdot \bm n - \alpha\|_{2,\Gamma_C}} \|\pi_h p - p\|_{2,\Gamma_C}. \label{eq:Ih2_no_reg_bnd}
\end{align}
Using the bounds \eqref{eq:I_h^1_bound} and \eqref{eq:Ih2_no_reg_bnd} and the stability of the $L^2$-projection, we find
\begin{align} \label{eq:rho_bound_low_regularity}
    \rho^2_{\textup{tot},h}(\Pi_h^{cr} p,\Pi_h^{rt}\bm u) \lesssim \|\Pi_h^{rt}\bm u-\bm u\|_{2,\Omega}^2 + \| (\bm K_h^{-1}-\bm K^{-1})\bm u\|_{2,\Omega}^2 + \|\pi_h p - p\|_{2,\Gamma_C}.
\end{align}
One can show using a density argument that as $h\to 0^+$, $\Pi_h^{rt} \bm u \to \bm u$ in $V$,  $\pi_h p \to p$ in $L^2(\Gamma_C)$, and $\bm K_h \to \bm K$ in $(L^2(\Omega))^{d \times d}$. Since $\bm K^{-1}, \bm K_h^{-1}$ are uniformly bounded, we conclude $\bm K_h^{-1} \to \bm K^{-1}$ in $(L^2(\Omega))^{d \times d}$ and therefore also pointwise a.e. in $\Omega$. The dominated convergence theorem then yields $\bm K_h^{-1} \bm u \to \bm K^{-1} \bm u$ in $(L^2(\Omega))^d$. Thus, passing to the limit as $h \to 0^+$ in \eqref{eq:rho_bound_low_regularity} yields \eqref{eq:convergence_no_regularity}.

\bigskip
\noindent (ii) Suppose now that $p \in H^{1+s}(\Omega)$ with $\tfrac{1}{2} < s \le 1$. By the assumed broken regularity $\bm K \in (W^{1,\infty}(\mathcal{T}_h))^{d \times d}$, it follows that $\bm u \in (H^s(\mathcal{T}_h))^d$ and since $s > \tfrac{1}{2}$, it holds that $\bm u \cdot \bm n |_{\Gamma_C}\in H^{s-\frac{1}{2}}(\Gamma_C)$. Moreover, the following result concerning the regularity of the positive (resp. negative) parts of functions holds (\emph{cf}. \cite[Rem. 0.1]{Savare1996}):
\begin{align*}
 p \in H^{s+\frac{1}{2}}(\Gamma_C) \quad \Rightarrow \quad p^+,p^- \in H^{s+\frac{1}{2}}(\Gamma_C), \quad \forall\, s < 1.
\end{align*}
Thus, at the endpoint $s=1$, we can only conclude that $p^+,p^- \in H^{s+\frac{1}{2}- \varepsilon}(\Gamma_C)$ for all $\varepsilon > 0$.

We begin by bounding the volume contribution $I_h^1$ via the right hand side of \eqref{eq:I_h^1_bound}. Note that
\begin{align*}
   I_h^1
   & \lesssim \| \bm K^{-\frac{1}{2}} \|_{\infty,\Omega}^2\|\Pi_h^{rt}\bm u-\bm u\|_{2,\Omega}^2 +  \| \bm K^{\frac{1}{2}} \|_{\infty,\Omega}^2 \| \bm K^{-1}-\bm K_h^{-1} \|_{\infty,\Omega}^2 \|\bm u\|_{2,\Omega}^2. 
\end{align*}
Therefore, the approximation properties of the Raviart--Thomas interpolant~\eqref{eq:RT-approx} and of the $L^2$-projection $\Pi_h$~\eqref{eq:proj-Linfty} yield
\begin{align} \label{eq:u-Pi_rtu}
  \|\Pi_h^{rt} \bm u - \bm u\|_{2,\Omega}^2 &\lesssim h^{2s}|\bm u|_{H^s(\mathcal{T}_h)}^2, \\
  \label{eq:Kinv-Pi_hKinv}
\|\bm K^{-1} - \bm K_h^{-1} \|_{\infty,\Omega}^2 &\lesssim h^{2} | \bm K^{-1} |_{W^{1,\infty}(\mathcal{T}_h)}^2.
\end{align}
and therefore, we have the following bound on the volume contribution:
\begin{align} \label{eq:Ih1_rate}
    I_h^1 \lesssim h^{2s}.
\end{align}
It remains to bound the boundary contribution $I_h^2$.
Note that proceeding from the bound \eqref{eq:Ih2_no_reg_bnd} used to prove \eqref{eq:convergence_no_regularity} results in a suboptimal error estimate. Therefore, we instead return to \eqref{eq:Ih2_reg_bnd} and use the convexity of the maps $p \mapsto p^+$, $p \mapsto p^-$ and Jensen's inequality,
\begin{align*}
   I_h^2 &\le  \langle \beta - \bm u\cdot \bm n, \pi_h (p^+) - p^+\rangle_{\Gamma_C}
   + \langle \bm u\cdot \bm n - \alpha, \pi_h (p^-) - p^-\rangle_{\Gamma_C}  \\
   &=  \langle z_\beta - \pi_h z_\beta, \pi_h (p^+) - p^+\rangle_{\Gamma_C}
   + \langle z_\alpha - \pi_h z_\alpha, \pi_h (p^-) - p^-\rangle_{\Gamma_C} \\
   &:= I_h^{\beta} + I_h^{\alpha},
\end{align*}
where we have defined $z_\beta = \beta - \bm u\cdot \bm n$ and $z_\alpha =  \bm u\cdot \bm n - \alpha$ for notational brevity. Adding and subtracting $\pi_h^1 p^+$, using the fact that $\pi_h\del{\pi_h (p^+) - \pi_h^1 (p^+)}= 0$, we have
\begin{align*}
    I_h^{\beta} &= \langle z_\beta - \pi_h z_\beta, \pi_h^1 (p^+) - p^+\rangle_{\Gamma_C} + \langle  z_\beta - \pi_h z_\beta, \pi_h (p^+) - \pi_h^1 (p^+)\rangle_{\Gamma_C}  \\
   &= \langle z_\beta - \pi_h z_\beta, \pi_h^1 (p^+) - p^+\rangle_{\Gamma_C} + \langle  z_\beta, \pi_h (p^+) - \pi_h^1 (p^+)\rangle_{\Gamma_C} \\
    &:= I_h^{\beta,1} + I_h^{\beta,2}.
\end{align*}
Similarly, we have
\begin{align*}
    I_h^\alpha &= \langle z_\alpha - \pi_h z_\alpha, \pi_h^1 (p^-) - p^-\rangle_{\Gamma_C} + \langle  z_\alpha, \pi_h (p^-) - \pi_h^1 (p^-)\rangle_{\Gamma_C} := I_h^{\alpha,1} + I_h^{\alpha,2}.
\end{align*}
To bound $I_h^{\beta,1}$ and $I_h^{\alpha,1}$, we apply the Cauchy--Schwarz
inequality, the facet approximation estimate~\eqref{eq:proj-approx-S} for
$\pi_h$ (with $r=t_\beta$ and $r=t_\alpha$) and~\eqref{eq:proj-approx-S1} for
$\pi_h^1$ (with $r=s+\tfrac12-\varepsilon$), together with the trace bound
$|p^\pm|_{H^{s+1/2-\varepsilon}(\mathcal S_h^C)}\lesssim\|p\|_{H^{1+s}(\Omega)}$:
\begin{align}
    I_h^{\beta,1} + I_h^{\alpha,1} &\le \|z_\beta - \pi_h z_\beta\|_{2,\Gamma_C} \|p^+ - \pi_h^1(p^+) \|_{2,\Gamma_C} + \|z_\alpha - \pi_h z_\alpha\|_{2,\Gamma_C} \|p^- - \pi_h^1(p^-) \|_{2,\Gamma_C} \notag \\
    &\lesssim h^{s+\textup{min}(t_\beta,t_\alpha)+\frac{1}{2}-\varepsilon} \del[2]{|\beta - \bm u \cdot \bm n|_{H^{t_\beta}(\mathcal{S}_h^C)}^2 + | \bm u \cdot \bm n - \alpha|_{H^{t_\alpha}(\mathcal{S}_h^C)}^2 + \|p\|_{H^{s+1}(\Omega)}^2}.
    \label{eq:Ihbeta1_Ihalpha1_sum}
\end{align}
To bound $I_h^{\beta,2}$ and $I_h^{\alpha,2}$, we apply \Cref{lem:mean_zero_tangential_gradient} to find
\begin{align*}
   I_h^{\beta,2} +  I_h^{\alpha,2} &= 
    -\sum_{S \in \mathcal{S}_h^C} \del[2]{\langle z_\beta , \nabla_S \pi_h^1 p^+ \cdot  (x - x_S) \rangle_{S} + \langle z_\alpha , \nabla_S \pi_h^1 p^- \cdot  (x - x_S) \rangle_{S}}.
\end{align*}
To derive sharp bounds on $I_h^{\beta,2} + I_h^{\alpha,2}$, we exploit complementarity (\emph{cf}. \Cref{rem:complementarity}) to perform a finer analysis on each facet $S \in \mathcal{S}_h^C$. 
Note that $z_\beta p^+ = 0$ and $z_\alpha p^- = 0$ a.e. on $\Gamma_C$. Consequently, $z_\beta = 0$ and $z_\alpha = 0$ a.e. on the sets $\cbr{p^+ > 0}$ and $\cbr{p^- > 0}$ respectively. This, combined with the fact that $\nabla_S p^+ = 0$ and $\nabla_S p^- = 0$ a.e. on the sets $\cbr{p^+ = 0}$ and $\cbr{p^- = 0}$ respectively, yields for all $S \in \mathcal{S}_h^C$
\begin{align}
\label{eq:tangential_p+_vanish}
    \langle z_\beta, \nabla_S p^+ \cdot (x-x_S) \rangle_{S} &= 0, \\
    \langle z_\alpha, \nabla_S p^- \cdot (x-x_S) \rangle_{S} &= 0.
    \label{eq:tangential_p-_vanish}
\end{align}
For a given facet $S \in \mathcal{S}_h^C$, there are three cases to consider depending on the relative sizes of the sets $\cbr[0]{p^+ > 0}$ and $\cbr[0]{p^- > 0}$. 
\medskip

\noindent \textbf{Case (1):} \,$|S\cap\{p^+>0\}|\ge\tfrac12|S|$ and therefore $|S\cap\{p^-=0\}|\ge\tfrac12|S|$. 

\smallskip
\noindent
In this case, since $z_\beta p^+ = 0$ a.e. on $\Gamma_C$,  complementarity forces $z_\beta = 0$ on $|S\cap\{p^+>0\}|$. Using this fact, \eqref{eq:tangential_p+_vanish}, H\"older's inequality, \eqref{eq:avg_bound_on_S}, the estimate~\eqref{eq:proj-approx-S} for $\pi_h$ (with $r=t_\beta$), and the estimate~\eqref{eq:proj-approx-S1-grad} for $\pi_h^1$ (with $r=s+\tfrac12-\varepsilon$), we find
\begin{align}
\langle z_\beta , \nabla_S &\pi_h^1 p^+ \cdot  (x - x_S) \rangle_{S} \notag \\&= \big\langle  z_\beta - \langle z_\beta \rangle_{S\cap\{p^+>0\}}, \nabla_S (\pi_h^1 p^+ - p^+) \cdot  (x - x_S) \big\rangle_{S} \notag \\
& \le \|z_\beta - \langle z_\beta \rangle_{S\cap\{p^+>0\}}\|_{2,S} \|\nabla_S (\pi_h^1 p^+ - p^+)\|_{2,S} \|x - x_S\|_{\infty,S} \notag \\
& \le 2h_S \| z_\beta - \pi_h z_\beta \|_{2,S} \|\nabla_S (\pi_h^1 p^+ - p^+)\|_{2,S} \notag \\
&\lesssim h_S^{s+t_\beta+\frac{1}{2}- \varepsilon} | \beta - \bm u \cdot \bm n |_{H^{t_\beta}(S)} |p^+|_{H^{s+\frac{1}{2}- \varepsilon}(S)}.\label{eq:p^+>0_dominant_bnd}
\end{align}
Moreover, using the fact that $\nabla_S\pi_h^1 p^- \cdot  (x - x_S)$ has vanishing mean, that $\langle \nabla_S p^- \rangle_{S\cap\{p^-=0\}} = 0$, H\"older's inequality, \eqref{eq:avg_bound_on_S}, the  estimate~\eqref{eq:proj-approx-S} for $\pi_h$ (with $r=t_\alpha$), the estimate~\eqref{eq:proj-approx-S1-grad} for $\pi_h^1$ (with $r=s+\tfrac12-\varepsilon$), and~\eqref{eq:proj-approx-S-grad} (with $r=s-\tfrac12-\varepsilon$), we find
\begin{align}
    \langle z_\alpha , \nabla_S &\pi_h^1 p^- \cdot  (x - x_S) \rangle_{S} \notag  \\&= \big\langle  z_\alpha - \pi_h z_\alpha, \del[1]{\nabla_S \pi_h^1 p^- - \langle \nabla_S p^- \rangle_{S\cap\{p^-=0\}}} \cdot  (x - x_S) \big\rangle_{S} \notag \\
    & \le \|z_\alpha - \pi_h z_\alpha\|_{2,S}\del{ \|\nabla_S (\pi_h^1 p^- - p^-) \|_{2,S} + \|\nabla_S p^- - \langle \nabla_S p^- \rangle_{S\cap\{p^-=0\}}\|_{2,S} }\|x - x_S\|_{\infty,S} \notag \\
    & \le 2h_S \|z_\alpha - \pi_h z_\alpha\|_{2,S} \del{ \|\nabla_S (\pi_h^1 p^- -p^-)\|_{2,S} + \|\nabla_S  p^- - \pi_h \nabla_S p^-)\|_{2,S} }\\
&\lesssim h_S^{s+t_\alpha+\frac{1}{2}- \varepsilon} | \bm u \cdot \bm n - \alpha|_{H^{t_\alpha}(S)} |p^-|_{H^{s+\frac{1}{2}- \varepsilon}(S)}. \label{eq:p^-=0_dominant_bnd}
\end{align}

\medskip
\noindent \textbf{Case (2):} \,$|S\cap\{p^->0\}|\ge\tfrac12|S|$ and therefore $|S\cap\{p^+=0\}|\ge\tfrac12|S|$. 

\smallskip
\noindent
As in the previous case, since $z_\alpha p^- = 0$ a.e. on $\Gamma_C$, complementarity forces $z_\alpha = 0$ on $|S\cap\{p^->0\}|$. Using this fact, \eqref{eq:tangential_p-_vanish}, H\"older's inequality, \eqref{eq:avg_bound_on_S}, the estimate~\eqref{eq:proj-approx-S} for $\pi_h$ (with $r=t_\alpha$), and the estimate~\eqref{eq:proj-approx-S1-grad} for $\pi_h^1$ (with $r=s+\tfrac12-\varepsilon$), we find
\begin{align} \label{eq:p^->0_dominant_bnd}
\langle z_\alpha , \nabla_S &\pi_h^1 p^- \cdot  (x - x_S) \rangle_{S}
\lesssim  h_S^{s+t_\alpha+\frac{1}{2}- \varepsilon} | \bm u \cdot \bm n - \alpha |_{H^{t_\alpha}(S)} |p^-|_{H^{s+\frac{1}{2}- \varepsilon}(S)}.
\end{align}
Moreover, using the fact that $\nabla_S \pi_h^1 p^+ \cdot  (x - x_S)$ has vanishing mean, that $\langle \nabla_S p^+ \rangle_{S\cap\{p^+=0\}} = 0$, H\"older's inequality, \eqref{eq:avg_bound_on_S}, the estimate~\eqref{eq:proj-approx-S} for $\pi_h$ (with $r=t_\beta$), the estimate~\eqref{eq:proj-approx-S1-grad} for $\pi_h^1$ (with $r=s+\tfrac12-\varepsilon$), and~\eqref{eq:proj-approx-S-grad} (with $r=s-\tfrac12-\varepsilon$), we find
\begin{align} 
    \langle z_\beta , \nabla_S &\pi_h^1 p^+ \cdot  (x - x_S) \rangle_{S} 
\lesssim h_S^{s+t_\beta+\frac{1}{2}- \varepsilon} | \beta - \bm u \cdot \bm n |_{H^{t_\beta}(S)} |p^+|_{H^{s+\frac{1}{2}- \varepsilon}(S)}.
   \label{eq:p^+=0_dominant_bnd}
\end{align}

\medskip
\noindent \textbf{Case (3):} Both $|S\cap\{p^+=0\}|\ge\tfrac12|S|$ and $|S\cap\{p^-=0\}|\ge\tfrac12|S|$. 

\smallskip
\noindent
In this case, we can argue identically as in the derivations of \eqref{eq:p^-=0_dominant_bnd} and \eqref{eq:p^+=0_dominant_bnd} to find
\begin{align} 
    \langle z_\alpha , \nabla_S &\pi_h^1 p^- \cdot  (x - x_S) \rangle_{S} 
    \lesssim h_S^{s+t_\alpha+\frac{1}{2}- \varepsilon} | \bm u \cdot \bm n - \alpha |_{H^{t_\alpha}(S)} |p^-|_{H^{s+\frac{1}{2}- \varepsilon}(S)}, 
    \label{eq:p^-=0_dominant_bnd.2} \\
    \langle z_\beta , \nabla_S &\pi_h^1 p^+ \cdot  (x - x_S) \rangle_{S} 
    \lesssim h_S^{s+t_\beta+\frac{1}{2}- \varepsilon} | \beta - \bm u \cdot \bm n |_{H^{t_\beta}(S)} |p^+|_{H^{s+\frac{1}{2}- \varepsilon}(S)}.
  \label{eq:p^+=0_dominant_bnd.2}
\end{align}

\smallskip
\noindent 
Combining \eqref{eq:p^+>0_dominant_bnd}--\eqref{eq:p^+=0_dominant_bnd.2}, summing over all $S \in \mathcal{S}_h^C$, and applying Young's inequality, we have
\begin{align} \label{eq:Ihbeta2_Ihalpha2_sum}
    I_{h}^{\beta,2} + I_h^{\alpha,2} \lesssim  
    h^{s+\textup{min}(t_\beta,t_\alpha)+\frac{1}{2}- \varepsilon} \del[2]{|\beta - \bm u \cdot \bm n|_{H^{t_\beta}(\mathcal{S}_h^C)}^2 + | \bm u \cdot \bm n - \alpha|_{H^{t_\alpha}(\mathcal{S}_h^C)}^2 + |p|_{H^{s+\frac{1}{2}}(\Gamma_C)}^2}.
\end{align}
Collecting the bounds \eqref{eq:Ih1_rate}, \eqref{eq:Ihbeta1_Ihalpha1_sum}, and \eqref{eq:Ihbeta2_Ihalpha2_sum} yields \eqref{eq:convergence_regularity}.
\end{proof}

\begin{remark}[Natural error norms]
\label{rem:disc-apriori-norms}
The estimates for $\rho_{\mathrm{tot},h}^2$ also give the corresponding
bounds for the natural discrete error norms. Indeed,
\begin{align*}
&\bigl\|\bm K_h^{\frac12}\nabla_h(\Pi_h^{\mathrm{cr}}p-p_h^{cr})\bigr\|_{2,\Omega}^{2}
 +\bigl\|\bm K_h^{-\frac12}\Pi_h(\Pi_h^{\mathrm{rt}}\bm u-\bm u_h^{rt})\bigr\|_{2,\Omega}^{2}
\lesssim
\rho_{\mathrm{tot},h}^{2}(\Pi_h^{\mathrm{cr}}p,\Pi_h^{\mathrm{rt}}\bm u).
\end{align*}
By the interpolation estimates for $\Pi_h^{\mathrm{cr}}$ and
$\Pi_h^{\mathrm{rt}}$, the triangle inequality further yields
\begin{align} \label{eq:true_error}
&\|\nabla p-\nabla_h p_h^{cr}\|_{2,\Omega}^{2}
 +\|\Pi_h\bm u-\Pi_h\bm u_h^{rt}\|_{2,\Omega}^{2}
\lesssim
\rho_{\mathrm{tot},h}^{2}(\Pi_h^{\mathrm{cr}}p,\Pi_h^{\mathrm{rt}}\bm u)+h^{2s}.
\end{align}
Consequently, passing to the limit as $h\to 0^+$ in \eqref{eq:true_error} shows that the left-hand side converges to zero in case~{\rm(i)} of
\Cref{thm:disc-apriori}. In case~{\rm(ii)}, explicit rates can be derived: for $\varepsilon = 0$ when $s < 1$ and every $\varepsilon>0$ when $s = 1$,
\begin{align*}
\|\nabla p-\nabla_h p_h^{cr}\|_{2,\Omega}^{2}
 +\|\Pi_h\bm u-\Pi_h\bm u_h^{rt}\|_{2,\Omega}^{2}
&\lesssim
h^{s+\min\{t_\alpha,t_\beta\}+\frac12-\varepsilon},
\end{align*}
and this becomes $h^{2s-\varepsilon}$ when
$t_\alpha=t_\beta=s-\frac12$.
\end{remark}

\begin{remark}
The question of whether explicit error rates can be deduced for $0<s\le\tfrac{1}{2}$ is open. The obstruction is that explicit error rates for the Raviart--Thomas interpolant require $\bm u \in (H^s(\Omega))^d$, $s>\tfrac{1}{2}$. Note that there are quasi-interpolation and projection operators into the
Raviart--Thomas space that yield approximation bounds under weaker
regularity assumptions; see, for instance,
\cite{ErnGuermond2017,ErnGudiSmearsVohralik2022}. However, they do not satisfy \eqref{eq:RT-comm-trace} and thus may not deliver an admissible vector field in $K_h^\star$.
\end{remark}

\section{The numerical algorithm}\label{Numerical-algorithm}
This section addresses the numerical solution of the discrete dual
problem~\eqref{eq:inner_max_unconstrained} and the recovery of the discrete
primal solution $p_h^{cr}\in\mathcal{S}^{1,cr}(\mathcal{T}_h)$ from the dual data.
In \Cref{sec:disc-KKT} we characterize the dual solution by a KKT system with
facetwise multipliers. In \Cref{sec:invMarini} we show that a discrete primal
solution is obtained from the KKT data by an explicit elementwise
postprocessing, generalizing the classical inverse Marini
formula~\cite{Marini85,BartelsKaltenbach24}. Notably, no further linear solve is
required to produce a primal approximation. In \Cref{sec:semismoothNewton} we formulate the semismooth Newton (interpreted as a primal--dual active-set method, see,  e.g. ~\cite{HintermuellerItoKunisch2003}) used to
solve the KKT system, with the discrete primal--dual gap~\eqref{eq:disc-gap}
serving as stopping criterion.

\subsection{The discrete KKT system}\label{sec:disc-KKT}
The discrete dual problem~\eqref{eq:inner_max_unconstrained} is a
finite-dimensional concave maximization problem over the polyhedral
set~$K_h^\star$ of~\eqref{eq:disc-Kh}: the objective is quadratic, and the
divergence and flux constraints are affine in the degrees of
freedom of $\mathcal{R}T^0(\mathcal{T}_h)$. Therefore, the KKT conditions are both necessary and sufficient for optimality (see, e.g., \cite[Ch.~26]{bauschke2017convex}). Introducing a multiplier
$\overline{p}_h^{rt}\in\mathcal{L}^0(\mathcal{T}_h)$ for the divergence
constraint and multipliers
$\lambda_h^a,\lambda_h^b\in\mathcal{L}^0(\mathcal{S}_h^C)$ for the lower and
upper flux bounds yields:

\begin{theorem}[Discrete KKT system]\label{thm:disc-KKT}
Assume $K_h^\star\ne\emptyset$. Then there exists a tuple
$(\bm u_h^{rt},\overline p_h^{rt},\lambda_h^{a},\lambda_h^{b})\in
 \mathcal{R}T^0(\mathcal{T}_h)\times\mathcal{L}^0(\mathcal{T}_h)\times\bigl(\mathcal{L}^0(\mathcal{S}_h^C)\bigr)^2$
such that for every $(\bm v_h,\overline q_h)^\top\in\mathcal{R}T^0(\mathcal{T}_h)\times\mathcal{L}^0(\mathcal{T}_h)$,
\begin{subequations}\label{eq:disc-KKT}
\begin{alignat}{2}
(\bm K_h^{-1}\Pi_h\bm u_h^{rt},\Pi_h\bm v_h)_\Omega-(\overline p_h^{rt},\mathrm{div}\,\bm v_h)_\Omega+\langle\lambda_h^a-\lambda_h^b,\bm v_h\cdot\bm n\rangle_{\Gamma_C}&=-\langle p_D^h,\bm v_h\cdot\bm n\rangle_{\Gamma_D},\label{eq:disc-KKT-a}\\
(\mathrm{div}\,\bm u_h^{rt},\overline q_h)_\Omega&=(f_h,\overline q_h)_\Omega,\label{eq:disc-KKT-b}\\
\alpha_h\le\bm u_h^{rt}\cdot\bm n&\le\beta_h\quad\text{on }\mathcal{S}_h^C,\label{eq:disc-KKT-c}\\
\lambda_h^a,\lambda_h^b&\le 0\quad\text{on }\mathcal{S}_h^C,\label{eq:disc-KKT-d}\\
\lambda_h^a(\bm u_h^{rt}\cdot\bm n-\alpha_h)=\lambda_h^b(\beta_h-\bm u_h^{rt}\cdot\bm n)&=0\quad\text{on }\mathcal{S}_h^C.\label{eq:disc-KKT-e}
\end{alignat}
\end{subequations}
Moreover, $\bm u_h^{rt}$ is the unique solution of~\eqref{eq:inner_max_unconstrained}. 
\end{theorem}

\subsection{Generalized inverse Marini formula}\label{sec:invMarini}
%
The discrete KKT system~\eqref{eq:disc-KKT} produces a piecewise
constant pressure approximation $\overline p_h^{rt}\in\mathcal L^0(\mathcal T_h)$.
We now show that the discrete primal solution
$p_h^{cr}\in\mathcal{S}^{1,cr}(\mathcal{T}_h)$ of
Theorem~\ref{thm:disc-strong-duality} can be recovered \emph{directly},
 by an inexpensive explicit elementwise post-processing of the pair $(\overline{p}_h^{rt},\bm{u}_h^{rt})$. This is the content of the following lemma, which
generalizes the classical inverse Marini formula~\cite{Marini85}:
\begin{lemma}[Generalized inverse Marini formula]\label{lem:marini}
Let $(\bm u_h^{rt},\overline{p}_h^{rt},\lambda_h^a,\lambda_h^b)^\top  \in  \mathcal{R}T^0(\mathcal{T}_h) \times \mathcal{L}^0(\mathcal{T}_h) \times (\mathcal{L}^0(\mathcal{S}_h^C))^2$ be such that the KKT conditions of~\eqref{eq:disc-KKT}~are~satisfied.~Then,
a discrete primal solution $\widehat{p}_h \in K_h$ is available via the following generalized inverse Marini formula:
\begin{align} \label{eq:inverse_marini}
   \widehat{p}_h = \overline{p}_h^{rt} -\bm K_h^{-1} \Pi_h \bm u_h^{rt} \cdot (\textup{id}_{\mathbb{R}^d} - \Pi_h  \textup{id}_{\mathbb{R}^d}).
    \end{align}
\end{lemma}
\begin{proof}
Throughout, let $p_h^{cr} \in K_h$ denote any minimizer of the discrete primal energy \eqref{eq:discrete_predual}. 
Define $\widehat{p}_h := \overline{p}_h^{rt} - \bm{K}_h^{-1} \smash{\Pi_h \bm u_h^{rt}} \cdot (\textup{id}_{\mathbb{R}^d} - \Pi_h \textup{id}_{\mathbb{R}^d})\in \mathcal{L}^1(\mathcal{T}_h)$. We first establish that $\widehat{p}_h\in \mathcal{S}^{1,cr}(\mathcal{T}_h)$. Observe that $\nabla_h \widehat{p}_h=  - \bm{K}_h^{-1}\Pi_h \bm u_h^{rt}$ and $\Pi_h \widehat{p}_h=\overline{p}_h^{rt}$ a.e. in $\Omega$. 
By construction, $ \nabla_h (\widehat{p}_h-p_h^{cr}) = 0
$ and thus $\widehat{p}_h-p_h^{cr}\in \mathcal{L}^0(\mathcal{T}_h)$.  Note that if $\Gamma_D \ne \emptyset$, $\widehat{p}_h-p_h^{cr} = 0$ by the discrete Poincar\'{e} inequality and we are finished. Therefore, suppose $\Gamma_D = \emptyset$. From the discrete integration-by-parts formula \eqref{eq:discrete-IBP} and \eqref{eq:disc-KKT-a},  
it follows that
    \begin{align*}
        (\widehat{p}_h-p_h^{cr},\textup{div}\,\bm v_h)_{\Omega}
        =(\overline{p}_h^{rt},\textup{div}\,\bm v_h)_{\Omega}-(\bm K_h^{-1}\Pi_h u_h^{rt},\Pi_h \bm v_h)_{\Omega} 
        =0, \quad \forall \bm v_h \in \mathcal{R}T^0_0(\mathcal{T}_h),
     \end{align*}
   \textit{i.e.}, $\widehat{p}_h-p_h^{cr} \perp_{L^2} \textup{div}\,(\mathcal{R}T^0_0(\mathcal{T}_h)) = \mathcal{L}^0(\mathcal{T}_h)/\mathbb{R}$. Consequently, $\widehat{p}_h-p_h^{cr} = \textup{const}.$~and,~thus,~${\widehat{p}_h \in  \mathcal{S}^{1,cr}(\mathcal{T}_h)}$.
   
 We must establish that, in fact, $I_h^{cr}(\widehat{p}_h) = I_h^{cr}(p_h^{cr})$.
To this end, integrating by parts in \eqref{eq:disc-KKT-a} against an arbitrary function $\bm v_h \in \mathcal{R}T^0(\mathcal{T}_h)$ yields
\begin{align*} 
\langle\widehat{p}_h, \bm v_h\cdot\bm n\rangle_{\Gamma_C}
&= 
\langle \lambda_h^a - \lambda_h^b, \bm v_h \cdot \bm n \rangle_{\Gamma_C},
\end{align*}
from which it can easily be deduced that
\begin{subequations}\label{eq:marini.3}
\begin{alignat}{2}
    \pi_h \widehat{p}_h&=\lambda_h^a-\lambda_h^b&&\quad \text{ a.e.\ on }\Gamma_C\,.\label{eq:marini.3.2}
\end{alignat}
\end{subequations}
By \Cref{lem:disc-convexity}, \eqref{eq:disc_optimality_1}, and \eqref{eq:marini.3.2}, 
\begin{align*}
    I_h^{cr} (\widehat{p}_h)-I_h^{cr}(p_h^{cr}) 
    &=- \langle \bm u_h^{rt} \cdot \bm n, \widehat{p}_h\rangle_{\Gamma_C}+ \tfrac{1}{2}\langle \alpha_h+\beta_h, \lambda_h^a-\lambda_h^b\rangle_{\Gamma_C} +  \tfrac{1}{2}\|(\beta_h - \alpha_h) ( \lambda_h^a-\lambda_h^b)\|_{1,\Gamma_C},
\end{align*}
from which it follows that
\begin{align} \label{eq:energy_diff_sup}
I_h^{cr} &(\widehat{p}_h)-I_h^{cr}(p_h^{cr}) \\ &=\sup_{\substack{ \bm v_h \in \mathcal{R}T^0(\mathcal{T}_h) \\ \|\bm v_h \cdot \bm n\|_{\infty, \Gamma_C} \le 1}}\cbr{ \tfrac{1}{2} \big\langle (\beta_h-\alpha_h)(\lambda_h^a-\lambda_h^b) , \bm v_h \cdot \bm n \big\rangle_{\Gamma_C}- \langle \bm u_h^{rt} \cdot \bm n, \lambda_h^a-\lambda_h^b \rangle_{\Gamma_C}
    +\tfrac{1}{2} \langle \alpha_h+\beta_h,\lambda_h^a-\lambda_h^b\rangle_{\Gamma_C}}. \notag
\end{align}
For notational brevity, let us define the following affine functional $\psi_S: \mathbb{P}^0(S) \to \mathbb{R}$:
\begin{align*}
  \psi_S(q_S):=  \tfrac{1}{2} \big\langle (\beta_h-\alpha_h)(\lambda_h^a-\lambda_h^b) ,   q_S \big\rangle_{S}- \langle \bm u_h^{rt} \cdot \bm n, \lambda_h^a-\lambda_h^b \rangle_{S}
    +\tfrac{1}{2} \langle \alpha_h+\beta_h,\lambda_h^a-\lambda_h^b\rangle_{S}.
\end{align*}
Since  $\gamma_{\bm n} : \mathcal{R}T^0(\mathcal{T}_h) \to \mathcal{L}_h^0(\mathcal{S}_h^C)$ is surjective and the inner product over $\Gamma_C$ decomposes additively over faces $S \in \mathcal{S}_h^C$, we can equivalently write \eqref{eq:energy_diff_sup} as
\begin{align} \label{eq:energy_diff_sup_sum}
   I_h^{cr} (\widehat{p}_h)-I_h^{cr}(p_h^{cr}) = \sum_{S \in \mathcal{S}_h^C} \,\sup_{\substack{q_S \in \mathbb{P}^0(S) \\ |q_S| \le 1} } \psi_S(q_S).
\end{align}
To conclude, we must show that each supremum in the right hand side of \eqref{eq:energy_diff_sup_sum} is zero.
We consider four cases depending on whether the constraints are active or inactive on a given face $S \in \mathcal{S}_h^C$: 

\medskip

\emph{$\bullet$ Case 1:} If $\alpha_h<\bm u_h^{rt}\cdot \bm n <\beta_h$ on $S$, the complementarity condition \eqref{eq:disc-KKT-e} yields $\lambda_h^a=\lambda_h^b=0$ and, thus, $\psi_S(q_S) = 0$ for all $q_S \in \mathbb{P}^0(S)$.

\emph{$\bullet$ Case 2:} If $\alpha_h=\bm u_h^{rt}\cdot \bm n =\beta_h$ on a given face $S$, then trivially $ \psi_S(q_S) = 0$ for all $q_S\in \mathbb{P}^0(S)$.

\emph{$\bullet$ Case 3:} If $\alpha_h<\bm u_h^{rt}\cdot \bm n =\beta_h$ on a given face $S$,  the complementarity condition \eqref{eq:disc-KKT-e} yields $\lambda_h^a =0$ and $\lambda_h^b \le 0$. Thus,  $\psi_S(q_S) = \tfrac{1}{2} \big\langle (\beta_h-\alpha_h)\lambda_h^b , 1 - q_S\rangle_{S}$  for all $q_S\in \mathbb{P}^0(S)$.

\emph{$\bullet$ Case 4:} If $\alpha_h=\bm u_h^{rt}\cdot \bm n <\beta_h$ on a given face $S$, the complementarity condition \eqref{eq:disc-KKT-e} yields $\lambda_h^a \le 0$ and $\lambda_h^b = 0$. Thus,  $\psi_S(q_S)= \tfrac{1}{2} \langle (\beta_h-\alpha_h)\lambda_h^a ,  q_S + 1 \rangle_{S}$  for all $q_S\in \mathbb{P}^0(S)$.
\medskip
\noindent 
The supremum is trivially zero in the first two cases. In the third case, since $\beta_h - \alpha_h > 0$ and $\lambda^b \le 0$, the supremum is zero and is attained by $q_S = 1$. In the fourth case, since $\beta_h - \alpha_h > 0$ and $\lambda^a \le 0$, the supremum is zero and is attained by $q_S = -1$.
Consequently, $I_h^{cr}(\widehat{p}_h) - I_h^{cr}(p_h^{cr}) = 0$,
and the result follows.
\end{proof}

\subsection{Semismooth Newton method for the KKT system}\label{sec:semismoothNewton}

For any fixed $c>0$, the
conditions~\eqref{eq:disc-KKT-c}--\eqref{eq:disc-KKT-e} are, facet by facet,
equivalent to the two nonsmooth equations
\begin{align}\label{eq:ncp}
   \lambda_h^a=\min\bigl(0,\lambda_h^a+c\,(\bm u_h^{rt}\cdot\bm n-\alpha_h)\bigr),
   \qquad
   \lambda_h^b=\min\bigl(0,\lambda_h^b+c\,(\beta_h-\bm u_h^{rt}\cdot\bm n)\bigr).
\end{align}
 The
system~\eqref{eq:disc-KKT-a}--\eqref{eq:disc-KKT-b}, \eqref{eq:ncp} is a
finite-dimensional nonsmooth equation to which, since the $\min$-function is
Newton (slantly) differentiable, we apply a semismooth Newton method. The
resulting iteration coincides with the primal--dual active set
strategy and converges locally
superlinearly~\cite{HintermuellerItoKunisch2003}. 

\subsubsection{Algebraic form}
Recall from \Cref{sec:preliminaries} the Raviart--Thomas basis
$\{\bm\psi_S\}_{S\in\mathcal{S}_h}$, characterized by
$\bm\psi_S|_{S'}\cdot\bm n_{S'}=\delta_{S,S'}$, so that every
$\bm u_h^{rt}\in\mathcal{R}T^0(\mathcal{T}_h)$ expands as
$\bm u_h^{rt}=\sum_{S\in\mathcal{S}_h}U_S\bm\psi_S$ with
\begin{equation}\label{eq:rt_DOF}
   U_S = (\bm u_h^{rt}\cdot\bm n_S)|_S,\qquad S\in\mathcal{S}_h.
\end{equation}
Set
$N_{rt}:=\dim\mathcal{R}T^0(\mathcal{T}_h)$,
$N_h^0:=\dim\mathcal{L}^0(\mathcal{T}_h)$, and
$N_h^X:=\mathrm{card}(\mathcal{S}_h^X)$ for $X\in\{C,D\}$, fix an ordering
$\{T_i\}_{i=1,\dots,N_h^0}$ of the elements and an ordering of the facets in
which those of $\mathcal{S}_h^C$ and $\mathcal{S}_h^D$ are enumerated first,
and let $\bm T_h^X\in\{0,1\}^{N_h^X\times N_{rt}}$ denote the Boolean
restriction matrix selecting the degrees of freedom on $\mathcal{S}_h^X$. The
matrix representations of the bilinear forms in~\eqref{eq:disc-KKT} read
\begin{align*}
    \bm A_h   &:= \bigl((\bm K_h^{-1}\Pi_h\bm\psi_{S_i},\Pi_h\bm\psi_{S_j})_\Omega\bigr)_{i,j=1,\dots,N_{rt}} \in \mathbb{R}^{N_{rt}\times N_{rt}}, \\
    \bm B_h   &:= \bigl((\mathrm{div}\,\bm\psi_{S_j},\chi_{T_i})_\Omega\bigr)_{i=1,\dots,N_h^0,\,j=1,\dots,N_{rt}} \in \mathbb{R}^{N_h^0\times N_{rt}}, \\
    \bm M_h^C &:= \bigl((\chi_{S_i},\chi_{S_j})_{\Gamma_C}\bigr)_{i,j}
    =\mathrm{diag}(|S_i|)\in \mathbb{R}^{N_h^C\times N_h^C},
\end{align*}
and the vector representations of the data,
\begin{align*}
    \bm F_h &:= \bigl((f_h,\chi_{T_i})_\Omega\bigr)_{i=1,\dots,N_h^0} \in \mathbb{R}^{N_h^0}, \qquad
    \bm G_h  := -(\bm T_h^{D})^\top\bm M_h^{D}\,\bm p_D^h \in \mathbb{R}^{N_{rt}},
\end{align*}
where $\bm M_h^D$ and $\bm p_D^h$ denote the Dirichlet-facet mass matrix and
the vector of facet averages of the Dirichlet datum on $\mathcal{S}_h^D$.
Then~\eqref{eq:disc-KKT} is equivalent to seeking
$(\bm U,\overline{\bm P},\bm\Lambda^a,\bm\Lambda^b)^\top\in\mathbb{R}^{N_{rt}}\times\mathbb{R}^{N_h^0}\times(\mathbb{R}^{N_h^C})^2$ such that
\begin{subequations}\label{eq:alg-KKT}
\begin{alignat}{2}
\bm A_h\bm U - \bm B_h^\top\overline{\bm P} + (\bm T_h^{C})^\top\bm M_h^{C}(\bm\Lambda^a-\bm\Lambda^b) &= \bm G_h,\label{eq:alg-KKT-a}\\
\bm B_h\bm U &= \bm F_h,\label{eq:alg-KKT-b}\\
\bm\alpha_h\le\bm T_h^{C}\bm U\le\bm\beta_h,\qquad\bm\Lambda^a,\bm\Lambda^b&\le\bm 0,\label{eq:alg-KKT-c}\\
\bm\Lambda^a\odot(\bm T_h^{C}\bm U-\bm\alpha_h)=\bm\Lambda^b\odot(\bm\beta_h-\bm T_h^{C}\bm U) &= \bm 0,\label{eq:alg-KKT-d}
\end{alignat}
\end{subequations}
where $\bm\alpha_h,\bm\beta_h\in\mathbb{R}^{N_h^C}$ collect the facet values
of $\alpha_h,\beta_h$ and $\odot$ denotes the Hadamard product.

\subsubsection{Semismooth Newton scheme}
Applying a semismooth Newton step to the reformulation~\eqref{eq:ncp}
of~\eqref{eq:alg-KKT-c}--\eqref{eq:alg-KKT-d} yields the following iteration.

\begin{algorithm}[Semismooth Newton method]\label{alg:ssn}
Let $c>0$, $\varepsilon_{\mathrm{STOP}}>0$, and an initial vector
$(\bm U,\overline{\bm P},\bm\Lambda^a,\bm\Lambda^b)^0$ be given. Set
$\bm E_h^C:=(\bm T_h^C)^\top\bm M_h^C$. For $k=0,1,2,\ldots$, repeat:
\begin{enumerate}[(i),leftmargin=2.4em,itemsep=.35em,topsep=.35em]
\item Determine the predicted active sets
\begin{align*}
\mathcal A_a^k
  &:=
  \bigl\{i:\bigl((\bm\Lambda^a)^k
       +c(\bm T_h^C\bm U^k-\bm\alpha_h)\bigr)_i<0\bigr\},\\
\mathcal A_b^k
  &:=
  \bigl\{i:\bigl((\bm\Lambda^b)^k
       +c(\bm\beta_h-\bm T_h^C\bm U^k)\bigr)_i<0\bigr\}.
\end{align*}
Let $\bm 1_{\mathcal A_a^k}$ and $\bm 1_{\mathcal A_b^k}$ denote the
corresponding diagonal indicator matrices, and set
$\bm 1_{(\mathcal A_a^k)^c}:=\bm I-\bm 1_{\mathcal A_a^k}$ and
$\bm 1_{(\mathcal A_b^k)^c}:=\bm I-\bm 1_{\mathcal A_b^k}$.

\item With
$\bm T_a^k:=\bm 1_{\mathcal A_a^k}\bm T_h^C$ and
$\bm T_b^k:=\bm 1_{\mathcal A_b^k}\bm T_h^C$, solve
\begin{equation*}
\begin{bmatrix}
\bm A_h & -\bm B_h^\top & \bm E_h^C & -\bm E_h^C\\
\bm B_h & \bm 0 & \bm 0 & \bm 0\\
\bm T_a^k & \bm 0 & \bm 1_{(\mathcal A_a^k)^c} & \bm 0\\
\bm T_b^k & \bm 0 & \bm 0 & \bm 1_{(\mathcal A_b^k)^c}
\end{bmatrix}
\begin{bmatrix}
\bm U^{k+1}\\
\overline{\bm P}^{k+1}\\
(\bm\Lambda^a)^{k+1}\\
(\bm\Lambda^b)^{k+1}
\end{bmatrix}
=
\begin{bmatrix}
\bm G_h\\
\bm F_h\\
\bm 1_{\mathcal A_a^k}\bm\alpha_h\\
\bm 1_{\mathcal A_b^k}\bm\beta_h
\end{bmatrix}.
\end{equation*}

\item Recover $\widehat p_h^{\,k+1}\in\mathcal S^{1,cr}(\mathcal T_h)$
from $(\bm u_h^{rt,k+1},\overline p_h^{rt,k+1})$ by
\eqref{eq:inverse_marini}. Stop if
\begin{equation*}
   \eta_{\mathrm{gap},h}^2(\widehat p_h^{\,k+1},\bm u_h^{rt,k+1})
   :=
   I_h^{cr}(\widehat p_h^{\,k+1})
   -D_h^{rt}(\bm u_h^{rt,k+1})
   \le \varepsilon_{\mathrm{STOP}}.
\end{equation*}
Otherwise, continue with the next value of $k$.
\end{enumerate}
\end{algorithm}

\begin{remark}[Cost and stopping criterion]\label{rem:ssn-stop}
Each semismooth Newton step solves one saddle-point system with the active
facet set fixed. The stopping criterion is not the residual of this linear
system, but the discrete primal--dual gap. Indeed, for every admissible pair,
\Cref{thm:disc-aposteriori-id} identifies $\eta_{\mathrm{gap},h}^{2}$ with
the squared discrete energy error. Once the active sets stabilize,
the iterate satisfies the discrete KKT system~\eqref{eq:disc-KKT}, and the
reconstructed pressure coincides with the discrete primal solution,
$\widehat p_h^{,k+1}=p_h^{cr}$.
\end{remark}

\section{Numerical results}\label{sec:Numerical-results}

In this section, we confirm the theoretical findings of the preceding sections
via numerical experiments: a manufactured-solution study of the \emph{a priori}
error estimates of \Cref{thm:disc-apriori}, an adaptive study based on
the \emph{a posteriori} error identity of \Cref{thm:aposteriori-id}, and
an application to miscible displacement in the SPE10 benchmark reservoir.
The discrete KKT
system~\eqref{eq:disc-KKT} is solved for the flux $\bm u_h^{rt} \in \mathcal{R}T^0(\mathcal{T}_h)$ using the semismooth Newton iteration of \Cref{alg:ssn},
and the pressure
$p_h^{cr} \in \mathcal{S}^{1,cr}(\mathcal{T}_h)$ is recovered from the generalized inverse Marini
formula~\eqref{eq:inverse_marini} in an explicit fashion. 
Linear systems are solved by sparse direct factorization (UMFPACK); the
semismooth Newton iteration is warm-started with the active set of the previous
mesh and, in the time-dependent example, of the previous time step. 
All  computations  use the finite  element  library  \texttt{NETGEN}/\texttt{NGSolve}  (version  v6.2.2602,  \textit{cf}.\  \cite{netgen}/\cite{ngsolve});  all  graphics  use \texttt{Matplotlib} (version 3.10.8, \textit{cf}.\ \cite{Hunter07}) or \texttt{PyVista} (version 0.48.4, \textit{cf}.\ \cite{pyvista}).

\subsection{Numerical example concerning the a priori analysis}

The first experiment is a manufactured-solution test on the L-shaped domain
$\Omega=(-1,1)^2\setminus\big([0,1]\times[-1,0]\big)$ with $\bm K=\bm I$,
$\Gamma_D:=\big([0,1]\times\{0\}\big)\cup\big(\{0\}\times[-1,0]\big)$, and
$\Gamma_C:=\partial\Omega\setminus\overline{\Gamma_D}$. In polar coordinates
$(r,\theta)$ at the re-entrant corner, $0\le\theta\le 3\pi/2$ measured from
the positive $x$-axis, we prescribe the corner singularity
$p=C\,r^{\lambda}\cos(\lambda\theta)$ with $\lambda=\tfrac35$, $C=0.0468$,
and set
$\bm u=-\nabla p
=-C\lambda\,r^{\lambda-1}\big(\cos((1-\lambda)\theta),\sin((1-\lambda)\theta)\big)^{\!\top}$,
$f=0$, and $p_D=p|_{\Gamma_D}$. With $g:=\bm u\cdot\bm n$ on $\Gamma_C$ and
$c:=0.3\,\|g\|_{L^\infty(\Gamma_C)}$, the bounds
$\alpha:=g-c\,\mathbf 1_{\{p\ge0\}}$ and $\beta:=g+c\,\mathbf 1_{\{p\le0\}}$
satisfy $\alpha\le g\le\beta$ and \eqref{eq:cts_optimality_2} by construction,
with the upper (resp.\ lower) bound active on $\{p>0\}\cap\Gamma_C$ (resp.\
$\{p<0\}\cap\Gamma_C$), both of positive length. Since
$|\bm u|=C\lambda\,r^{\lambda-1}$ and $\bm u$ is smooth away from the corner,
$p\in H^{1+s}(\Omega)$ and $\bm u\in(H^{s}(\Omega))^2$ for every $s<\lambda$
but not for $s=\lambda$, and $\bm u\in(L^{q}(\Omega))^2$ for every
$q<2/(1-\lambda)=5$. Apart from the single switching point on the left
boundary edge, $\alpha,\beta$ are smooth along each side of $\Gamma_C$, hence
belong to $H^{t}(\mathcal S_h^{C})$ for every $t<\tfrac12$; taking
$t_\alpha=t_\beta=s-\tfrac12$ in \Cref{thm:disc-apriori}(ii) for every
$s\in(\tfrac12,\lambda)$ predicts
$\rho^2_{\textup{tot},h}
=\mathcal O(h^{2\lambda})
\simeq\mathcal O(N_k^{-\lambda})$.
 Table~\ref{tab:test1-eoc} and Figure~\ref{fig:test1-conv} report the results
under uniform refinement: the discrete gap
$\eta^2_{\mathrm{gap},h}(\Pi_h^{cr}p,\Pi_h^{rt}\bm u)=\rho^2_{\textup{tot},h}$
and the squared broken $H^1$-seminorm error of the pressure attain the
predicted slope $N_k^{-\lambda}$, consistent with \Cref{thm:disc-apriori} and
the norm estimates of \Cref{rem:disc-apriori-norms}. The convergence rate in squared $L^2$-norm of the pressure approximation appears to be approaching $\mathcal{O}(N_k^{-1.4})$.

\begin{figure}[ht]
  \centering
  \includegraphics[width=0.85\linewidth]{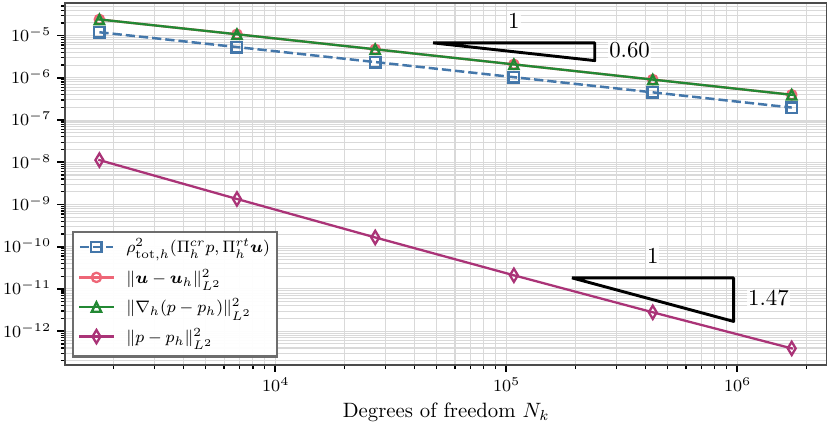}
  \caption{Convergence history under uniform refinement.}
  \label{fig:test1-conv}
\end{figure}

\begin{table}[ht]
  \centering\footnotesize
  \setlength{\tabcolsep}{4pt}
  \resizebox{\linewidth}{!}{\begin{tabular}{r r c c c c c c c c}
\toprule
$k$ & $N_k$ & $I^{cr}_h(p^{cr}_h)$ & $D^{rt}_h(\bm u^{rt}_h)$ & $\rho^2_{\mathrm{tot},h}$ & EOC & $\|\nabla_h(p-p_h)\|^2_{L^2}$ & EOC & $\|p-p_h\|^2_{L^2}$ & EOC \\
\midrule
  2 &     1732 & -1.471892e-03 & -1.471845e-03 & 1.199e-05 & --- & 2.413e-05 & --- & 1.127e-08 & --- \\
  3 &     6824 & -1.464285e-03 & -1.464271e-03 & 5.351e-06 & 0.589 & 1.075e-05 & 0.590 & 1.353e-09 & 1.546 \\
  4 &    27088 & -1.460950e-03 & -1.460946e-03 & 2.361e-06 & 0.593 & 4.741e-06 & 0.594 & 1.660e-10 & 1.522 \\
  5 &   107936 & -1.459492e-03 & -1.459490e-03 & 1.036e-06 & 0.596 & 2.079e-06 & 0.596 & 2.109e-11 & 1.493 \\
  6 &   430912 & -1.458856e-03 & -1.458855e-03 & 4.529e-07 & 0.598 & 9.091e-07 & 0.598 & 2.802e-12 & 1.458 \\
  7 &  1721984 & -1.458578e-03 & -1.458578e-03 & 1.976e-07 & 0.599 & 3.967e-07 & 0.599 & 3.916e-13 & 1.420 \\
\bottomrule
\end{tabular}
}
  \caption{Discrete primal and dual energies
  $I^{cr}_h( p_h^{cr})$ and $D^{rt}_h(\bm u_h^{rt})$; the discrete gap $\rho^2_{\mathrm{tot},h}=\eta^2_{\mathrm{gap},h}$ evaluated at
  $(\Pi_h^{cr}p,\Pi_h^{rt}\bm u)$; and the $H^1$-seminorm and $L^2$-norm errors of the pressure. 
  }
  \label{tab:test1-eoc}
\end{table}

\FloatBarrier

\subsection{Numerical example concerning the a posteriori analysis}

The second experiment tests the gap estimator as a refinement indicator. For
a pair $(q,\bm v)\in K\times K^\star$, \Cref{lem:gap-decomp} localizes
\eqref{eq:gap_estimator} into the elementwise indicators
\begin{equation}\label{eq:gap-local}
\begin{aligned}
  \eta^2_{\mathrm{gap},T}(q,\bm v)
  &:=\eta^2_{\mathrm{gap,I},T}(q,\bm v)+\eta^2_{\mathrm{gap,II},T}(q,\bm v),
  \qquad T\in\mathcal T_k,\\
  \eta^2_{\mathrm{gap,I},T}(q,\bm v)
  &:=\tfrac12\,\|\bm K^{\frac12}\nabla q+\bm K^{-\frac12}\bm v\|_{2,T}^2,\\
  \eta^2_{\mathrm{gap,II},T}(q,\bm v)
  &:=\langle \beta-\bm v\cdot\bm n,\,q^+\rangle_{\partial T\cap\Gamma_C}
   +\langle \bm v\cdot\bm n-\alpha,\,q^-\rangle_{\partial T\cap\Gamma_C},
\end{aligned}
\end{equation}
so that $\eta^2_{\mathrm{gap}}(q,\bm v)=\sum_{T\in\mathcal T_k}
\eta^2_{\mathrm{gap},T}(q,\bm v)$; both contributions are nonnegative, the
boundary one since $\alpha\le\bm v\cdot\bm n\le\beta$ for $\bm v\in K^\star$.
These indicators, evaluated at $q=\overline p_k$ and $\bm v=\bm u_k^{rt}$,
drive the adaptive loop in \Cref{alg:afem} below.

\begin{algorithm}[AFEM]\label{alg:afem}
Given $\theta\in(0,1]$, $\varepsilon_{\mathrm{STOP}}\ge0$, and a conforming
triangulation $\mathcal T_0$ of $\Omega$ resolving the partition of
$\partial\Omega$ and the jump sets of the data, iterate for $k=0,1,2,\dots$:
\begin{description}
\item[\textnormal{(\textsc{Solve})}] Solve the discrete KKT
  system~\eqref{eq:disc-KKT} on $\mathcal T_k$ by
  Algorithm~\ref{alg:ssn}, warm-started with the final active set of
  $\mathcal T_{k-1}$; recover $p_k\in\mathcal S^{1,cr}(\mathcal T_k)$ by the
  inverse Marini formula~\eqref{eq:inverse_marini} and
  $\overline p_k\in\mathcal S^{1}(\mathcal T_k)$ by nodal averaging.
\item[\textnormal{(\textsc{Estimate})}] Compute
  $\eta^2_{\mathrm{gap},T}(\overline p_k,\bm u_k^{rt})$, $T\in\mathcal T_k$;
  stop if $\eta^2_{\mathrm{gap}}:=\sum_{T\in\mathcal T_k}
  \eta^2_{\mathrm{gap},T}\le\varepsilon_{\mathrm{STOP}}$.
\item[\textnormal{(\textsc{Mark})}] Select a minimal
  $\mathcal M_k\subseteq\mathcal T_k$ with
  $\sum_{T\in\mathcal M_k}\eta^2_{\mathrm{gap},T}\ge\theta^2\,
  \eta^2_{\mathrm{gap}}$ (D\"orfler marking).
\item[\textnormal{(\textsc{Refine})}] Generate $\mathcal T_{k+1}$ by
  newest-vertex bisection of all $T\in\mathcal M_k$.
\end{description}
\end{algorithm}

\label{ss:test3}
We take $\Omega=(0,1)^2$, $f=1$, $\Gamma_D=\emptyset$, and
$\Gamma_C=\partial\Omega$. The permeability
$\bm K=\kappa\bm I$ is the Kellogg checkerboard with values
$(\kappa_{LL},\kappa_{LR},\kappa_{UR},\kappa_{UL})=(1,250,1,50)$ on the four
subsquares meeting at the cross point $(\tfrac12,\tfrac12)$, with leading
singular exponent $s\approx0.139$. The bounds $\alpha,\beta$ are piecewise
constant on $\Gamma_C$, with $8$, $16$, $8$, and $8$ subintervals on the
bottom, right, top, and left faces (\Cref{fig:test2-bounds}).
Algorithm~\ref{alg:afem} is run with $\theta=\tfrac12$ and compared against
uniform refinement. \Cref{fig:test2-conv} shows $\eta^2_{\mathrm{gap}}\simeq N_k^{-0.32}$ under
uniform refinement. Notably, this is faster than the $N_k^{-s}$ expected due to the cross-point
singularity, yet far from optimal. On the other hand, the adaptive loop restores the
optimal decay $\eta^2_{\mathrm{gap}}\simeq N_k^{-1}$.
Since the exact minimal energy is unknown, we approximate the value
$I(p)=D(\bm u)$ via Aitken's $\delta^2$-process (\emph{cf}.~\cite{Aitken1927Bernoulli}), applied
to the sequence of adaptive discrete primal energies, yielding
$I(p)\approx-3.28\times10^{-3}$; the primal energy $I(\bar p_h^{cr})$ and dual energy of $D(\bm u_h^{rt})$  of the adaptively refined
sequence quickly converge to this value, whereas the energies of the uniformly refined sequence appear to converge significantly more slowly.

\Cref{fig:test2-surf} compares the pressure solution computed on a sequence of adaptive and uniformly refined meshes. The adaptive meshes concentrate refinement at the
cross point and along the boundary portions where the active set switches
(\emph{cf}. \Cref{fig:test2-refine}). \Cref{fig:test2-tri} displays
$\bm u_h^{rt}\cdot\bm n$ and $\pi_h p_h^{cr}$ on $\Gamma_C$ at the finest
level ($N_k=312\,960$, $207\,559$ elements): in accordance with the discrete
complementarity conditions in \Cref{rem:disc_complementarity},
$\bm u_h^{rt}\cdot\bm n=\beta$ on the $503$ facets where $\pi_h p_h^{cr}>0$,
$\bm u_h^{rt}\cdot\bm n=\alpha$ on the $1352$ facets where $\pi_h p_h^{cr}<0$, and
$\pi_h p_h^{cr}$ vanishes to machine precision on the remaining $1388$ facets.

\begin{figure}[h!]
  \centering
  \includegraphics[width=0.85\linewidth]{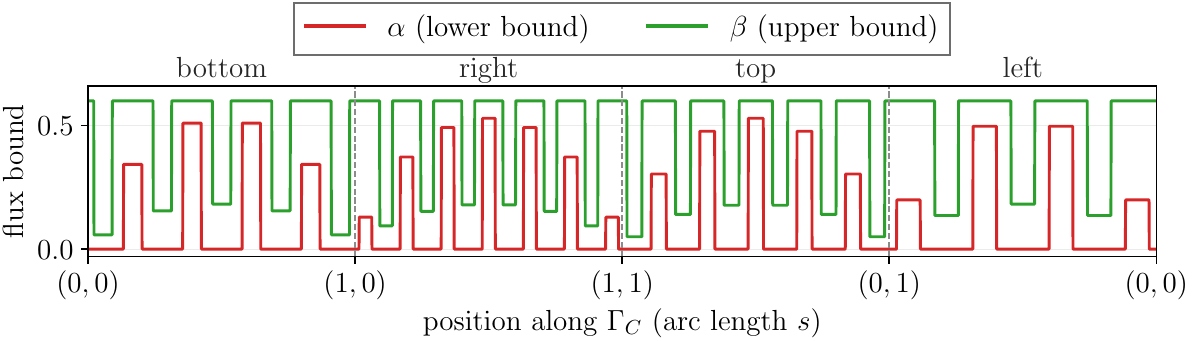}
  \caption{Flux bounds $\alpha,\beta$ along $\Gamma_C=\partial\Omega$.}
  \label{fig:test2-bounds}
\end{figure}

\begin{figure}[h!]
  \centering
  \includegraphics[width=0.9\linewidth]{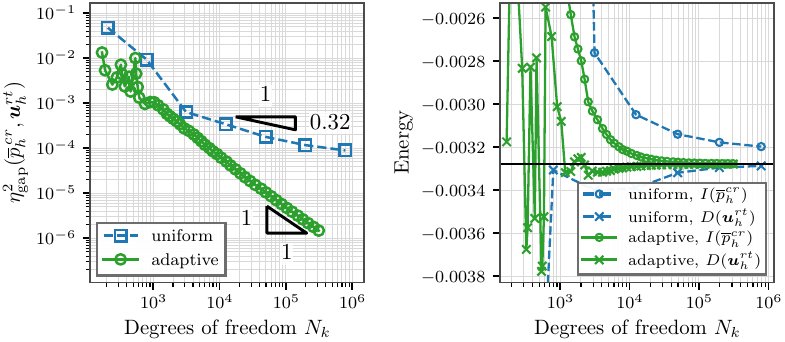}
 \caption{Gap estimator $\eta^2_{\mathrm{gap}}$ versus $N_k$
 for uniform and adaptive refinement (left). Primal energy and dual energies
 $I(\overline p_h^{cr})$ and $D(\bm u_h^{rt})$ approaching the approximation of $I(p)=D(\bm u)$
 obtained via Aitken's $\delta^2$-process (right). \label{fig:test2-conv}}
\end{figure}

\begin{figure}[h!]
  \centering
  \includegraphics[width=1\linewidth]{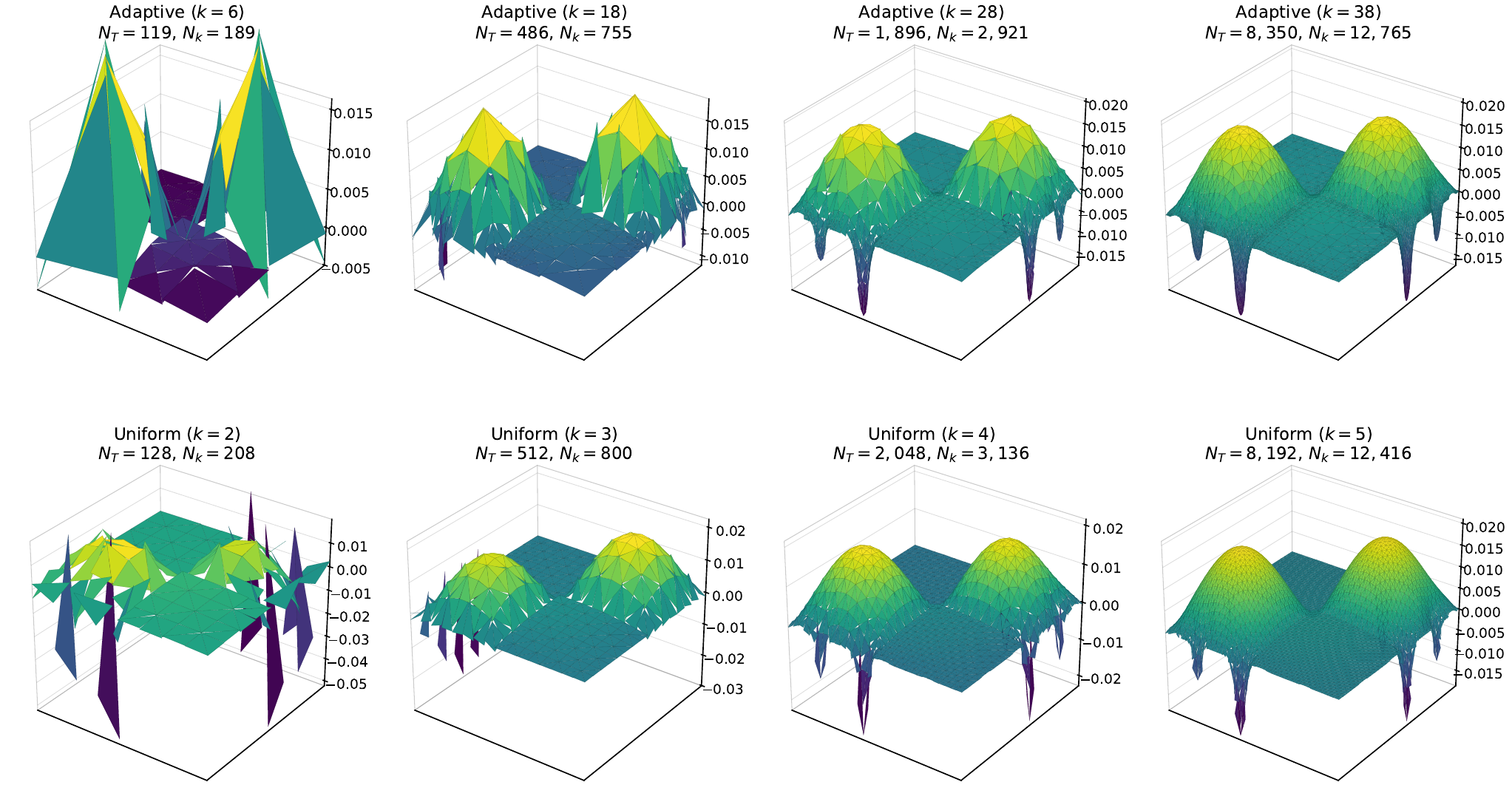}
  \caption{Crouzeix--Raviart pressure approximation $p_h^{cr}$ across various levels of refinement, both adaptive (top)
  and uniform (bottom).}
  \label{fig:test2-surf}
\end{figure}

\begin{figure}[h!]
  \centering
  \includegraphics[width=1\linewidth]{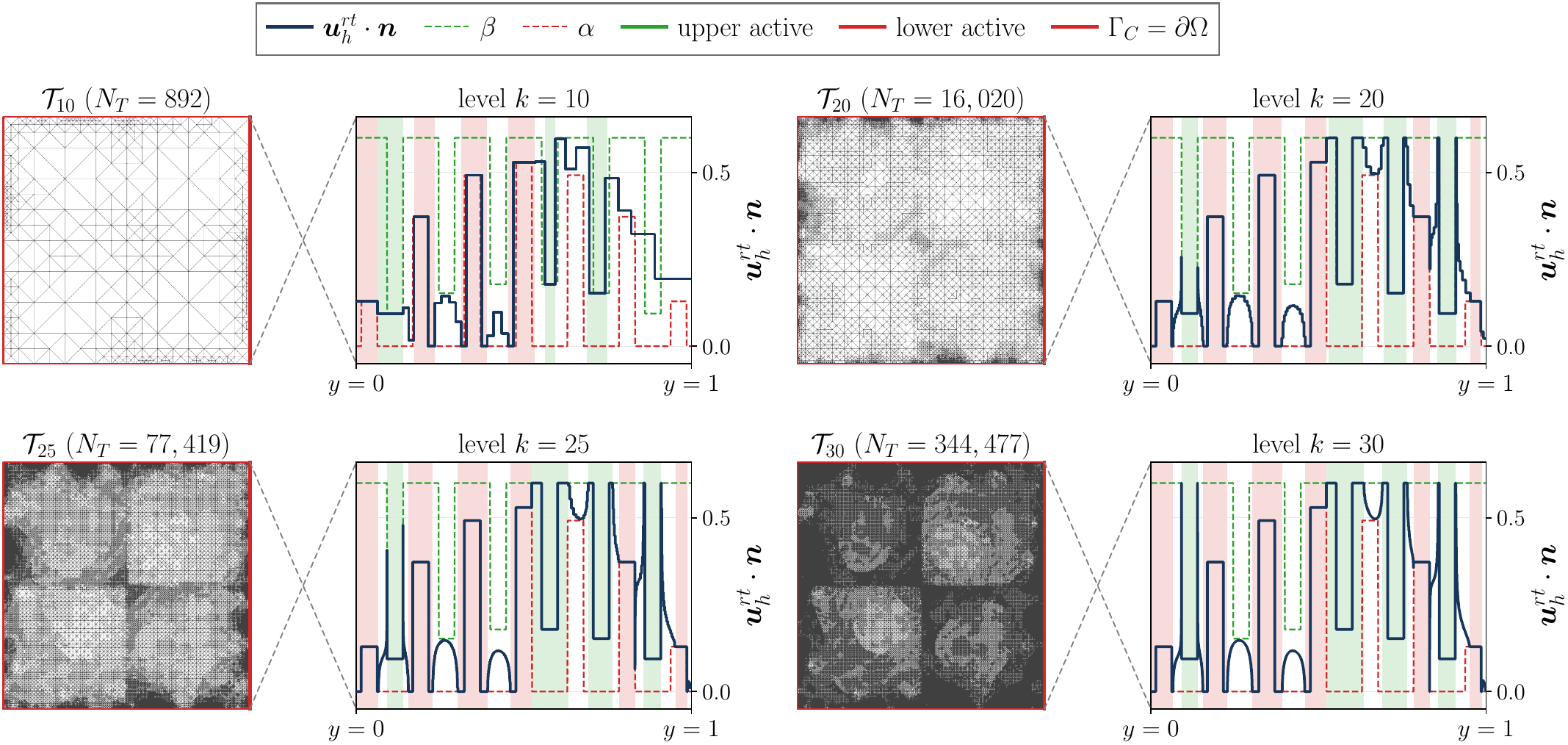}
  \caption{Adaptive meshes and normal flux on the right-face of $\Omega$ across various levels of refinement.}
  \label{fig:test2-refine}
\end{figure}

\begin{figure}[h!]
  \centering
  \includegraphics[width=0.9\linewidth]{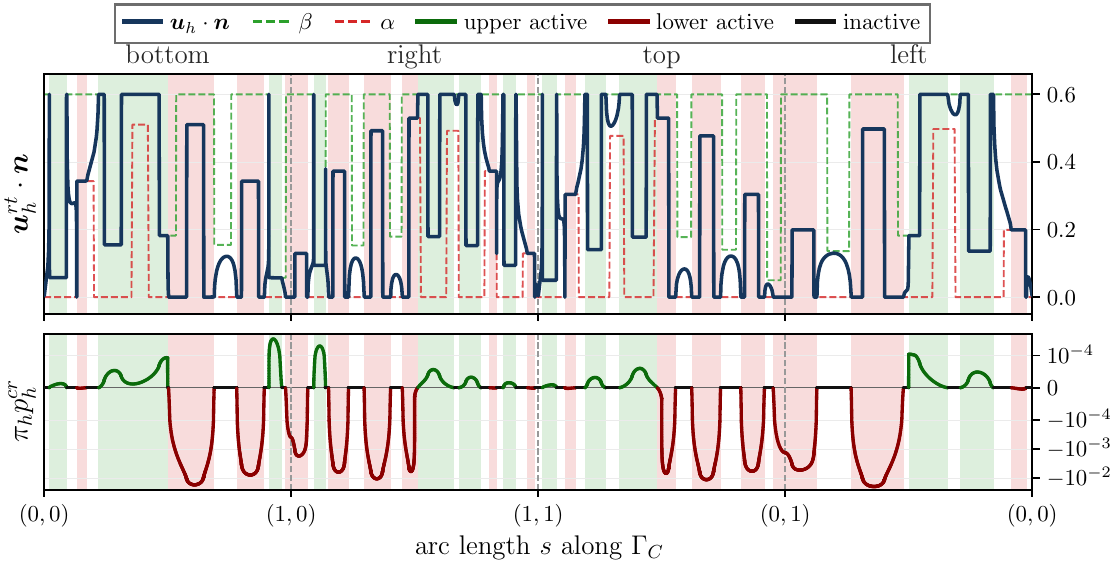}
  \caption{$\bm u_h^{rt}\cdot\bm n$ with bounds (top) and
  $\pi_h p_h^{cr}$ (bottom) on $\Gamma_C$; $N_k=312\,960$.}
  \label{fig:test2-tri}
\end{figure}

\FloatBarrier

\subsection{Miscible displacement in a heterogeneous reservoir}
\label{ss:miscible}

We close with the incompressible miscible-displacement model of porous-media
flow~\cite{LiRiviere2015Miscible}: a solvent of concentration $c$ displaces a
resident fluid through the mixture-dependent mobility
$\bm K(c)=\mu(c)^{-1}\bm k$, with $\bm k$ the intrinsic permeability. The
coupled system is
\begin{subequations}\label{eq:miscible}
\begin{align}
  \operatorname{div}\bm u &= q_I - q_P, \label{eq:miscible_mass}\\
  \bm u &= -\bm K(c)\,(\nabla p - \rho(c)\,\bm g), \label{eq:miscible_darcy}\\
  \phi\,\partial_t c - \operatorname{div}(\bm D(\bm u)\nabla c - \bm u\,c)
    &= q_I\,\bar c - q_P\,c. \label{eq:transport}
\end{align}
\end{subequations}
Here
$\mu(c)=(c\,\mu_s^{-1/4}+(1-c)\,\mu_o^{-1/4})^{-4}$ is the empirical quarter-power
mixing law, $\rho(c)=c\,\rho_s+(1-c)\,\rho_o$ is the density, and
$\bm D(\bm u)=d_m\bm I+\|\bm u\|(\alpha_\ell\bm E+\alpha_t(\bm I-\bm E))$,
$\bm E=\bm u\bm u^\top/\|\bm u\|^2$.  The fully coupled problem is nonconvex, and
the duality framework of the preceding sections does not directly apply.
We therefore discretize in time with implicit Euler and decouple flow
and transport by a sequential time-lagged scheme: with the mobility frozen at
the previous concentration, the flow step at time level $j$ is precisely the
flux-constrained problem~\eqref{eq:cts_max_prob} with permeability
$\bm K(c^{j-1})$, the buoyancy term contributing the additional linear term
$(\rho(c^{j-1})\bm g,\cdot)_\Omega$ to the dual objective. Once $\bm u_h^j$ has been computed, $c_h^j$ is then computed from the the transport equation~\eqref{eq:transport} using an interior-penalty discontinuous Galerkin discretization in space with
upwind convective fluxes and harmonic-mean weighting of the diffusive fluxes,
see, e.g.,~\cite[Section 4.5]{Pietro:book}.

The domain $\Omega \subset \mathbb{R}^3$ is the
$365.76\,\mathrm{m}\times670.56\,\mathrm{m}\times30.48\,\mathrm{m}$ subvolume of Model~2 of the tenth SPE comparative solution
project~\cite{ChristieBlunt01} consisting of the lower $50$ layers of the 
Upper Ness formation, and is discretized by a tetrahedral mesh of $105,456$ elements. The Upper Ness
permeability field is channelized, consisting of high-permeability sand
channels embedded in a low-permeability mudstone background; the permeability
spans roughly seven orders of magnitude (\emph{cf}. \Cref{fig:perm} (a)). A
five-spot well pattern places one injector at the center of the domain and
four producers at its corners; the fluid, transport, and dispersion
parameters (\emph{cf}. \Cref{tab:params}) are taken from~\cite{LiRiviere2015Miscible}. We compare the displacement in three
configurations of this reservoir, an impermeable reference configuration and
two configurations in which a partially sealing fault runs along one lateral
face of the domain (\emph{cf}. \Cref{fig:perm} (b) and (c)).

\begin{table}[ht]
\centering
\caption{Parameters for the miscible-displacement example.}
\label{tab:params}
\small
\begin{tabular}{llll|llll}
\hline
Quantity & Symbol & Value & Units & Quantity & Symbol & Value & Units \\
\hline
Solvent visc.\ & $\mu_s$ & $1.0\times10^{-3}$ & Pa\,s & Molec.\ diff.\ & $d_m$ & $1.8\times10^{-7}$ & m$^2$s$^{-1}$ \\
Resident visc.\ & $\mu_o$ & $9.0\times10^{-4}$ & Pa\,s & Long.\ disp.\ & $\alpha_\ell$ & $1.8\times10^{-5}$ & m \\
Mob.\ ratio & $M$ & $0.9$ & --- & Transv.\ disp.\ & $\alpha_t$ & $1.8\times10^{-6}$ & m \\
Densities & $\rho_s,\rho_o$ & $1000$ & kg\,m$^{-3}$ & Injection & $q_I$ & $1.7$ & m$^3$s$^{-1}$ \\
Gravity & $\|\bm g\|$ & $9.8$ & m\,s$^{-2}$ & Producers & $q_{P_j}$ & $0.34$ & m$^3$s$^{-1}$ \\
Porosity & $\phi$ & $0.2$ & --- & Step / horizon & $\Delta t, T$ & $0.25,\ 15$ & days \\
\hline
\end{tabular}
\end{table}

In the faulted configurations, $\Gamma_C$ consists of the facets of the fault
face where the permeability exceeds the median over the face, and
$\bm u\cdot\bm n=0$ on the rest of the boundary. On $\Gamma_C$ we impose
$0\le\bm u\cdot\bm n\le\beta$, where $\beta$ is the maximal leakage flux per
unit area; the lower bound prevents fluid from re-entering through the fault.
The \emph{longitudinal} fault is parallel to the channels and $\Gamma_C$
comprises $676$ facets with $|\Gamma_C|=10\,219\,\mathrm{m}^2$; the
\emph{transverse} fault is orthogonal to them and $\Gamma_C$ comprises $338$
facets with $|\Gamma_C|=5\,574\,\mathrm{m}^2$
(\emph{cf}.\ Figure~7(b),(c)). Both faults have the same total leakage
capacity $\int_{\Gamma_C}\beta\dif s=c_\beta\,\Delta Q$, $c_\beta=5$, giving
$\beta=1.66\times10^{-4}\,\mathrm{m\,s^{-1}}$ (longitudinal) and
$\beta=3.05\times10^{-4}\,\mathrm{m\,s^{-1}}$ (transverse). Concentration snapshots for the three configurations are shown in
\Cref{fig:plumes-testcase3}, and the concentration histories at the four
producers in \Cref{fig:breakthrough-testcase3}. The semi-smooth Newton
iteration (\Cref{alg:ssn}), warm-started with the active set of the previous
time step, converged in at most $6$ steps per time step in every run.

\begin{figure}[h]
  \centering
  \includegraphics[width=1\linewidth]{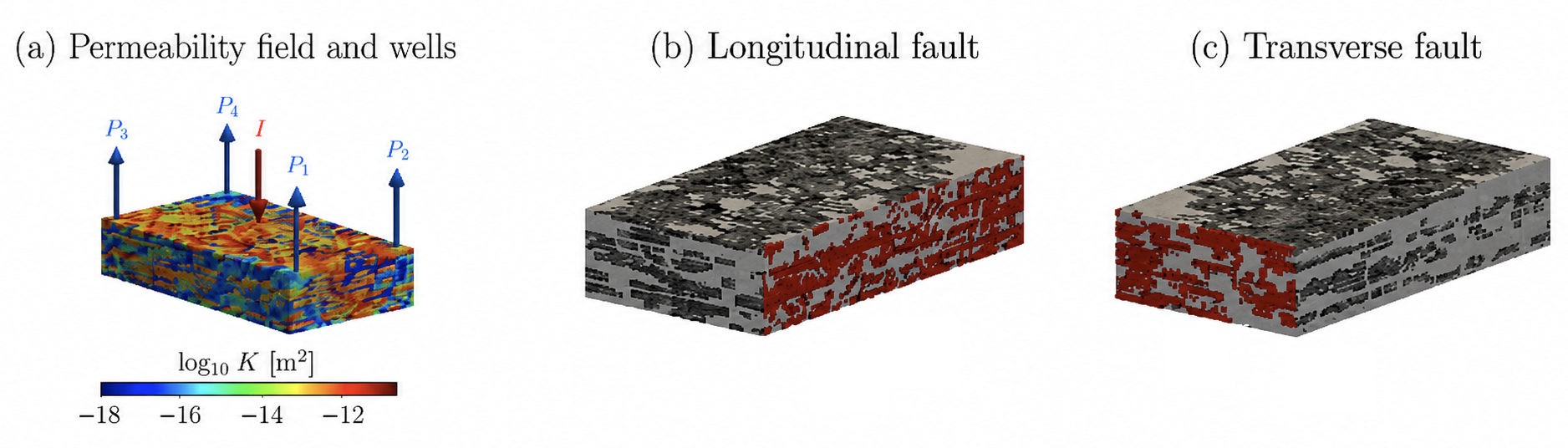}
  \caption{(a) Base-$10$ logarithm of the SPE10 permeability field and the
  five-spot well pattern (central injector, four corner producers). Fault configurations $\Gamma_C$ (pictured in red): (b)~longitudinal, (c)~transverse.}
  \label{fig:perm}
\end{figure}

\begin{figure}[h!]
  \centering
  \includegraphics[width=\linewidth]{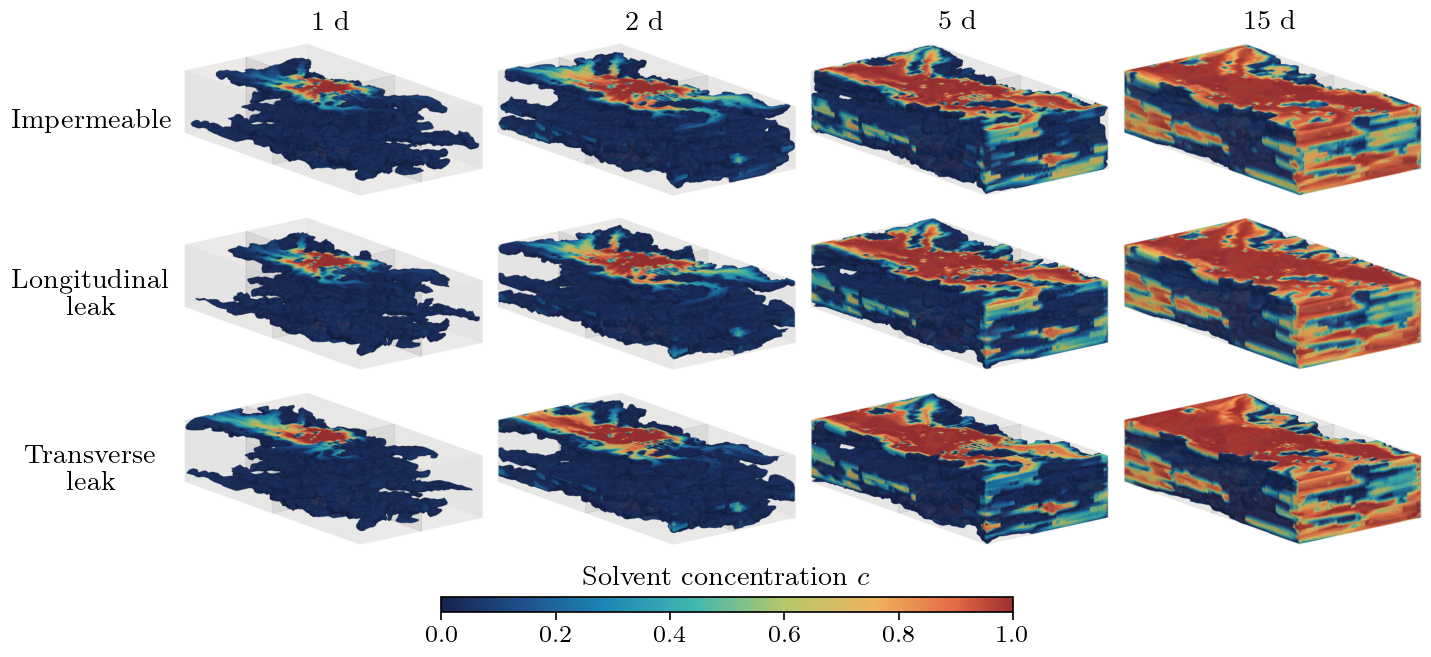}
  \caption{Concentration snapshots at $t=1,2,5,15$~days: fully sealed (top),
  longitudinal fault (middle), transverse fault (bottom).}
  \label{fig:plumes-testcase3}
\end{figure}

\begin{figure}[h!]
  \centering
  \includegraphics[width=0.9\linewidth]{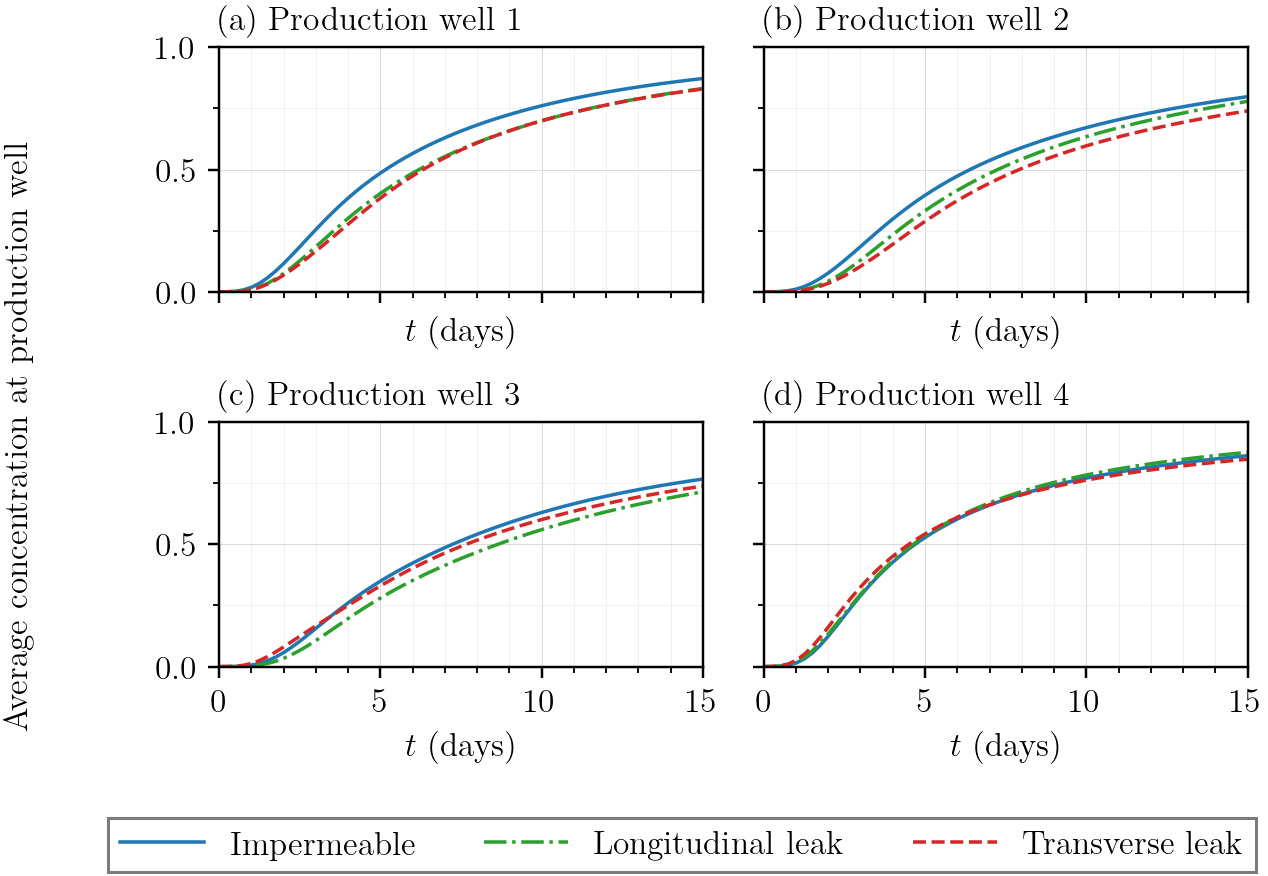}
  \caption{Production concentration histories at the four production wells.
  Each panel shows the average concentration for one well for each of the three cases:
  impermeable (solid), longitudinal (dash-dotted), transverse (dashed).}
  \label{fig:breakthrough-testcase3}
\end{figure}

\section{Conclusion}

We introduced and analyzed a Darcy flow problem with bilateral constraints on the normal flux across a portion of the boundary. On the basis of a Fenchel duality theory at the continuous level, we established well-posedness, strong duality, and convex optimality conditions with a complementarity structure on $\Gamma_C$ (\Cref{thm:cts-dual-exist}, \Cref{lem:strong-duality}), and derived an exact a posteriori error identity for arbitrary admissible approximations (\Cref{thm:aposteriori-id}). On the basis of a discrete Fenchel duality theory for the Raviart--Thomas/Crouzeix--Raviart discretization (\Cref{thm:disc-strong-duality}), we derived a discrete error identity (\Cref{thm:disc-aposteriori-id}) and a priori error decay rates under fractional regularity assumptions on the solution and the flux bounds (\Cref{thm:disc-apriori}). The discrete dual problem is solved by a semismooth Newton method (\Cref{alg:ssn}), the discrete primal solution is recovered by a generalized inverse Marini formula (\Cref{lem:marini}), and the discrete primal-dual gap serves as an optimality-certifying stopping criterion. Numerical experiments confirmed the predicted rates, showed that adaptive refinement driven by the localized gap indicators restores optimal decay, and demonstrated the model in fault-leakage scenarios for the SPE10 benchmark reservoir. Open questions include explicit error decay rates for $0<s\le\tfrac{1}{2}$ and a convergence analysis of the adaptive algorithm (\Cref{alg:afem}).

\bibliographystyle{aomplain}
\bibliography{references}
\end{document}